\documentclass[12pt,reqno]{amsart}

\usepackage{enumerate}
\usepackage{hyperref}
\usepackage{a4}

\DeclareMathOperator{\cimg}{\overline{img}}
\DeclareMathOperator{\spec}{spec}

\DeclareMathOperator*{\LIM}{FP}
\DeclareMathOperator{\sdet}{sdet}
\DeclareMathOperator{\sgn}{sgn}
\DeclareMathOperator{\img}{img}

\DeclareMathOperator{\tr}{tr}

\DeclareMathOperator{\Ad}{Ad}

\newcommand\grr{\textrm{gr}}

\newcommand\Riem{\mathrm{Riem}}

\newcommand\aoe{\mathfrak a}
\newcommand\hoe{\mathfrak h}
\newcommand\goe{\mathfrak g}
\newcommand\poe{\mathfrak p}

\newcommand\ii{\mathbf i}
\DeclareMathOperator{\Lie}{Lie}

\DeclareMathOperator{\id}{id}
\newcommand{\gl}{\mathfrak{gl}}
\DeclareMathOperator{\GL}{GL}
\DeclareMathOperator{\Aut}{Aut}
\DeclareMathOperator{\eend}{end}

\newcommand\Z{\mathbb Z}
\newcommand\N{\mathbb N}
\newcommand\R{\mathbb R}
\newcommand\C{\mathbb C}
\newcommand\Q{\mathbb Q}

\newcommand\trs{{\tr}_*}
\newcommand\dets{{\det}_*}
\newcommand\sdets{{\sdet}_*}

\theoremstyle{plain}
  \newtheorem{theorem}{Theorem}
  \newtheorem{corollary}{Corollary}
  \newtheorem{lemma}{Lemma}
  \newtheorem{proposition}{Proposition}
\theoremstyle{remark}
  \newtheorem{remark}{Remark}

\begin{document}

\title[Regularized determinants of the Rumin complex]{Regularized determinants of the Rumin complex\\ in irreducible unitary representations\\ of the (2,3,5) nilpotent Lie group}

\author{Stefan Haller}

\address{Stefan Haller,
	Department of Mathematics,
	University of Vienna,
	Oskar-Morgenstern-Platz 1,
	1090 Vienna,
	Austria.
	ORCID: \href{https://orcid.org/0000-0002-7064-2215}{0000-0002-7064-2215}}

\email{stefan.haller@univie.ac.at}

\begin{abstract}
	We study the Rumin differentials of the 5-dimensional graded nilpotent Lie group that appears as osculating group of generic rank two distributions in dimension five.
	In irreducible unitary representations of this group, the Rumin differentials provide intriguing generalizations of the quantum harmonic oscillator.
	For the Schr\"odinger representations, we compute the spectrum and the zeta regularized determinant of each Rumin differential.
	In the generic representations, we evaluate their alternating product, i.e., the analytic torsion of the Rumin complex.
\end{abstract}

\keywords{Analytic torsion; Rumin complex; Rockland complex; Generic rank two distribution; (2,3,5) distribution; Sub-Riemannian geometry; Unitary dual;}

\subjclass[2010]{58J52 (primary) and 11M41, 34L40, 53C17, 58A30, 58J10, 81Q20, 81Q80}


\maketitle

\section{Introduction}\label{S:intro}

	Consider the generic real 5-dimensional graded nilpotent Lie algebra 
	\begin{equation}
		\goe=\goe_{-3}\oplus\goe_{-2}\oplus\goe_{-1}
	\end{equation}
	with $\dim\goe_{-1}=2=\dim\goe_{-3}$ and $\dim\goe_{-2}=1$.
	A more explicit description of the Lie algebra structure reads: if $X_1,X_2$ is a basis of $\goe_{-1}$, then $X_3=[X_1,X_2]$ is a basis of $\goe_{-2}$, and $X_4=[X_1,X_3]$, $X_5=[X_2,X_3]$ is a basis of $\goe_{-3}$.

	This Lie algebra appears as the osculating algebra of generic rank two distributions in dimension five \cite{C10}.
	These geometric structures are also known as $(2,3,5)$ distributions, and they have attracted quite some attention recently, see \cite{AN18,BHN18,BM09,BH93,CN09,CSa09,DH19,DH20,H22,HS09,HS11,LNS17,N05,S06,S08} to name a few.
	They can equivalently be described as parabolic geometries \cite{CS09} of type $(G_2,P)$ where $G_2$ denotes the split real form of the exceptional Lie group, and $P$ denotes the parabolic subgroup corresponding to the longer root, cf.~\cite{C10,S08} and \cite[Section~4.3.2]{CS09}. 
	This parabolic subgroup induces a grading on the Lie algebra of $G_2$ which has the form 
	\begin{equation}\label{E:LieG2}
		\Lie(G_2)=\goe_{-3}\oplus\goe_{-2}\oplus\goe_{-1}\oplus\goe_0\oplus\goe_1\oplus\goe_2\oplus\goe_3
	\end{equation}
	where the negative part coincides with the graded nilpotent Lie algebra $\goe$, and $\goe_0\cong\gl_2(\R)$.
	The subspace $\goe_{-1}$ determines a left invariant rank two distribution on the simply connected Lie group $G$ with Lie algebra $\goe$ which is locally diffeomorphic to the flat model $G_2/P$, cf.~\cite{S06}.

	The naturally associated Rumin complex \cite{R90,R94,R99,R05} is a sequence of left invariant differential operators
	\[
		D_q\colon C^\infty(G)\otimes H^q(\goe)\to C^\infty(G)\otimes H^{q+1}(\goe),\qquad q=0,1,2,3,4
	\]
	satisfying $D_{q+1}D_q=0$.
	Here $H^q(\goe)=H^q(\goe;\C)$ denotes the Lie algebra cohomology of $\goe$ with coefficients in the trivial complex representation.
	Explicit formulas for $D_q$ can be found in Section~\ref{SS:D.formula}.
	Rumin has shown that the complex $D_q$ becomes exact in each nontrivial irreducible unitary representation of $G$, see \cite[Theorem~3]{R99} or \cite[Theorem~5.2]{R01}.

	According to Dixmier \cite[Proposition~8]{D58} the nilpotent Lie group $G$ has three types of irreducible unitary representations: 1-dimensional scalar representations which factor through the abelianization $G/[G,G]$; Schr\"odinger representations $\rho_\hbar$ on $L^2(\R,d\theta)$, $0\neq\hbar\in\R$ which factor through the Heisenberg group $H=G/Z(G)$; and generic representations $\rho_{\lambda,\mu,\nu}$ on $L^2(\R,d\theta)$ which can be parametrized by three real numbers $\lambda,\mu,\nu$ with $(\lambda,\mu)\neq(0,0)$.

	In the Schr\"odinger representation $\rho_\hbar$ on $L^2(\R,d\theta)$
	the sub-Laplacian $D_0^*D_0$ acts as the quantum harmonic oscillator 
	\begin{equation}\label{E:harm.osc}
		\rho_\hbar(D_0)^*\rho_\hbar(D_0)=|\hbar|\bigl(-\partial^2_\theta+\theta^2\bigr).
	\end{equation}
	This operator has well know spectrum $|\hbar|(2n+1)$, $n=0,1,2,\dotsc$ and zeta regularized determinant $\sqrt2$. 
	Generalizing this classical case, we will compute the spectrum and the zeta regularized determinant of $\rho_\hbar(D_q)^*\rho_\hbar(D_q)$ for each $q$, see Theorem~\ref{T:dets.Schroedinger} below.

	In the generic representation $\rho_{\lambda,\mu,\nu}$ on $L^2(\R;d\theta)$ the sub-Laplacian $D_0^*D_0$ acts as the oscillator
	\[
		\rho_{\lambda,\mu,\nu}(D_0)^*\rho_{\lambda,\mu,\nu}(D_0)=(\lambda^2+\mu^2)^{1/3}\bigl(-\partial_\theta^2+V(\theta)\bigr)
	\]
	with quartic potential 
	\[
		V(\theta)=\left(\frac{\theta^2+\nu(\lambda^2+\mu^2)^{-2/3}}2\right)^2.
	\]
	The spectrum of this operator is not explicitly known, but for $\nu=0$ its regularized determinant has been evaluated by Voros using exact WKB methods, see \cite[Section~4]{V80}, \cite[Eq.~(10.33)]{V83}, \cite{V23} and the references therein.
	The case $\nu=0$ corresponds to the potential $V(\theta)=\theta^4/4$, separating the regime $\nu<0$ with double well potential $V(\theta)=\bigl(\theta+\sqrt{|\nu|}(\lambda^2+\mu^2)^{-1/3}\bigr)^2\bigl(\theta-\sqrt{|\nu|}(\lambda^2+\mu^2)^{-1/3}\bigr)^2/4$ from the regime $\nu>0$ where $V(\theta)$ is convex and strictly positive.
	In more recent work \cite{V04,V08}, Voros obtained asymptotic results relating determinants for more general potentials to the determinants of harmonic oscillators, perturbatively.
	We are not able to compute the regularized determinants of $\rho_{\lambda,\mu,\nu}(D_q)^*\rho_{\lambda,\mu,\nu}(D_q)$ for each $q$, but we will evaluate their alternating product, i.e., the analytic torsion of the Rumin complex in the representation $\rho_{\lambda,\mu,\nu}$, see Theorem~\ref{T:tor.gen} below.

	A remarkable feature of the graded nilpotent Lie algebra $\goe$ is the fact that it has pure cohomology \cite{R99,R01}, i.e., the grading automorphism acts as a scalar in each cohomology $H^q(\goe)$.
	Equivalently, the Rumin differentials $D_q$ on the nilpotent Lie group $G$ are all homogeneous with respect to the grading automorphism.
	This implies that the Rumin complex of a filtered manifold with osculating algebras isomorphic to $\goe$ is a Rockland \cite{R78} complex \cite[Section~2.3]{DH22} and permits to define an analytic torsion for filtered manifolds of this type, cf.~\cite[Section~3]{H22}. 
	
	Having pure cohomology appears to be a rather restrictive property, cf.~\cite[Section~3.7]{H22}.
	We only know three types of graded nilpotent Lie algebras with pure cohomology:
	trivially graded abelian Lie algebras corresponding to the Ray--Singer analytic torsion \cite{RS71,Ch77,Ch79,M78,BZ92};
	Heisenberg algebras corresponding to the Rumin--Seshadri analytic torsion for contact manifolds \cite{RS12,AQ22,K20,K22};
	and the Lie algebra $\goe$ corresponding to generic rank two distributions in dimension five, for which an analytic torsion has recently been proposed in \cite{H22}.
	Abelian Lie algebras are discussed in Appendix~\ref{S:abelian}.
	Analogous results for the 3-dimensional Heisenberg algebra can be found in Appendix~\ref{S:Heisenberg}.
	
	Our study of the regularized determinants of $\rho(D_q)^*\rho(D_q)$ was motivated by an attempt to compute the analytic torsion of nilmanifolds with $(2,3,5)$ geometry using harmonic analysis, that is, by decomposing their Rumin complex into irreducible $G$-representations, cf.~\cite{H71,R71}.
	Building on the results contained in the paper at hand, we show in \cite{H23b} that for (2,3,5) distributions on nilmanifolds the analytic torsion of the Rumin complex coincides with the Ray--Singer torsion. 

\section{Statement of the main results}\label{S:results}

	Let $\goe=\goe_{-3}\oplus\goe_{-2}\oplus\goe_{-1}$ denote the graded nilpotent Lie algebra mentioned in the introduction.
	Let $X_1,X_2$ be a basis of $\goe_{-1}$ and put
	\begin{equation}\label{E:basis.Xi}
		X_3=[X_1,X_2],\qquad X_4=[X_1,X_3],\qquad X_5=[X_2,X_3].
	\end{equation}
	Then $X_1,X_2,X_3,X_4,X_5$ is a graded basis of $\goe$, i.e., $X_3$ is a basis of $\goe_{-2}$ and $X_4,X_5$ is a basis of $\goe_{-3}$.
	The simply connected nilpotent Lie group with Lie algebra $\goe$ will be denoted by $G$.
	Translation of $\goe_{-1}$ provides a left invariant rank two subbundle in the tangent bundle of $G$.
	The associated Rumin complex \cite{R99,R05} is a sequence of left invariant differential operators on $G$,
	\begin{multline}\label{E:Rumin.G}
		C^\infty(G)\otimes H^0(\goe)
		\xrightarrow{D_0}C^\infty(G)\otimes H^1(\goe)
		\xrightarrow{D_1}C^\infty(G)\otimes H^2(\goe)
		\\
		\xrightarrow{D_2}C^\infty(G)\otimes H^3(\goe)
		\xrightarrow{D_3}C^\infty(G)\otimes H^4(\goe)
		\xrightarrow{D_4}C^\infty(G)\otimes H^5(\goe),
	\end{multline}
	satisfying 
	\begin{equation}\label{E:DD=0}
		D_{q+1}D_q=0.
	\end{equation}
	Here $H^q(\goe)=H^q(\goe;\C)$ denotes the cohomology of $\goe$ with coefficients in the trivial complex representation.
	The Betti numbers are $\dim H^q(\goe)=1,2,3,3,2,1$ for $q=0,\dotsc,5$.
	There exist differential operators $L_q\colon C^\infty(G)\otimes H^q(\goe)\to\Omega^q(G)$ that embed the Rumin complex as a subcomplex in the de~Rham complex and induce an isomorphism on cohomology.
	Explicit formulas for the Rumin differentials will be provided in Section~\ref{SS:D.formula}.
	Further details may be found in \cite[\S5]{BENG11}, \cite{DH22}, \cite[Section~3.1]{H22}, and \cite{FT23}.

	The Lie algebra $\goe$ has pure cohomology \cite{R99,R01}, that is, the grading automorphism $\phi_t$, given by multiplication with $t^j$ on $\goe_j$, acts by a scalar on $H^q(\goe)$.
	In fact, $H^q(\phi_t)=t^{N_q}$ where $N_q=0,1,4,6,9,10$ for $q=0,\dotsc,5$.
	This implies that $D_q$ is homogeneous of order $k_q=N_{q+1}-N_q$ with respect to the automorphism $\Phi_t$ of $G$ integrating $\phi_t$.
	More explicitly, $\Phi_t^*D_q=t^{k_q}D_q$ for all real numbers $t\neq0$.
	Hence, we may regard
	\[
		D_q\in\mathcal U^{-k_q}(\goe)\otimes L\bigl(H^q(\goe),H^{q+1}(\goe)\bigr)
	\]
	where the grading of the universal enveloping algebra $\mathcal U(\goe)$ is the one induced from the grading of $\goe=\goe_{-3}\oplus\goe_{-2}\oplus\goe_{-1}$.
	We find $k_q=1,3,2,3,1$ for $q=0,\dotsc,4$.

	Rumin has shown that the operators in \eqref{E:Rumin.G} form a Rockland \cite{R78} complex, i.e., the sequence becomes exact on smooth vectors in every nontrivial irreducible unitary representation of $G$, see \cite[Theorem~1]{R99}, \cite[theorem~5.2]{R01}, or \cite[Corollary~4.18(b)]{DH22}.
	More explicitly, if $\rho\colon G\to U(\mathcal H)$ is a (strongly continuous) irreducible unitary representation on a Hilbert space $\mathcal H$, then
	\begin{multline}\label{E:Rumin.rho}
		\mathcal H_\infty\otimes H^0(\goe)
		\xrightarrow{\rho(D_0)}\mathcal H_\infty\otimes H^1(\goe)
		\xrightarrow{\rho(D_1)}\mathcal H_\infty\otimes H^2(\goe)
		\\
		\xrightarrow{\rho(D_2)}\mathcal H_\infty\otimes H^3(\goe)
		\xrightarrow{\rho(D_3)}\mathcal H_\infty\otimes H^3(\goe)
		\xrightarrow{\rho(D_4)}\mathcal H_\infty\otimes H^5(\goe)
	\end{multline}
	is exact where $\mathcal H_\infty$ denotes the (dense) subspace of smooth vectors in $\mathcal H$, cf.~\cite[Section~2.3]{DH22}.
	Background on the relevant representation theory may be found in \cite[Appendix~V]{K04}.
	For graded nilpotent Lie algebras whose cohomology is not pure, the Rumin complex is Rockland only in a graded sense, and the analysis becomes more involved, see \cite[Section~5]{DH22}.

	All irreducible unitary representations of nilpotent Lie groups can be obtained by the orbit method due to Kirillov \cite{K04}.
	For the 5-dimensional Lie group $G$ these representations have been worked out explicitly by Dixmier in \cite[Proposition~8]{D58}.
	There are three types of irreducible unitary representations of $G$ which we now describe in terms of a graded basis $X_1,\dotsc,X_5$ of $\goe$ satisfying \eqref{E:basis.Xi}.
	In our notation we will not distinguish between the representation of $G$ and the corresponding infinitesimal representation of $\goe$.
	The dual grading will be denoted by $\goe_j^*:=(\goe_{-j})^*$.

	\begin{enumerate}[(I)]
		\item	\emph{Scalar representations:}
			For each $\alpha\in\goe_1^*$ we have a unitary representation $\rho_\alpha$ on $\C$ with
			\begin{equation}\label{E:rep.scalar}
				\rho_\alpha(X_1)=\mathbf i\alpha(X_1),\qquad
				\rho_\alpha(X_2)=\mathbf i\alpha(X_2),
			\end{equation}
			and $\rho_\alpha(X_3)=\rho_\alpha(X_4)=\rho_\alpha(X_5)=0$.
			These are precisely the representations which factor through the abelianization $G/[G,G]\cong\R^2$.
			They correspond to coadjoint orbits contained in $\goe_1^*\subseteq\goe^*$.
			These coadjoint orbits are all 0-dimensional.
		\item	\emph{Schr\"odinger representations:}
			For $0\neq\hbar\in\R$ we have an irreducible unitary representation $\rho_\hbar$ on $L^2(\R)=L^2(\R,d\theta)$ with
			\begin{equation}\label{E:Schroedinger}
				\rho_\hbar(X_1)=\sqrt{|\hbar|}\partial_\theta,\quad
				\rho_\hbar(X_2)=\ii\sgn(\hbar)\sqrt{|\hbar|}\theta,\quad
				\rho_\hbar(X_3)=\ii\hbar,
			\end{equation}
			and $\rho_\hbar(X_4)=\rho_\hbar(X_5)=0$.
			These are precisely the irreducible unitary representations which factor through the 3-dimensional Heisenberg group $H=G/Z(G)$ but do not factor through $G/[G,G]$.
			They correspond to coadjoint orbits contained in $(\goe_1^*\oplus\goe_2^*)\setminus\goe_1^*$.
			These coadjoint orbits are all 2-dimensional.
		\item	\emph{Generic representations:}
			For real numbers $\lambda,\mu,\nu$ with $(\lambda,\mu)\neq(0,0)$ we have an irreducible unitary representation $\rho_{\lambda,\mu,\nu}$ on $L^2(\R)=L^2(\R,d\theta)$ with:
			\begin{align}
				\label{E:rep.gen.X1}
				\rho_{\lambda,\mu,\nu}(X_1)
				&=\frac\lambda{(\lambda^2+\mu^2)^{1/3}}\partial_\theta-\frac{\mathbf i\mu}{(\lambda^2+\mu^2)^{1/3}}\cdot\frac{\theta^2+\nu(\lambda^2+\mu^2)^{-2/3}}2,
				\\\label{E:rep.gen.X2}
				\rho_{\lambda,\mu,\nu}(X_2)
				&=\frac\mu{(\lambda^2+\mu^2)^{1/3}}\partial_\theta+\frac{\mathbf i\lambda}{(\lambda^2+\mu^2)^{1/3}}\cdot\frac{\theta^2+\nu(\lambda^2+\mu^2)^{-2/3}}2,
				\\\label{E:rep.gen.X3}
			        \rho_{\lambda,\mu,\nu}(X_3)
				&=\mathbf i(\lambda^2+\mu^2)^{1/3}\theta,
				\\\label{E:rep.gen.X4}
			        \rho_{\lambda,\mu,\nu}(X_4)
				&=\mathbf i\lambda,
				\\\label{E:rep.gen.X5}
			        \rho_{\lambda,\mu,\nu}(X_5)
				&=\mathbf i\mu.
			\end{align}
			This differs from the representation given in \cite[Equation~(24)]{D58} by a conjugation with a unitary scaling on $L^2(\mathbb R)$, namely $f(\theta)\leftrightarrow\sqrt rf(r\theta)$, where $r=(\lambda^2+\mu^2)^{2/3}$, which we have introduced for better compatibility with the grading automorphism.
			Note that
			\begin{equation}\label{E:rep.gen.nu}
				\rho_{\lambda,\mu,\nu}\bigl(X_3X_3+2X_1X_5-2X_2X_4\bigr)=\nu.
			\end{equation}
			Generic representations correspond to coadjoint orbits in $\goe^*\setminus(\goe_1^*\oplus\goe_2^*)$.
			These coadjoint orbits are all 2-dimensional.
	\end{enumerate}
	The representations listed above are mutually nonequivalent, and they comprise all equivalence classes of irreducible unitary representations of $G$. 

	Let $h_q$ be a Hermitian inner product on $H^q(\goe)$.
	We will denote the corresponding graded Hermitian inner product on $H^*(\goe)$ by $h=\bigoplus_qh_q$.
	If $\rho$ is an irreducible unitary representation on a Hilbert space $\mathcal H$, we let $\rho(D_q)^{*_h}$ denote the formal adjoint $\rho(D_q)$ with respect to $h$ and the inner product on $\mathcal H$.
	The operator 
	\begin{equation}\label{E:|D|}
		|\rho(D_q)|_h:=\sqrt{\rho(D_q)^{*_h}\rho(D_q)}
	\end{equation}
	has an infinite dimensional kernel in general, but the remaining part of the spectrum consists of isolated positive eigenvalues with finite multiplicity only.
	We will denote its zeta function by
	\begin{equation}\label{E:zeta.|rho.D|}
		\zeta_{|\rho(D_q)|_h}(s)
		:=\trs\bigl(\rho(D_q)^{*_h}\rho(D_q)\bigr)^{-s/2}.
	\end{equation}
	Here the notation $\trs$ indicates that the eigenvalue zero does not contribute.

	In Section~\ref{S:heat.asymp} we will work out the asymptotic expansion of the heat trace for positive Rockland operators in irreducible unitary representations and deduce:

	\begin{theorem}\label{T:zeta}
		Let $D_q\in\mathcal U(\goe)\otimes L(H^q(\goe),H^{q+1}(\goe))$ denote the Rumin differentials associated with the 5-dimensional graded Lie algebra $\goe$.
		Moreover, let $h_q$ be any Hermitian inner product on $H^q(\goe)$.
		Then the following hold true:
		\begin{enumerate}[(I)]
			\item	For every nontrivial irreducible unitary representation $\rho$ of $G$, the right hand side of \eqref{E:zeta.|rho.D|} converges for $\Re s$ sufficiently large.
				Moreover, this zeta function extends to a meromorphic function on the entire complex plane which has only simple poles and is holomorphic at $s=0$.
			\item	In the Schr\"odinger representation $\rho=\rho_\hbar$, the right hand side of \eqref{E:zeta.|rho.D|} converges for $\Re s>2/k_q$, and the poles of this zeta function can only be located at $s=(1-j)2/k_q$ with $1\neq j\in\N_0$.
			\item	In the generic representation $\rho=\rho_{\lambda,\mu,\nu}$, the right hand side in \eqref{E:zeta.|rho.D|} converges for $\Re s>3/2k_q$, this zeta function vanishes at $s=0$, and its poles can only be located at $s=(3-2j)/2k_q$, $j\in\N_0$.
		\end{enumerate}
	\end{theorem}

	In view of Theorem~\ref{T:zeta}, we may define zeta regularized determinants
	\begin{equation}\label{E:det.def}
		\log\dets|\rho(D_q)|_h
		:=-\zeta'_{|\rho(D_q)|_h}(0).
	\end{equation}
	Here the notation $\dets$ indicates that we are considering the regularized product of all nonzero eigenvalues.
	The analytic torsion $\tau_h(\rho(D))$ is the graded determinant of the Rumin complex, i.e.,
	\[
		\tau_h(\rho(D))
		:=\sdets|\rho(D)|_h
		:=\frac{\dets|\rho(D_0)|_h\cdot\dets|\rho(D_2)|_h\cdot\dets|\rho(D_4)|_h}{\dets|\rho(D_1)|_h\cdot\dets|\rho(D_3)|_h}.
	\]
	Equivalently,
	\begin{equation}\label{E:tau.D.def}
		\log\tau_h(\rho(D))
		=\log\sdets|\rho(D)|_h
		=\sum_{q=0}^4(-1)^q\log\dets|\rho(D_q)|_h.
	\end{equation}

	To formulate our first result, we consider a Hermitian inner product $h_g$ on $H^*(\goe)$ that is induced from a graded Euclidean inner product $g=g_{-3}\oplus g_{-2}\oplus g_{-1}$ on $\goe=\goe_{-3}\oplus\goe_{-2}\oplus\goe_{-1}$.
	To such a graded Euclidean inner product we associate a positive number $a_g$ and a Euclidean inner product $b_g$ on $\goe_{-1}$ such that 
	\begin{align}\label{E:ag}
		g_{-2}([X,Y],[X,Y])
		&=a_g\cdot\bigl(g_{-1}(X,X)g_{-1}(Y,Y)-g_{-1}(X,Y)^2\bigr)
		\\\label{E:bg}
		g_{-3}\bigl([X,Z],[Y,Z]\bigr)
		&=g_{-2}(Z,Z)\cdot b_g(X,Y)
	\end{align}
	hold for all $X,Y\in\goe_{-1}$ and $Z\in\goe_{-2}$.

	\begin{theorem}\label{T:dets.scalar}
		Let $D_q\in\mathcal U(\goe)\otimes L(H^q(\goe),H^{q+1}(\goe))$ denote the Rumin differentials associated with the 5-dimensional graded Lie algebra $\goe=\goe_{-3}\oplus\goe_{-2}\oplus\goe_{-1}$.
		Moreover, let $h_g$ denote the Hermitian inner product on $H^*(\goe)$ induced from a graded Euclidean inner product $g$ on $\goe$.
		Then, in the nontrivial scalar representation $\rho_\alpha$ on $\C$ corresponding to $0\neq\alpha\in\goe_1^*$, we have
		\begin{align*}
			\dets|\rho_\alpha(D_4)|_{h_g}
			=\dets|\rho_\alpha(D_0)|_{h_g}
			&=\|\alpha\|_g,
			\\
			\dets|\rho_\alpha(D_3)|_{h_g}
			=\dets|\rho_\alpha(D_1)|_{h_g}
			&=\frac{\|\alpha\|_g^2\|\alpha\|_{b_g}}{\sqrt{a_g}},
			\\
			\dets|\rho_\alpha(D_2)|_{h_g}
			&=\frac{\|\alpha\|_g^2\|\alpha\|_{b_g}^2}{a_g\cdot\tr(b_g^{-1}g_{-1})},
		\end{align*}
		where $a_g$ and $b_g$ are defined in \eqref{E:ag} and \eqref{E:bg}.
		For the torsion we find
		\begin{equation}\label{E:tor.scalar}
			\tau_{h_g}(\rho_\alpha(D))=\frac1{\tr(b_g^{-1}g_{-1})}.
		\end{equation}
	\end{theorem}

	As the Rumin complex associated with a scalar representation is finite dimensional, the latter theorem is an elementary statement that will be proved in Section~\ref{S:scalar}.
	We have included it here mainly because it shows that for scalar representations the torsion does depend on the graded Euclidean inner product on $\goe$.
	For all other (infinite dimensional) irreducible unitary representations of $G$, the torsion turns out to be independent of the graded Euclidean inner product $g$ on $\goe$.
	We do not know if the analytic torsion of a generic rank two distribution on a closed 5-manifold is independent of the metric choices, but the dependence has been shown to be local, see \cite[Theorem~1.1]{H22}.

	Given a Euclidean inner product $g_{-1}$ on $\goe_{-1}$, the equations \eqref{E:ag} and \eqref{E:bg} can be used to extend it to a graded Euclidean inner product $g$ on $\goe$ by requiring $a_g=1$ and $b_g=g_{-1}$.
	According to \eqref{E:tor.scalar}, this seemingly natural choice of an extension yields a nontrivial torsion, $\tau_{h_g}(\rho_\alpha(D))=1/2$.

	If $B$ denotes the Killing form of $\Lie(G_2)$, see \eqref{E:LieG2}, and $\theta$ is a Cartan involution such that $\theta(\goe_j)=\goe_{-j}$, cf.~\cite[Section~3.3.1]{CS09}, then $B_\theta(X,Y)=-B(X,\theta Y)$ is a graded Euclidean inner product on $\Lie(G_2)$.
	For its restriction $g=B_\theta|_\goe$ we find $a_g=4$ and $b_g=3g_{-1}$, see \cite[Equations~(3.21) and (3.39)]{S08} or \cite[Eqs.~(147) and (148)]{H22}.
	In view of \eqref{E:tor.scalar}, this $g$ yields a nontrivial torsion too, $\tau_{h_g}(\rho_\alpha(D))=3/2$.

	The next result gives the regularized determinants of the Rumin differentials in the Schr\"odinger representations with respect to slightly restricted graded Euclidean inner products on $\goe$.

	\begin{theorem}\label{T:dets.Schroedinger}
		Let $D_q\in\mathcal U(\goe)\otimes L(H^q(\goe),H^{q+1}(\goe))$ denote the Rumin differentials associated with the 5-dimensional graded Lie algebra $\goe=\goe_{-3}\oplus\goe_{-2}\oplus\goe_{-1}$.
		Moreover, let $h_g$ denote the Hermitian inner product on $H^*(\goe)$ induced from a graded Euclidean inner product $g$ on $\goe$ such that $b_g$ is proportional to $g$, cf.~\eqref{E:bg}.
		Then, in the Schr\"odinger representation $\rho_\hbar$ on $L^2(\R)$, $\hbar\neq0$, we have
		\begin{align}
			\label{E:det.D0.Schroedinger}
			\dets|\rho_\hbar(D_4)|_{h_g}
			=\dets|\rho_\hbar(D_0)|_{h_g}
			&=2^{1/4},
			\\
			\label{E:det.D1.Schroedinger}
			\dets|\rho_\hbar(D_3)|_{h_g}
			=\dets|\rho_\hbar(D_1)|_{h_g}
			&=2^{3/4}\cdot\sin^{1/2}\left(\pi\cdot\frac{\sqrt2-1}2\right),
			\\
			\label{E:det.D2.Schroedinger}
			\dets|\rho_\hbar(D_2)|_{h_g}
			&=2\cdot\sin\left(\pi\cdot\frac{\sqrt2-1}2\right).
		\end{align}
		In particular, the torsion is trivial,
		\begin{equation}\label{E:tor.Schroedinger}
			\tau_{h_g}(\rho_\hbar(D))=1.
		\end{equation}
	\end{theorem}

	The statement expressed in \eqref{E:det.D0.Schroedinger} is a well know classical result that has only been included for the sake of completeness.
	Indeed, the operator $\rho_\hbar(D_0)^*\rho_\hbar(D_0)$ is the quantum harmonic oscillator with spectrum $|\hbar|(2n+1)$, $n\in\N_0$.
	Therefore, $\zeta_{|\rho_\hbar(D_0)|_{h_g}}(s)=|\hbar|^{-s/2}\cdot(1-2^{-s/2})\cdot\zeta_\Riem(s/2)$, where $\zeta_\Riem(s)=\sum_{n=1}^\infty n^{-s}$ denotes the Riemann zeta function.
	Classically, we have $\zeta_\Riem(0)=-\frac12$ and thus $\zeta_{|\rho_\hbar(D_0)|_{h_g}}'(0)=-\frac14\log2$, that is \eqref{E:det.D0.Schroedinger}.

	In Section~\ref{S:Schroedinger} we will compute the spectra of $|\rho_\hbar(D_q)|_{h_g}$ and evaluate $\zeta_{|\rho_\hbar(D_q)|_{h_g}}(s)$ and its derivative at $s=0$ to obtain the new determinants in Theorem~\ref{T:dets.Schroedinger} corresponding to $q=1,2,3$.
	Clearly, the equalities in \eqref{E:det.D0.Schroedinger} remain true for general graded Euclidean inner products $g$ on $\goe$, but we do not know if this is also the case for \eqref{E:det.D1.Schroedinger} and \eqref{E:det.D2.Schroedinger}.

	Equation \eqref{E:tor.Schroedinger}, however, remains true for arbitrary graded Euclidean inner products $g$ on $\goe$ and all infinite dimensional irreducible unitary representations of $G$.
	More precisely, we have

	\begin{theorem}\label{T:tor.gen}
		Let $D_q\in\mathcal U(\goe)\otimes L(H^q(\goe),H^{q+1}(\goe))$ denote the Rumin differentials associated with the 5-dimensional graded Lie algebra $\goe=\goe_{-3}\oplus\goe_{-2}\oplus\goe_{-1}$.
		Then, for every generic representation $\rho_{\lambda,\mu,\nu}$ of $G$, and any graded Hermitian inner product $h$ on $H^*(\goe)$, we have
		\[
			\tau_h(\rho_{\lambda,\mu,\nu}(D))=1.
		\]
		Moreover, for every Schr\"odinger representation $\rho_\hbar$ of $G$, and any Hermitian inner product $h_g$ induced from a graded Euclidean inner product $g$ on $\goe$, we have 
		\[
			\tau_{h_g}(\rho_\hbar(D))=1.
		\]
	\end{theorem}

	Let us close this section with a few remarks regarding the proofs of Theorems~\ref{T:zeta} and \ref{T:tor.gen}.
	Following \cite{RS71,RS12,H22}, we will express the zeta function of $|\rho(D_q)|_h$ in terms of a Hodge type Laplacian which is analytically better behaved.
	To this end, we fix natural numbers $a_q$ such that $\kappa=a_qk_q$ is independent of $q$.
	The smallest possible choice is $a_q=6,2,3,2,6$ for $q=0,1,2,3,4$, which gives $\kappa=6$.
	Then, $(D_q^{*_h}D_q)^{a_q}$ and $(D_{q-1}D_{q-1}^{*_h})^{a_{q-1}}$ both are homogeneous of degree $2\kappa$, and so is the Rumin--Seshadri \cite{RS12} operator $\Delta_{h,q}\in\mathcal U^{-2\kappa}(\goe)\otimes\eend(H^q(\goe))$,
	\begin{equation}\label{E:Rumin.Seshadri}
		\Delta_{h,q}
		:=(D_{q-1}D_{q-1}^{*_h})^{a_{q-1}}+(D_q^{*_h}D_q)^{a_q},
	\end{equation}
	cf.~\cite[Section~2.2]{H22}.
	As the sequence in \eqref{E:Rumin.rho} is exact, $\Delta_{h,q}$ is a Rockland operator, i.e.,
	\[
		\rho(\Delta_{h,q})
		=\bigl(\rho(D_{q-1})\rho(D_{q-1})^{*_h}\bigr)^{a_{q-1}}+\bigl(\rho(D_q)^{*_h}\rho(D_q)\bigr)^{a_q}
	\]
	is injective on smooth vectors, for every nontrivial irreducible unitary representation $\rho$ of $G$, cf.~\cite[Lemma~2.14]{DH22}.
	In Section~\ref{S:heat.asymp} we will show that $\rho(\Delta_{h,q})^{-s}$ is trace class for $\Re s$ sufficiently large, depending on $q$ and the representation $\rho$, and that the zeta function
	\[
		\zeta_{\rho(\Delta_{h,q})}(s)
		=\tr\rho(\Delta_{h,q})^{-s}
	\]
	can be analytically continued to a meromorphic function on the entire complex plane which has only simple poles and is holomorphic at $s=0$.
	In particular, we may define the regularized determinants by
	\begin{equation}\label{E:def.det.Delta}
		\log\det(\rho(\Delta_{h,q}))
		:=-\zeta'_{\rho(\Delta_{h,q})}(0).
	\end{equation}

	From the properties of the Rumin--Seshadri operators collected in the preceding paragraph one can extract analogous facts for $|\rho(D_q)|_h$, cf.~\cite[Remark~2.9]{H22}.
	Indeed, using \eqref{E:zeta.|rho.D|}, \eqref{E:DD=0} and the fact that $\rho(D_{q-1})\rho(D_{q-1})^{*_h}$ and $\rho(D_{q-1})^{*_h}\rho(D_{q-1})$ have the same nontrivial spectrum, we obtain
	\begin{equation}\label{E:zeta.Delta}
		\zeta_{\rho(\Delta_{h,q})}(s/2)=\zeta_{|\rho(D_{q-1})|_h}(a_{q-1}s)+\zeta_{|\rho(D_q)|_h}(a_qs).
	\end{equation}
	As the pole structure of the analytic continuation $\zeta_{\rho(\Delta_{h,q})}(s)$ is understood for each $q$, the analytic extension and pole structure of $\zeta_{|\rho(D_q)|_h}(s)$ can be deduced recursively using \eqref{E:zeta.Delta}, whence Theorem~\ref{T:zeta}.
	From \eqref{E:zeta.Delta} we also obtain
	\[
		\tfrac12\log\det(\rho(\Delta_{h,q}))
		=a_{q-1}\log\dets|\rho(D_{q-1})|_h+a_q\log\dets|\rho(D_q)|_h.
	\]
	Combining this with $a_qk_q=\kappa$, and $k_{q-1}=N_q-N_{q-1}$, and \eqref{E:tau.D.def}, we see that the analytic torsion can be expressed in the form
	\begin{equation}\label{E:tor.Delta}
		\log\tau_h(\rho(D))
		=\frac1{2\kappa}\sum_{q=0}^5(-1)^{q+1}N_q\log\det(\rho(\Delta_{h,q})),
	\end{equation}
	cf.~\cite[Definition~1.6]{RS71}, \cite[p.~728]{RS12}, and \cite[Section~2.3]{H22}.
	
	To prove Theorem~\ref{T:tor.gen}, we will, for $\nu<0$, establish a polyhomogeneous expansion of the heat trace $\tr\bigl(e^{-t\rho_{r\lambda,r\mu,\nu}(\Delta_{h,q})}\bigr)$ in the variables $t$ and $r$, see Theorem~\ref{T:poly} in Section~\ref{S:poly.heat}.
	This yields, for $\nu<0$, an asymptotic expansion of the form 
	\begin{equation}\label{E:asymp.tau}
		\zeta_{\rho_{r\lambda,r\mu,\nu}(\Delta_{h,q})}'(0)
		\sim\sum_{k=0}^\infty z_kr^{2k-2}+\sum_{k=1}^\infty\tilde z_kr^{4k-2}\log r
	\end{equation}
	as $r\to0$, with constant term
	\[
		\LIM_{r\to0}\zeta_{\rho_{r\lambda,r\mu,\nu}(\Delta_{h,q})}'(0)
		=z_1
		=\zeta_{\rho_{\hbar}(\Delta_{h,q})}'(0)+\zeta_{\rho_{-\hbar}(\Delta_{h,q})}'(0),
	\]
	where $\rho_{\pm\hbar}$ denotes the Schr\"odinger representations with $\hbar=\sqrt{|\nu|}$,
	see Theorem~\ref{T:Melin} in Section~\ref{S:zeta.asymp}.
	For the analytic torsion this implies
	\begin{equation}\label{E:LIM.tau}
		\LIM_{r\to0}\log\tau_h\bigl(\rho_{r\lambda,r\mu,\nu}(D)\bigr)
		=\log\bigl(\tau_h(\rho_\hbar(D)\bigr)+\log\tau_h\bigl(\rho_{-\hbar}(D)\bigr).
	\end{equation}
	In Section~\ref{S:proof} we will show that $\log\tau_h\bigl(\rho_{\lambda,\mu,\nu}(D)\bigr)$ is independent of $\lambda,\mu,\nu$ and $h$.
	If $h=h_g$ is induced from a graded Euclidean inner product on $\goe$ such that $b_g$ is proportional to $g$, then the right hand side in \eqref{E:LIM.tau} vanishes according to Theorem~\ref{T:dets.Schroedinger}.
	Combining these facts with a symmetry argument, we will complete the proof of Theorem~\ref{T:tor.gen}.
	
	For $q=0$ the leading (singular and constant) terms of an asymptotic formula similar to the one in \eqref{E:asymp.tau} have been determined by Voros \cite[Eq.~(5.7)]{V04} using exact WKB methods.
	However, the perturbation studied by Voros does not appear to coincide with the one used in this paper, see Remark~\ref{R:Voros} in Section~\ref{S:proof} for more details.

\section{The Rumin complex of the (2,3,5) nilpotent Lie group}\label{S:Rumin}

	In this section we collect several elementary facts about the Rumin complex.

\subsection{Invariance under graded automorphisms}

	We will denote the group of graded automorphism of $\goe=\goe_{-3}\oplus\goe_{-2}\oplus\goe_{-1}$ by $\Aut_\grr(\goe)$.
	Let $X_1,\dotsc,X_5$ be a graded basis of $\goe$ such that \eqref{E:basis.Xi} holds.
	With respect to this basis, a graded automorphism $\phi\in\Aut_\grr(\goe)$ is represented by the block matrix
	\begin{equation}\label{E:Aut.basis}
		\phi
		=\begin{pmatrix}\Phi\\&\det\Phi\\&&\Phi\det(\Phi)\end{pmatrix}
	\end{equation}
	where $\Phi$ denotes the $2\times2$ matrix representing the component $\phi|_{\goe_{-1}}\in\GL(\goe_{-1})$ with respect to the basis $X_1,X_2$ of $\goe_{-1}$.
	This follows readily from the relations in \eqref{E:basis.Xi}.
	In particular, restriction provides an isomorphism
	\begin{equation}\label{E:Aut}
		\Aut_\grr(\goe)=\GL(\goe_{-1})\cong\GL_2(\R).
	\end{equation}
	
	Suppose $\phi\in\Aut_\grr(\goe)$.
	By naturality, cf.~\cite{R99,R01,R05} or \cite[Lemma~3.2]{H22}, 
	\begin{equation}\label{E:Aut.D}
		\phi\cdot D_q=D_q
	\end{equation}
	where the dot denotes the natural left action of $\phi$ on $\mathcal U(\goe)\otimes L\bigl(H^q(\goe),H^{q+1}(\goe)\bigr)$.

	Suppose $\rho\colon G\to U(\mathcal H)$ is a (strongly continuous) unitary representation on a Hilbert space $\mathcal H$.
	From \eqref{E:Aut.D} we obtain
	\begin{equation}\label{E:Aut.pi.D}
		(\rho\circ\phi)(D_q)
		=H^{q+1}(\phi)^{-1}\rho(D_q)H^q(\phi).
	\end{equation}
	Here $H^q(\phi)$ denotes the isomorphism induced by $\phi^{-1}$ on the cohomology $H^q(\goe)$, that is, we have covariant functoriality, $H^q(\phi_1\phi_2)=H^q(\phi_1)H^q(\phi_2)$.

	Suppose $h_q$ is a Hermitian inner product on $H^q(\goe)$.
	From \eqref{E:Aut.pi.D}, we obtain
	\begin{equation}\label{E:Aut.pi.Ds}
		\bigl((\rho\circ\phi)(D_q)\bigr)^{*_h}
		=H^q(\phi)^{-1}\rho(D_q)^{*_{\phi\cdot h}}H^{q+1}(\phi).
	\end{equation}
	Here $*_h$ denotes the formal adjoint with respect to the Hermitian inner product on the Hilbert space $\mathcal H$ and the graded inner product $h=\bigoplus_qh_q$ on $H^*(\goe)$.
	Moreover, $\phi\cdot h$ denotes the natural left action of $\phi$ on $h$.

	Combining \eqref{E:|D|} with \eqref{E:Aut.pi.D} and \eqref{E:Aut.pi.Ds}, we obtain
	\begin{equation}\label{E:Aut.pi.|D|}
		|(\rho\circ\phi)(D_q)|_h
		=H^q(\phi)^{-1}|\rho(D_q)|_{\phi\cdot h}H^q(\phi).
	\end{equation}
	In particular, $|(\rho\circ\phi)(D_q)|_h$ and $|\rho(D_q)|_{\phi\cdot h}$ are isospectral.
	From \eqref{E:zeta.|rho.D|} we get
	\begin{equation}\label{E:Aut.zeta}
		\zeta_{|(\rho\circ\phi)(D_q)|_h}(s)
		=
		\zeta_{|\rho(D_q)|_{\phi\cdot h}}(s)
	\end{equation}
	and using \eqref{E:det.def} we obtain
	\begin{equation}\label{E:Aut.det}
		\dets|(\rho\circ\phi)(D_q)|_h
		=
		\dets|\rho(D_q)|_{\phi\cdot h}.
	\end{equation}
	In particular, see \eqref{E:tau.D.def},
	\begin{equation}\label{E:Aut.tau}
		\tau_h\bigl((\rho\circ\phi)(D)\bigr)
		=
		\tau_{\phi\cdot h}(\rho(D)).
	\end{equation}

	Suppose $g$ is a graded Euclidean inner product on $\goe=\goe_{-3}\oplus\goe_{-2}\oplus\goe_{-1}$.
	The induced Hermitian inner product on $\Lambda^*\goe^*\otimes\C$ provides an adjoint $\partial^{*_g}_q$ of the Chevalley--Eilenberg differentials $\partial_q\colon\Lambda^q\goe^*\otimes\C\to\Lambda^{q+1}\goe^*\otimes\C$ which provides an identification
	\[
		H^q(\goe)=\frac{\ker\partial_q}{\img\partial_{q-1}}=\ker\partial_q\cap\ker\partial_{q-1}^{*_g}
	\] 
	by finite dimensional Hodge theory for the Chevalley--Eilenberg complex.
	Via this identification, the induced Hermitian inner product $h_{g,q}$ on $H^q(\goe)$ coincides with the restriction of the Hermitian inner product on $\Lambda^q\goe^*\otimes\C$ to harmonic forms.
	As the action of $\phi$ on $\Lambda^*\goe^*\otimes\C$ commutes with $\partial$, we have 
	\begin{equation}\label{E:Aut.hg}
		\phi\cdot h_{g,q}=h_{\phi\cdot g,q}
	\end{equation}
	where the dots denote the natural left action of $\phi\in\Aut_\grr(\goe)$ on the Hermitian inner product on $H^q(\goe)$ and the graded Euclidean inner product on $\goe$, respectively.

	From \eqref{E:ag} and \eqref{E:bg}, we obtain
	\begin{equation}\label{E:Aut.ag.bg}
		a_{\phi\cdot g}=a_g,\qquad
		b_{\phi\cdot g}=\phi\cdot b_g.
	\end{equation}
	For the expressions appearing in the statement of Theorem~\ref{T:dets.scalar} we find
	\begin{equation}\label{E:Aut.norms}
		\|\alpha\circ\phi|_{\goe_{-1}}\|_g=\|\alpha\|_{\phi\cdot g},\qquad
		\|\alpha\circ\phi|_{\goe_{-1}}\|_{b_g}=\|\alpha\|_{b_{\phi\cdot g}},
	\end{equation}
	as well as $b^{-1}_{\phi\cdot g}(\phi\cdot g)_{-1}=\phi b_g^{-1}g_{-1}\phi|^{-1}_{\goe_{-1}}$ and thus
	\begin{equation}\label{E:Aut.tr}
		\tr\bigl(b^{-1}_{\phi\cdot g}(\phi\cdot g)_{-1}\bigr)=\tr(b_g^{-1}g_{-1}).
	\end{equation}

\subsection{The action of graded automorphisms on the unitary dual}\label{SS:Aut.dual}

	In this section we collect a few facts about the action of $\Aut_\grr(\goe)$ on the set of equivalence classes of irreducible unitary representations of $G$.
	According to Kirillov \cite{K04} this may be identified with the action of $\Aut_\grr(\goe)$ on the set of coadjoint orbits $\goe^*/G$.
	If $\rho_{\mathcal O}$ denotes an irreducible unitary representation corresponding to the coadjoint orbit $\mathcal O$, and $\phi\in\Aut_\grr(\goe)$, then the representation $\rho_{\mathcal O}\circ\phi$ is unitarily equivalent to the representation corresponding to the coadjoint orbit $\phi^t(\mathcal O)$ where $\phi^t\in\GL(\goe^*)$ denotes the transposed map.
	Below we will make this more explicit in terms of the parametrizations of different types of representations described in Section~\ref{S:results}.
	
	For $\alpha\in\goe_1^*$ we clearly have
	\begin{equation}\label{E:Aut.rho.alpha}
		\rho_\alpha\circ\phi
		=\rho_{\alpha\circ\phi|_{\goe_{-1}}}
	\end{equation}
	cf.~\eqref{E:rep.scalar}.
	In particular, the action of $\Aut_\grr(\goe)$ on the space of nontrivial unitary scalar representations is transitive, cf.~\eqref{E:Aut}.
	
	For $\hbar\neq0$ we have a unitary equivalence of representations,
	\begin{equation}\label{E:Aut.rho.h}
		\rho_\hbar\circ\phi\sim\rho_{\hbar\det(\phi|_{\goe_{-1}})}.
	\end{equation}
	This follows from \eqref{E:Aut.basis} and the classification of irreducible unitary representations recalled in Section~\ref{S:results}, characterizing $\rho_\hbar$ as the unique irreducible unitary representation, up to unitary equivalence, for which $\rho_\hbar(X_3)=\mathbf i\hbar$ and $\rho_\hbar(X_4)=0=\rho_\hbar(X_5)$, cf.~\eqref{E:Schroedinger}.
	In particular, the action of $\Aut_\grr(\goe)$ on the space of Schr\"odinger representations is transitive, cf.~\eqref{E:Aut}.

	For real numbers $\lambda,\mu,\nu$ with $(\lambda,\mu)\neq(0,0)$ we have a unitary equivalence of representations,
	\begin{equation}\label{E:Aut.rho.gen}
		\rho_{\lambda,\mu,\nu}\circ\phi\sim\rho_{\lambda',\mu',\nu'}
	\end{equation}
	where 
	\[
		(\lambda',\mu')=(\lambda,\mu)\Phi\det\Phi,\qquad 
		\nu'=\nu{\det}^2\Phi,
	\]
	and $\Phi$ is the matrix appearing in \eqref{E:Aut.basis}.
	Indeed, $\rho_{\lambda,\mu,\nu}$ can be characterized as the unique irreducible unitary representation, up to unitary equivalence, satisfying \eqref{E:rep.gen.X4}, \eqref{E:rep.gen.X5}, and \eqref{E:rep.gen.nu}.
	
	In particular, the action of $\Aut_\grr(\goe)$ on the space of generic representations has three orbits, characterized by $\nu<0$, $\nu=0$, and $\nu>0$, respectively.
	The latter will play an important role in the proof of Theorem~\ref{T:tor.gen} for it implies, by connectedness, that a continuous function on the space of generic representations $\left\{(\lambda,\mu,\nu)\in\R^3:(\lambda,\mu)\neq(0,0)\right\}$ which is constant on orbits of $\Aut_\grr(\goe)$ must be constant altogether.

\subsection{Formulas for the Rumin differentials}\label{SS:D.formula}

	We continue to use a graded basis $X_1,\dotsc,X_5$ such that \eqref{E:basis.Xi} holds.
	Let $\chi^j$ denote the dual basis, that is, $\chi^j(X_k)=\delta^j_k$.
	Then $\chi^1,\chi^2,\chi^3,\chi^4,\chi^5$ is a graded basis of $\goe^*=\goe_1^*\oplus\goe_2^*\oplus\goe_3^*$ where we use the notation $\goe_j^*:=(\goe_{-j})^*$ for the dual grading.
	More explicitly, $\chi^1,\chi^2$ is a basis of $\goe_1^*=(\goe_{-1})^*$; $\chi^3$ is a basis of $\goe_2^*=(\goe_{-2})^*$; and $\chi^4,\chi^5$ is a basis of $\goe_3^*=(\goe_{-3})^*$.
	A straight forward computation \cite[Appendix~B]{DH17} yields the following bases for the kernels of the Chevalley--Eilenberg differentials $\partial_q$:
	\begin{equation}\label{E:bases}
		\begin{array}{|c|r|l|}
			\multicolumn{3}{c}{\rule[-1.5ex]{0ex}{1ex}\text{basis of $\ker\bigl(\Lambda^q\goe^*\otimes\C\xrightarrow{\partial_q}\Lambda^{q+1}\goe^*\otimes\C\bigr)$}}
			\\\hline
			q
			&\text{basis of $\img\partial_{q-1}$}
			&\rule[-1.6ex]{0ex}{4.4ex}\text{induces basis of $H^q(\goe)=\frac{\ker\partial_q}{\img\partial_{q-1}}$}
			\\\hline
			0
			&\emptyset
			&1
			\\
			1
			&\emptyset
			&\chi^1,\chi^2
			\\
			2
			&\chi^{12},\chi^{13},\chi^{23}
			&\chi^{14},\tfrac1{\sqrt2}(\chi^{15}+\chi^{24}),\chi^{25}
			\\
			3
			&\chi^{123},\chi^{124},\chi^{125},\tfrac1{\sqrt2}(\chi^{135}-\chi^{234})
			&\chi^{134},\tfrac1{\sqrt2}(\chi^{135}+\chi^{234}),\chi^{235}
			\\
			4
			&\chi^{1234},\chi^{1235},\chi^{1245}
			&\chi^{1345},\chi^{2345}
			\\
			5
			&\emptyset
			&\chi^{12345}
			\\\hline
		\end{array}
	\end{equation}
	Here we use the notation $\chi^{j_1\dotsc j_k}=\chi^{j_1}\wedge\cdots\wedge\chi^{j_k}$.

	With respect to the basis of $H^q(\goe)=\frac{\ker\partial_q}{\img\partial_{q-1}}$ provided by the forms in the right column in \eqref{E:bases}, the Rumin differentials $D_q\in\mathcal U(\goe)\otimes L(H^q(\goe),H^{q+1}(\goe))$ are:
	\begin{align}
		\label{E:D0}
		D_0&=\begin{pmatrix}X_1\\ X_2\end{pmatrix}
		\\\label{E:D1}
		D_1&=\begin{pmatrix}-X_{112}-X_{13}-X_4&X_{111}\\-\sqrt2X_{122}-\sqrt2X_5&\sqrt2X_{211}-\sqrt2X_4\\-X_{222}&X_{221}-X_{23}-X_5\end{pmatrix}
		\\\label{E:D2}
		D_2&=\begin{pmatrix}-X_{12}-X_3&X_{11}/\sqrt2&0\\-X_{22}/\sqrt2&-\tfrac32X_3&X_{11}/\sqrt2\\0&-X_{22}/\sqrt2&X_{21}-X_3\end{pmatrix}
		\\\label{E:D3}
		D_3&=\begin{pmatrix}X_{122}+X_{32}-X_5&-\sqrt2X_{112}+\sqrt2X_4&X_{111}\\X_{222}&-\sqrt2X_{221}-\sqrt2X_5&X_{211}-X_{31}+X_4\end{pmatrix}
		\\\label{E:D4}
		D_4&=\begin{pmatrix}-X_2,X_1\end{pmatrix}
	\end{align}
	Here we use the notation $X_{j_1\cdots j_k}=X_{j_1}\cdots X_{j_k}$ in the universal enveloping algebra $\mathcal U(\goe)$.
	These formulas for $D_q$ have been derived in \cite[Appendix~B]{DH17}, where they are expressed with respect to a slightly different basis, see also \cite[Example~4.21]{DH22}.

	Let $g$ be a graded Euclidean inner product in $\goe$.	
	Up to a graded automorphism of $\goe$, we may assume that $X_1,X_2$ is a orthonormal basis of $\goe_{-1}$ with respect to $g_{-1}$, cf.~\eqref{E:Aut}.
	By the principal axis theorem and \eqref{E:Aut.ag.bg}, we may moreover assume $b_g(X_1,X_2)=0$, i.e., the matrix $b^g_{ij}:=b_g(X_i,X_j)$, $1\leq i,j\leq2$, is diagonal.
	Hence, the graded Euclidean inner product $g$ on $\goe$ and the induced Euclidean inner product $g^{-1}$ on $\goe^*$, with respect to the basis $X_1,X_2,X_3,X_4,X_5$ of $\goe$ and the dual basis $\chi^1,\chi^2,\chi^3,\chi^4,\chi^5$ on $\goe^*$ are given by the matrices
	\[
		g=\begin{pmatrix}1\\&1\\&&a_g\\&&&a_gb^g_{11}\\&&&&a_gb^g_{22}\end{pmatrix}
		\quad\text{and}\quad
		g^{-1}=\begin{pmatrix}1\\&1\\&&1/a_g\\&&&b^{11}_g/a_g\\&&&&b^{22}_g/a_g\end{pmatrix},
	\]
	respectively, where $b^{jj}_g:=1/b^g_{jj}$, see \eqref{E:ag} and \eqref{E:bg}.
	With respect to the ordered bases of $H^q(\goe)$ provided by the forms in the right column of \eqref{E:bases}, the Hermitian inner products $h_{g,q}$ on $H^q(\goe)=\frac{\ker\partial_q}{\img\partial_{q-1}}$ are given by the diagonal matrices:
	\begin{align}\label{E:hq}
		\notag
		h_{g,0}&=\begin{pmatrix}1\end{pmatrix},
		&
		h_{g,1}&=\begin{pmatrix}1\\&1\end{pmatrix},
		\\
		h_{g,2}&=\frac1{a_g}\begin{pmatrix}b^{11}_g\\&\frac{b^{11}_g+b^{22}_g}2\\&&b^{22}_g\end{pmatrix},
		&
		h_{g,3}&=\frac{b^{11}_gb^{22}_g}{a^2_g}\begin{pmatrix}\frac1{b^{22}_g}\\&\frac2{b^{11}_g+b^{22}_g}\\&&\frac1{b^{11}_g}\end{pmatrix},
		\\\notag
		h_{g,4}&=\frac{b^{11}_gb^{22}_g}{a^3_g}\begin{pmatrix}1\\&1\end{pmatrix},
		&
		h_{g,5}&=\frac{b^{11}_gb^{22}_g}{a^3_g}\begin{pmatrix}1\end{pmatrix}.
	\end{align}
	Note here that all forms in the right column of \eqref{E:bases} are orthogonal to $\img\partial$ but $\frac1{\sqrt2}(\chi^{135}+\chi^{234})$ if $b^{11}_g\neq b^{22}_g$.
	The form representing the same cohomology class that is orthogonal to $\img\partial_2$ is
	\[
		\frac1{\sqrt2}\bigl(\chi^{135}+\chi^{234}\bigr)+\frac{b^{11}_g-b^{22}_g}{b^{11}_g+b^{22}_g}\cdot\frac1{\sqrt2}\bigl(\chi^{135}-\chi^{234}\bigr)
		=\frac{\sqrt2}{b^{11}_g+b^{22}_g}\bigl(b^{11}_g\chi^{135}+b^{22}_g\chi^{234}\bigr).
	\]
	This needs to be taken into account when computing the middle entry of $h_{g,3}$.

\subsection{Poincar\'e duality}\label{SS:PD}

	Let $g$ be a graded Euclidean inner product on $\goe$, and consider the induced Hermitian inner product $h_g$ on $H^*(\goe)$.
	The Hodge star operator $\star_{g,q}\colon\Lambda^q\goe^*\otimes\C\to\Lambda^{5-q}\goe^*\otimes\C$ associated with $g$ restricts to an isometry on $H^*(\goe)=\ker\partial\cap\ker\partial^{*_g}$ that will be denoted by $\star_g=\bigoplus_q\star_{g,q}$.
	We have
	\[
		D_q^{*_{h_g}}
		=(-1)^{q+1}\star^{-1}_gD_{4-q}\star_g
	\]
	and
	\[
		D_q^{*_{h_g}}D_q
		=\star_g^{-1}D_{4-q}D_{4-q}^{*_{h_g}}\star_g,
	\]
	see \cite[Section~2]{R99}, \cite[Proposition~2.8]{R01}, or \cite[Section~3.3]{H22}.
	Moreover, $\star_g^2=\id$ as $\dim\goe=5$ is odd.

	If $\rho$ is a unitary representation of $G$, we obtain
	\begin{equation}\label{E:Hodge.star}
		\rho(D_q)^{*_{h_g}}=(-1)^{q+1}\star^{-1}_g\rho(D_{4-q})\star_g
	\end{equation}
	and
	\begin{equation}\label{E:Hodge.DDs}
		\rho(D_q)^{*_{h_g}}\rho(D_q)
		=\star_g^{-1}\rho(D_{4-q})\rho(D_{4-q})^{*_{h_g}}\star_g.
	\end{equation}
	As $A^*A$ and $AA^*$ have the same nontrivial spectrum, this yields
	\begin{equation}\label{E:Hodge.zeta}
		\zeta_{|\rho(D_q)|_{h_g}}(s)
		=
		\zeta_{|\rho(D_{4-q})|_{h_g}}(s)
	\end{equation}
	and then
	\begin{equation}\label{E:Hodge.det}
		\dets|\rho(D_q)|_{h_g}
		=
		\dets|\rho(D_{4-q})|_{h_g}
	\end{equation}
	cf.~\eqref{E:zeta.|rho.D|} and \eqref{E:det.def}.
	In particular, the analytic torsion \eqref{E:tau.D.def} may be expressed in terms of three determinants,
	\begin{equation}\label{E:PD.tor}
		\tau_{h_g}(\rho(D))
		=\frac{\dets^2|\rho(D_0)|_{h_g}\cdot\dets|\rho(D_2)|_{h_g}}{\dets^2|\rho(D_1)|_{h_g}}.
	\end{equation}

	If $X_1,X_2$ is an orthonormal basis and $b_g(X_1,X_2)=0$, then, with respect to the basis of $H^*(\goe)$ provided by the forms in the right column in \eqref{E:bases}, the Hodge star operators are represented by the following matrices:
	\begin{align*}
		\star_{g,0}&=\frac{\sqrt{a_g^3}}{\sqrt{b_g^{11}b_g^{22}}}\begin{pmatrix}1\end{pmatrix},
		&
		\star_{g,1}&=\frac{\sqrt{a_g^3}}{\sqrt{b_g^{11}b_g^{22}}}\begin{pmatrix}&-1\\1\end{pmatrix},
		\\
		\star_{g,2}&=\frac{\sqrt{a_g}}{\sqrt{b_g^{11}b_g^{22}}}\begin{pmatrix}&&b_g^{22}\\&\frac{b_g^{11}+b_g^{22}}{-2}\\b_g^{11}\end{pmatrix},
		&
		\star_{g,3}&=\frac{\sqrt{b_g^{11}b_g^{22}}}{\sqrt{a_g}}\begin{pmatrix}&&\frac1{b_g^{11}}\\&\frac{-2}{b_g^{11}+b_g^{22}}\\\frac1{b_g^{22}}\end{pmatrix},
		\\
		\star_{g,4}&=\frac{\sqrt{b_g^{11}b_g^{22}}}{\sqrt{a_g^3}}\begin{pmatrix}&1\\-1\end{pmatrix},
		&
		\star_{g,5}&=\frac{\sqrt{b_g^{11}b_g^{22}}}{\sqrt{a_g^3}}\begin{pmatrix}1\end{pmatrix}.
	\end{align*}

\section{The Rumin complex in the scalar representations}\label{S:scalar}

	In this section we will prove Theorem~\ref{T:dets.scalar}.
	Fix $0\neq\alpha\in\goe_1^*$ and consider the corresponding nontrivial unitary representation $\rho_\alpha$ of $G$ on $\C$.
	We will mostly omit the symbol $\rho_\alpha$ in our notation throughout this section.
	All operators are understood to act in this representation.

	Let $X_1,\dotsc,X_5$ denote a graded basis of $\goe$ such that \eqref{E:basis.Xi} holds.
	In the representation $\rho_\alpha$ we have
	\[
		X_1=\mathbf iv,\qquad X_2=\mathbf iw,\qquad X_3=X_4=X_5=0,
	\]
	where $v:=\alpha(X_1)$ and $w:=\alpha(X_2)$, see \eqref{E:rep.scalar}. 
	Note that $(v,w)\neq(0,0)$ as $\rho_\alpha$ is assumed to be nontrivial.

	From \eqref{E:D0}--\eqref{E:D4} we see that with respect to the basis of $H^q(\goe)$ provided by the forms in the right column in \eqref{E:bases}, the Rumin differentials in the scalar representations are given by the matrices:
	\begin{align}
		D_0
		&=\mathbf i\begin{pmatrix}v\\w\end{pmatrix},
		\label{E:D0.scalar}
		\\
		D_1
		&=\mathbf i^3\begin{pmatrix}-v^2w&v^3\\-\sqrt2vw^2&\sqrt2v^2w\\-w^3&vw^2\end{pmatrix}
		=\mathbf i^3\begin{pmatrix}v^2\\\sqrt2vw\\w^2\end{pmatrix}\begin{pmatrix}-w&v\end{pmatrix},
		\label{E:D1.scalar}
		\\
		D_2
		&=\mathbf i^2\begin{pmatrix}-vw&v^2/\sqrt2&0\\-w^2/\sqrt2&0&v^2/\sqrt2\\0&-w^2/\sqrt2&vw\end{pmatrix},
		\notag
		\\
		&=\mathbf i^2\begin{pmatrix}-v&0\\-w/\sqrt2&v/\sqrt2\\0&w\end{pmatrix}\begin{pmatrix}w&-v/\sqrt2&0\\0&-w/\sqrt2&v\end{pmatrix},
		\label{E:D2.scalar}
		\\
		D_3
		&=\mathbf i^3\begin{pmatrix}vw^2&-\sqrt2v^2w&v^3\\w^3&-\sqrt2vw^2&v^2w\end{pmatrix}
		=\mathbf i^3\begin{pmatrix}v\\w\end{pmatrix}\begin{pmatrix}w^2&-\sqrt2vw&v^2\end{pmatrix},
		\label{E:D3.scalar}
		\\
		D_4
		&=\mathbf i\begin{pmatrix}-w&v\end{pmatrix}.
		\label{E:D4.scalar}
	\end{align}

	Let $g$ be a graded Euclidean inner product on $\goe$.
	In view of \eqref{E:Aut}, \eqref{E:Aut.det}, \eqref{E:Aut.hg}, \eqref{E:Aut.ag.bg}, \eqref{E:Aut.norms}, \eqref{E:Aut.tr}, and \eqref{E:Aut.rho.alpha}, we may assume w.l.o.g.\ that $X_1,X_2$ is an orthonormal basis of $\goe_{-1}$ with respect to $g_{-1}$.
	By the principal axis theorem, we may, moreover, assume $b(X_1,X_2)=0$.
	Hence, with respect to the basis of $H^q(\goe)$ provided the forms in the right column of \eqref{E:bases}, the Hermitian inner products on $H^q(\goe)$ are represented by the matrices in \eqref{E:hq}.
	To keep the notation light, we drop the subscript $g$ and write $a=a_g$, $b=b_g$, $b^{ij}=b^{ij}_g$, $h=h_g$, and $h_q=h_{g,q}$.
	Note that
	\begin{equation}\label{E:norms}
		\|\alpha\|_g=\bigl(v^2+w^2\bigr)^{1/2},
		\qquad
		\|\alpha\|_b=\bigl(b^{11}v^2+b^{22}w^2\bigr)^{1/2}.
	\end{equation}
	and
	\begin{equation}\label{E:trbg}
		\tr(b^{-1}g_{-1})=b^{11}+b^{22}.
	\end{equation}
	As $D_q^{*_h}=h_q^{-1}D_q^*h_{q+1}$, we have from \eqref{E:|D|}
	\begin{equation}\label{E:det.basic}
		\dets|D_q|_h
		=\dets^{1/2}\bigl(h_q^{-1}D_q^*h_{q+1}D_q\bigr).
	\end{equation}
	Since the Rumin complex in the scalar representation is finite dimensional, $\dets$ here is just the product of all nonzero eigenvalues.
	Clearly, 
	\begin{equation}\label{E:specAB}
		\spec_*(AB)=\spec_*(BA)
	\end{equation} 
	for any two linear maps $A\colon V\to W$ and $B\colon W\to V$ between finite dimensional (complex) vector spaces, where $\spec_*$ denotes the nonzero spectrum.

	Combining \eqref{E:D0.scalar}, \eqref{E:det.basic}, \eqref{E:hq}, and \eqref{E:norms}, we obtain
	\[
		\dets|D_0|_h
		={\det}^{1/2}\bigl(D_0^*D_0\bigr)
		=(v^2+w^2)^{1/2}
		=\|\alpha\|_g.
	\]


	Using the factorization for $D_1$ provided in \eqref{E:D1.scalar}, as well as \eqref{E:hq}, \eqref{E:specAB}, and \eqref{E:norms}, we see that $D_1^{*_h}D_1=D_1^*h_2D_1$ has the same nonzero spectrum as
	\begin{multline*}
		\frac1a\cdot
		\begin{pmatrix}-w&v\end{pmatrix}\begin{pmatrix}-w&v\end{pmatrix}^*
		\cdot
		\begin{pmatrix}v^2\\\sqrt2vw\\w^2\end{pmatrix}^*
		\begin{pmatrix}b^{11}\\&\frac{b^{11}+b^{22}}2\\&&b^{22}\end{pmatrix}
		\begin{pmatrix}v^2\\\sqrt2vw\\w^2\end{pmatrix}
		\\
		=\frac{\bigl(v^2+w^2\bigr)^2\bigl(b^{11}v^2+b^{22}w^2\bigr)}a
		=\frac{\|\alpha\|_g^4\|\alpha\|_b^2}a.
	\end{multline*}
	In view of \eqref{E:det.basic} this gives $\dets|D_1|_h=\frac{\|\alpha\|_g^2\|\alpha\|_b}{\sqrt a}$.


	The factorization of $D_2$ in \eqref{E:D2.scalar}, together with \eqref{E:hq}, \eqref{E:specAB}, \eqref{E:norms}, and \eqref{E:trbg} show that $D_2^{*_h}D_2=h_2^{-1}D_2^*h_3D_2$ has the same nonzero spectrum as
	\begin{align*}
		\frac{b^{11}b^{22}}a&
		\cdot\begin{pmatrix}-v&0\\\frac{-w}{\sqrt2}&\frac v{\sqrt2}\\0&w\end{pmatrix}^*\begin{pmatrix}\frac1{b^{22}}\\&\frac2{b^{11}+b^{22}}\\&&\frac1{b^{11}}\end{pmatrix}\begin{pmatrix}-v&0\\\frac{-w}{\sqrt2}&\frac v{\sqrt2}\\0&w\end{pmatrix}
		\\
		&\qquad\qquad\cdot\begin{pmatrix}w&\frac{-v}{\sqrt2}&0\\0&\frac{-w}{\sqrt2}&v\end{pmatrix}\begin{pmatrix}\frac1{b^{11}}\\&\frac2{b^{11}+b^{22}}\\&&\frac1{b^{22}}\end{pmatrix}\begin{pmatrix}w&\frac{-v}{\sqrt2}&0\\0&\frac{-w}{\sqrt2}&v\end{pmatrix}^*
		\\
		&=\frac{b^{11}b^{22}}a
		\cdot\begin{pmatrix}\frac{v^2}{b^{22}}+\frac{w^2}{b^{11}+b^{22}}&-\frac{vw}{b^{11}+b^{22}}\\-\frac{vw}{b^{11}+b^{22}}&\frac{w^2}{b^{11}}+\frac{v^2}{b^{11}+b^{22}}\end{pmatrix}
		\cdot\begin{pmatrix}\frac{w^2}{b^{11}}+\frac{v^2}{b^{11}+b^{22}}&\frac{vw}{b^{11}+b^{22}}\\\frac{vw}{b^{11}+b^{22}}&\frac{v^2}{b^{22}}+\frac{w^2}{b^{11}+b^{22}}\end{pmatrix}
		\\
		&=\frac{\bigl(v^2+w^2\bigr)\bigl(b^{11}v^2+b^{22}w^2\bigr)}{a(b^{11}+b^{22})}\begin{pmatrix}1&0\\0&1\end{pmatrix}
		\\&=\frac{\|\alpha\|_g^2\|\alpha\|_b^2}{a\cdot\tr(b^{-1}g_{-1})}\begin{pmatrix}1&0\\0&1\end{pmatrix}.
	\end{align*}
	In view of \eqref{E:det.basic} this gives $\dets|D_2|_h=\frac{\|\alpha\|_g^2\|\alpha\|_b^2}{a\cdot\tr(b^{-1}g_{-1})}$.

	Combining these formulas for $\dets|D_q|_h$, $q=0,1,2$ with \eqref{E:Hodge.det} and \eqref{E:tau.D.def}, we complete the proof of Theorem~\ref{T:dets.scalar}.

	For the entire spectra of $|D_q|_h$ in the scalar representation $\rho_\alpha$ we obtain:
	\begin{align*}
		\spec\bigl(|D_0|_h\bigr):&\qquad\|\alpha\|_g\\
		\spec\bigl(|D_1|_h\bigr):&\qquad0,\tfrac{\|\alpha\|_g^2\|\alpha\|_b}{\sqrt a}\\
		\spec\bigl(|D_2|_h\bigr):&\qquad0,\tfrac{\|\alpha\|_g\|\alpha\|_b}{\sqrt{a\cdot\tr(b^{-1}g_{-1})}},\tfrac{\|\alpha\|_g\|\alpha\|_b}{\sqrt{a\cdot\tr(b^{-1}g_{-1})}}\\
		\spec\bigl(|D_3|_h\bigr):&\qquad0,0,\tfrac{\|\alpha\|_g^2\|\alpha\|_b}{\sqrt a}\\
		\spec\bigl(|D_4|_h\bigr):&\qquad0,\|\alpha\|_g
	\end{align*}

\section{The Rumin complex in the Schr\"odinger representations}\label{S:Schroedinger}

	In this section we will prove Theorem~\ref{T:dets.Schroedinger}.
	Fix a real number $\hbar\neq0$ and consider the Schr\"odinger representation $\rho_\hbar$ of $G$ on $L^2(\R,d\theta)$, see \eqref{E:Schroedinger}.
	We will mostly drop the symbol $\rho_\hbar$ in our notation throughout this section.
	All operators are considered as acting in this representation.

	Let $X_1,\dotsc,X_5$ denote a basis as in \eqref{E:basis.Xi}.
	In the Schr\"odinger representation
	\begin{equation}\label{E:X3.Schroedinger}
		X_3=[X_1,X_2]=\mathbf i\hbar
	\end{equation} 
	is central and we have 
	\begin{equation}\label{E:X45.Schroedinger}
		X_4=[X_1,X_3]=0,\qquad
		X_5=[X_2,X_3]=0.
	\end{equation}
	Hence, from \eqref{E:D0}--\eqref{E:D4} we see that with respect to the bases of $H^q(\goe)$ provided by the forms in the right column of \eqref{E:bases}, the Rumin differentials are represented by the matrices:
	\begin{align}
		D_0&=\begin{pmatrix}X_1\\ X_2\end{pmatrix},\label{E:D0.Schroedinger}
		\\
		D_1&=\begin{pmatrix}-X_{112}-X_{13}&X_{111}\\-\sqrt2X_{122}&\sqrt2X_{211}\\-X_{222}&X_{221}-X_{23}\end{pmatrix},\label{E:D1.Schroedinger}\\
		D_2&=\begin{pmatrix}-X_{12}-X_3&X_{11}/\sqrt2&0\\-X_{22}/\sqrt2&-\tfrac32X_3&X_{11}/\sqrt2\\0&-X_{22}/\sqrt2&X_{21}-X_3\end{pmatrix},\label{E:D2.Schroedinger}\\
		D_3&=\begin{pmatrix}X_{122}+X_{23}&-\sqrt2X_{112}&X_{111}\\X_{222}&-\sqrt2X_{221}&X_{211}-X_{13}\end{pmatrix},\label{E:D3.Schroedinger}\\
		D_4&=\begin{pmatrix}-X_2,X_1\end{pmatrix}\label{E:D4.Schroedinger}.
	\end{align}
	
	Let $g$ be a graded Euclidean inner product on $\goe$ with $b_g$ proportional to $g$.
	For the time being, we will assume $a_g=1$ and $b_g=g$, cf.~\eqref{E:ag} and \eqref{E:bg}.
	In view of \eqref{E:Aut}, \eqref{E:Aut.det}, \eqref{E:Aut.hg}, \eqref{E:Aut.ag.bg}, and \eqref{E:Aut.rho.h}, we may assume w.l.o.g.\ that $X_1,X_2$ is an orthonormal basis of $\goe_{-1}$.
	Then the Hermitian inner products $h_{g,q}$ on $H^q(\goe)$ are represented by the identical matrices with respect to the basis provided by the forms in the right column of \eqref{E:bases}, cf.~\eqref{E:hq}.
	Hence, the adjoint of a matrix of operators with respect to $h_g$ is represented by the transposed conjugate matrix which will be denoted using the undecorated symbol $*$.
	By unitarity, 
	\begin{equation}\label{E:Xjs}
		X_j^*=-X_j.
	\end{equation}

\subsection{The spectrum of $D_0^*D_0$}

	Let us introduce the notation
	\begin{equation}\label{E:Delta}
		\Delta^*
		=\Delta
		:=D_0^*D_0
		=X_1^*X_1+X_2^*X_2
		=-X_{11}-X_{22}.
	\end{equation}
	In the Schr\"odinger representation $\rho_\hbar$, see~\eqref{E:Schroedinger}, this operator becomes the classical quantum harmonic oscillator 
	\[
		\Delta
		=D_0^*D_0
		=|\hbar|\bigl(\theta^2-\partial_\theta^2\bigr)
		=|\hbar|\bigl(1+(\theta+\partial_\theta)^*(\theta+\partial_\theta)\bigr)
	\]
	whose spectrum is well known,
	\begin{equation}\label{E:spec0.Schroedinger}
		\spec(\Delta)=\spec(D_0^*D_0):\quad|\hbar|(2n+1),\quad n=0,1,2,3,\dotsc
	\end{equation}
	All eigenvalues have multiplicity one.

\subsection{The spectrum of $D_1^*D_1$}\label{SS:spec1.Schroedinger}

	In this section we will compute the spectrum of $D_1^*D_1$ by means of the factorization
	\begin{equation}\label{E:D1=ZA}
		D_1=ZA
	\end{equation}
	where
	\begin{equation}\label{E:Z}
		Z
		=\begin{pmatrix}0&-X_1\\X_1/\sqrt2&-X_2/\sqrt2\\X_2&0\end{pmatrix}
	\end{equation}
	and
	\begin{equation}\label{E:A}
		A^*
		=A
		:=\begin{pmatrix}-X_{22}&X_{21}-X_3\\X_{12}+X_3&-X_{11}\end{pmatrix}.
	\end{equation}
	It will be convenient to use the notation
	\begin{equation}\label{E:J}
		-J^{-1}
		=-J^*
		=J
		:=\begin{pmatrix}0&1\\-1&0\end{pmatrix}.
	\end{equation}
	Using \eqref{E:X3.Schroedinger}, \eqref{E:X45.Schroedinger}, \eqref{E:D0.Schroedinger}, \eqref{E:D1.Schroedinger}, and \eqref{E:Xjs}, one readily checks \eqref{E:D1=ZA} as well as the following relations:
	\begin{align}
		A
		&=(JD_0)(JD_0)^*-JX_3,
		\label{E:AJD0}
		\\
		AD_0
		&=0,
		\label{E:AD0}
		\\
		Z^*Z
		&=A+\tfrac12D_0D_0^*+2JX_3,
		\label{E:Z*Z}
		\\
		D_0^*JD_0
		&=-X_3,
		\label{E:D0sJD0}
		\\
		D_0D_0^*+JD_0D_0^*J^*
		&=\begin{pmatrix}\Delta&-X_3\\X_3&\Delta\end{pmatrix}.
		\label{E:D0Delta}
	\end{align}

	From \eqref{E:AJD0} and \eqref{E:J} we also obtain
	\begin{equation}\label{E:D0D0s}
		D_0D_0^*
		=J^*AJ+JX_3,
	\end{equation}
	and then 
	\begin{equation}\label{E:AJA}
		AJA
		=AX_3.
	\end{equation}
	using \eqref{E:AD0}.
	Combining \eqref{E:D1=ZA}, \eqref{E:Z*Z}, \eqref{E:AD0}, \eqref{E:AJA}, $A^*=A$, and using the fact that $X_3$ is central, we obtain
	\begin{align}\label{E:D1*D1}
		D_1^*D_1=A(A^2+2X_{33}).
	\end{align}

	The symbol $\spec_*(P):=\spec(P)\setminus\{0\}$ will be used to denote the nonzero part of the spectrum of an operator $P$.

	\begin{lemma}\label{L:specA2}
		We have $\spec_*(A^2)=\spec(\Delta^2)$, including multiplicities.
	\end{lemma}

	\begin{proof}
		We consider the operator 
		\begin{equation}\label{E:B}B:=AJD_0\end{equation}
		which satisfies the relations
		\begin{equation}\label{E:AB=BD}
			D_0^*B=0,\qquad AB=B\Delta.
		\end{equation}
		The first equation immediately follows from \eqref{E:AD0}.
		To see the second equation, we multiply \eqref{E:D0D0s} with $AJ$ from the left and with $D_0$ from the right, and use \eqref{E:Delta}, \eqref{E:B}, and the fact that $X_3$ is central.
		
		In view of \eqref{E:AD0} and \eqref{E:AB=BD}, the decomposition 
		\[
			L^2(\R)=\ker B\oplus\ker^\perp B
		\]
		is invariant under $\Delta=\Delta^*$, the decomposition 
		\[
			L^2(\R)\otimes\C^2=\cimg D_0\oplus\ker D_0^*=\cimg D_0\oplus\cimg B\oplus(\ker D_0^*\cap\ker B^*)
		\] 
		is invariant under $A=A^*$, and we have:
		\begin{align}
			\spec(\Delta)		&=\spec(\Delta|_{\ker B})\sqcup\spec(\Delta|_{\ker^\perp B}),\label{E:spec1}\\
			\spec_*(A)		&=\spec_*(A|_{\ker D_0^*}),\\
			\spec(A|_{\ker D_0^*})	&=\spec(A|_{\img B})\sqcup\spec(A|_{\ker D_0^*\cap\ker B^*}),\label{E:spec3}\\
			\spec(A|_{\img B})	&=\spec(\Delta|_{\ker^\perp B}).\label{E:spec4}
		\end{align}
		Hence, it remains to show 
		\begin{equation}\label{E:spec5}
			\spec(-A|_{\ker D_0^*\cap\ker B^*})=\spec(\Delta|_{\ker B}).
		\end{equation}
		Indeed, combining \eqref{E:spec1}--\eqref{E:spec5}, we obtain \[\spec_*(A^2)=\spec_*(\Delta^2)=\spec(\Delta^2),\] for $\Delta$ is strictly positive.

		Let us now show \eqref{E:spec5}. 
		Using \eqref{E:Delta}, \eqref{E:D0D0s}, \eqref{E:J}, and \eqref{E:B}, we obtain
		\begin{equation}\label{E:B2}
			B=JD_0\Delta+D_0X_3.
		\end{equation}
		Combining this with \eqref{E:Delta}, \eqref{E:J}, and \eqref{E:D0sJD0}, a straight forward computations yields
		\begin{equation}\label{E:B*B}
			B^*B=\Delta(\Delta^2+X_{33})=\Delta(\Delta+\ii X_3)(\Delta-\ii X_3).
		\end{equation}
		As $\Delta>0$, we obtain
		\begin{equation}\label{E:kerB}
			\ker B
			=\ker(\Delta+\ii X_3)(\Delta-\ii X_3)
			=\begin{cases}
				\ker(\Delta+\ii X_3)&\text{if $\ii X_3<0$, and}
				\\
				\ker(\Delta-\ii X_3)&\text{if $\ii X_3>0$.}
			\end{cases}
		\end{equation}
		Recall from \eqref{E:X3.Schroedinger} that $\mathbf iX_3=-\hbar$.
		In view of \eqref{E:B2} and \eqref{E:D0Delta} we have \begin{equation}\label{E:kerD0B}\ker D_0^*\cap\ker B^*=\ker D_0^*\cap\ker D_0^*J^*=\ker\begin{pmatrix}\Delta&-X_3\\X_3&\Delta\end{pmatrix}.\end{equation}
		Using 
		\[
			\begin{pmatrix}1&1\\-\ii&\ii\end{pmatrix}^{-1}
			\begin{pmatrix}\Delta&-X_3\\X_3&\Delta\end{pmatrix}
			\begin{pmatrix}1&1\\-\ii&\ii\end{pmatrix}
			=\begin{pmatrix}\Delta+\ii X_3&0\\0&\Delta-\ii X_3\end{pmatrix}
		\]
		we obtain from \eqref{E:kerD0B} an isomorphism given by conjugation
		\begin{equation}\label{E:kerrr}
			\ker D_0^*\cap\ker B^*
			\cong\begin{cases}
				\ker(\Delta+\ii X_3)&\text{if $\ii X_3<0$, and}
				\\
				\ker(\Delta-\ii X_3)&\text{if $\ii X_3>0$.}
			\end{cases}
		\end{equation}
		From \eqref{E:AJD0} we see that 
		\[
			-A=JX_3=\begin{pmatrix}\Delta&0\\0&\Delta\end{pmatrix}\qquad\text{on the subspace \eqref{E:kerD0B}.}
		\]
		Moreover, the the isomorphism in \eqref{E:kerrr} intertwines this operator with $\Delta$ acting on $\ker B$, cf.~\eqref{E:kerB}.
		This completes the proof of \eqref{E:spec5}, whence the lemma.
	\end{proof}

	Combining Lemma~\ref{L:specA2} with \eqref{E:D1*D1}, \eqref{E:spec0.Schroedinger}, and \eqref{E:X3.Schroedinger}, we obtain the (simple) spectrum of $D_1^*D_1$ in the Schr\"o\-din\-ger representation $\rho_\hbar$, 
	\begin{equation}\label{E:spec1.Schroedinger}
		\spec_*(D_1^*D_1):\quad|\hbar|^3(2n+1)\left|(2n+1)^2-2\right|,\quad n=0,1,2,3,\dotsc
	\end{equation}

\subsection{The spectrum of $D_2^*D_2$}

	One readily checks the relations
	\begin{align}
		D_2&=ZJZ^*\star-X_3,\label{E:D2.ZJZs}
		\\
		Z^*\star Z&=-\tfrac12D_0D_0^*-JX_3,\label{E:Z*sZ}
		\\
		\star ZJA&=-ZJA\label{E:sZJA},
	\end{align}
	where
	\begin{equation}\label{E:star}
		\star^{-1}
		=\star^*
		=\star
		:=\begin{pmatrix}0&0&1\\0&-1&0\\1&0&0\end{pmatrix},
	\end{equation}
	and $Z$, $A$, $J$ are the operators considered in Section~\ref{SS:spec1.Schroedinger}, see \eqref{E:Z}, \eqref{E:A}, and \eqref{E:J}.

	Combining \eqref{E:Z*Z} and \eqref{E:D0D0s}, we also get
	\begin{equation}\label{E:J*Z*ZJ}
		J^*Z^*ZJ=\tfrac12A+D_0D_0^*+\tfrac32JX_3.
	\end{equation}
	From \eqref{E:D2.ZJZs} we obtain
	\begin{equation}\label{E:D2*D2}
		(\star D_2)^*=-(\star D_2),\qquad 
		D_2^*=-\star D_2\star,\qquad 
		D_2^*D_2=-(\star D_2)^2.
	\end{equation}

	The statement of the subsequent lemma is analogous to a result for contact 3-manifolds stated in \cite[Eq.~(63)]{BHR07}, see also \cite[Section~4.2]{RS12}.

	\begin{lemma}\label{L:specsD2}
		$\spec_*(\star D_2)=\spec(G)$, including multiplicities, where
		\begin{equation}\label{E:G}
			G:=\begin{pmatrix}-2X_3&\Delta/\sqrt2\\-\Delta/\sqrt2&-\tfrac12X_3\end{pmatrix}.
		\end{equation}
	\end{lemma}

	\begin{proof}
		Using \eqref{E:D1=ZA}, \eqref{E:AD0}, \eqref{E:Z*Z}, \eqref{E:AJA}, \eqref{E:D1*D1}, and \eqref{E:Z*sZ} one readily checks:
		\begin{align}
			D_1^*D_1=D_1^*(ZA)&=A(A^2+2X_{33}),\label{E:D1*ZA}
			\\
			D_1^*\star D_1=D_1^*(\star ZA)&=-AX_{33},\label{E:D1*sZA}
			\\
			D_1^*(ZJA)&=-A^2X_3\label{E:D1*ZJA}.
		\end{align}
		Furthermore, from \eqref{E:DD=0}, \eqref{E:AD0}, \eqref{E:Z*Z}, \eqref{E:D2.ZJZs}, \eqref{E:sZJA}, and \eqref{E:J*Z*ZJ} we obtain:
		\begin{align}
			\star D_2D_1=\star D_2(ZA)&=0\label{E:sD2ZA}
\\
			\star D_2(\star ZA)&=-(ZA)X_3-2(\star ZA)X_3-(ZJA)A\label{E:sD2sZA}
			\\
			\star D_2(ZJA)&=\tfrac12(\star ZA)A-\tfrac12(ZJA)X_3\label{E:sD2ZJA}
		\end{align}
		Put
		\begin{align}
			E&:=\star ZA(A^2+2X_{33})+(ZA)X_{33}=\star D_1(A^2+2X_{33})+D_1X_{33}\label{E:E},\\
			F&:=ZJA(A^2+2X_{33})+(ZA)AX_3\label{E:F}.
		\end{align}
		Using \eqref{E:Z*Z}, \eqref{E:AJA}, \eqref{E:Z*sZ}, \eqref{E:J*Z*ZJ}, \eqref{E:D1*ZA}, \eqref{E:D1*sZA}, \eqref{E:D1*ZJA}, and the fact that $X_3$ is central, one readily checks
		\begin{equation}\label{E:D1*EF}
			D_1^*E=0=D_1^*F,
		\end{equation}
		as well as
		\begin{align*}
			E^*E&=A(A^2+X_{33})(A^2+2X_{33})(A^2+3X_{33}),\\
			F^*F&=A(A^2+X_{33})(A^2+2X_{33})(\tfrac12A^2+3X_{33}),\\
			-F^*E=E^*F&=A(A^2+X_{33})(A^2+2X_{33})AX_3.
		\end{align*}
		Moreover, using \eqref{E:sD2ZA}, \eqref{E:sD2sZA}, and \eqref{E:sD2ZJA}, we get
		\begin{align*}
			\star D_2E&=-2EX_3-FA,\\
			\star D_2F&=\tfrac12EA-\tfrac12FX_3.
		\end{align*}
		Hence, 
		\begin{equation}\label{E:EF*EF}(E,F)^*(E,F)=A(A^2+X_{33})(A^2+2X_{33})\begin{pmatrix}A^2+3X_{33}&AX_3\\-AX_3&\tfrac12A^2+3X_{33}\end{pmatrix}\end{equation}
		and 
		\begin{equation}\label{E:sD2EF}
			\star D_2(E,F)=(E,F)\begin{pmatrix}-2X_3&\tfrac12A\\-A&-\tfrac12X_3\end{pmatrix}.
		\end{equation}
		Putting 
		\begin{equation}\label{E:C}
			C:=(EB,\sqrt2FB),
		\end{equation} 
		where $B$ is defined in \eqref{E:B}, we obtain from \eqref{E:G}, \eqref{E:D1*EF}, \eqref{E:EF*EF}, \eqref{E:sD2EF}, \eqref{E:AB=BD}, \eqref{E:B*B}
		\begin{align}
			D_1^*C&=0,\label{E:D1*C}\\
			\star D_2C&=CG,\label{E:sD2C}
		\end{align}
		and
		\begin{equation}\label{E:C*C}
			C^*C=\Delta^2(\Delta^2+X_{33})^2(\Delta^2+2X_{33})\begin{pmatrix}\Delta^2+3X_{33}&\sqrt2\Delta X_3\\-\sqrt2\Delta X_3&\Delta^2+6X_{33}\end{pmatrix}.
		\end{equation}
		Because of \eqref{E:sD2ZA} and \eqref{E:sD2C} the decomposition 
		\[
			L^2(\R)\otimes\C^2=\ker C\oplus\ker^\perp C
		\]
		is invariant under $G=-G^*$, the decomposition 
		\[
			L^2(\R)\otimes\C^3
			=\cimg D_1\oplus\ker D_1^*
			=\cimg D_1\oplus\cimg C\oplus(\ker D_1^*\cap\ker C^*)
		\]
		is invariant under $\star D_2=-(\star D_2)^*$, and we have:
		\begin{align}
			\spec(G)			&=\spec(G|_{\ker C})\sqcup\spec(G|_{\ker^\perp C}),\label{E:spec2.1}\\
			\spec_*(\star D_2)		&=\spec_*(\star D_2|_{\ker D_1^*}),\\
			\spec(\star D_2|_{\ker D_1^*})	&=\spec(\star D_2|_{\cimg C})\sqcup\spec(\star D_2|_{\ker D_1^*\cap\ker C^*}),\\
			\spec(\star D_2|_{\cimg C})	&=\spec(G|_{\ker^\perp C}).\label{E:spec2.4}
		\end{align}
		It remains to show 
		\begin{equation}\label{E:spec2.5}
			\spec(\star D_2|_{\ker D_1^*\cap\ker C^*})=\spec(G|_{\ker C}).
		\end{equation}
		Indeed, combining \eqref{E:spec2.1}--\eqref{E:spec2.5}, we obtain 
		\[
			\spec_*(\star D_2)=\spec_*(G)=\spec(G).
		\] 
		Here $\ker G=0$ as $(-2X_3)(-\tfrac12X_3)-(\Delta/\sqrt2)(-\Delta/\sqrt2)=\tfrac12(\Delta^2+2X_{33})$ has trivial kernel, cf.~\eqref{E:G}, \eqref{E:spec0.Schroedinger}, and \eqref{E:X3.Schroedinger}.

		In order to prove \eqref{E:spec2.5}, we assume w.l.o.g.\ $\ii X_3<0$.
		Combining \eqref{E:C*C} with 
		\[
			(\Delta^2+3X_{33})(\Delta^2+6X_{33})-(-\sqrt2\Delta X_3)(\sqrt2\Delta X_3)=(\Delta^2+2X_{33})(\Delta^2+9X_{33}),
		\] 
		and $\ker\Delta=0=\ker(\Delta^2+2X_{33})$, we conclude that 
		\[
			\begin{pmatrix}\psi_0\\0\end{pmatrix},\quad\begin{pmatrix}0\\\psi_0\end{pmatrix},\quad\begin{pmatrix}\psi_1\\\sqrt2\ii\psi_1\end{pmatrix}
		\] 
		is a basis of $\ker C$, where 
		\begin{equation}\label{E:psi1}
			0\neq\psi_0\in\ker(\Delta^2+X_{33})=\ker(\Delta+\ii X_3)
		\end{equation} 
		and 
		\[
			0\neq\psi_1\in\ker(\Delta^2+9X_{33})=\ker(\Delta+3\ii X_{3}).
		\]
		With respect to this basis, the action of $G$ on $\ker C$ is represented by the matrix 
		\begin{equation}\label{E:GkerC}
			\begin{pmatrix}-2&-\ii/\sqrt2&0\\\ii/\sqrt2&-\tfrac12&0\\0&0&1\end{pmatrix}X_3.
		\end{equation}

		In view of \eqref{E:D1=ZA}, \eqref{E:C}, \eqref{E:E}, \eqref{E:F}, \eqref{E:AB=BD}, and $\ker\Delta(\Delta^2+2X_{33})=0$, we have
		\begin{equation}\label{E:kerD1C.1}
			\ker D_1^*\cap\ker C^*=\ker(ZA)^*\cap\ker(ZJB)^*\cap\ker(\star ZB)^*.
		\end{equation}
		Using \eqref{E:B}, \eqref{E:spec3}, \eqref{E:spec4}, and \eqref{E:spec5}, we get $\ker B^*\supseteq\ker A=\cimg D_0$.
		The latter equality can alternatively be obtain from the obvious inclusion $\ker A\subseteq\ker D_1$, cf.~\eqref{E:D1=ZA}, and the exactness of the Rumin complex.
		Combining this with \eqref{E:kerD0B}, we obtain the orthogonal decomposition 
		\[
			\ker B^*=\ker A\oplus(\ker D_0^*\cap\ker D_0^*J).
		\]
		As the right summand is invariant under $J$, this yields
		\[
			\ker B^*\cap\ker AJ=\ker A\cap\ker AJ=\ker A\cap\ker B^*J.
		\]
		Hence,
		\[
			\ker(ZA)^*\cap\ker(ZJA)^*=\ker(ZB)^*\cap\ker(ZJA)^*=\ker(ZA)^*\cap\ker(ZJB)^*
		\]
		and
		\[
			\ker(\star ZA)^*\cap\ker(\star ZJA)^*=\ker(\star ZB)^*\cap\ker(\star ZJA)^*.
		\]
		Combining the latter two with \eqref{E:kerD1C.1} and \eqref{E:sZJA}, we obtain
		\begin{equation}\label{E:kerD1C.2}
			\ker D_1^*\cap\ker C^*=\ker(ZA)^*\cap\ker(ZJA)^*\cap\ker(\star ZA)^*\cap\ker(\star ZJA)^*.
		\end{equation}

		Since the subspace in \eqref{E:kerD1C.2} is invariant under $\star$, we decompose 
		\[
			L^2(\R)\otimes\C^3=\Omega^2_+\oplus\Omega^2_-
		\] 
		where $\Omega^2_+=\ker(\star-1)$ and $\Omega^2_-=\ker(\star+1)$ denote the eigenspaces of $\star$.
		Accordingly, we decompose 
		\begin{equation}\label{E:ZWtW}
			Z=W+\tilde W
		\end{equation} 
		with
		\begin{equation}\label{E:W}
			W=\tfrac12(Z+\star Z)=\frac12\begin{pmatrix}X_2&-X_1\\0&0\\X_2&-X_1\end{pmatrix}=\frac12\begin{pmatrix}1\\0\\1\end{pmatrix}D_0^*J
		\end{equation}
		and
		\begin{equation}\label{E:tW}
			\tilde W=\tfrac12(Z-\star Z)=\frac12\begin{pmatrix}-X_2&-X_1\\\sqrt2X_1&-\sqrt2X_2\\X_2&X_1\end{pmatrix}=\frac12\begin{pmatrix}1&0\\0&\sqrt2\\-1&0\end{pmatrix}\begin{pmatrix}-X_2&-X_1\\X_1&-X_2\end{pmatrix}.
		\end{equation}
		From \eqref{E:kerD1C.2} we thus get 
		\begin{multline}\label{E:kerD1C.3}
			\ker D_1^*\cap\ker C^*=\Bigl(\Omega^2_+\cap\ker(WA)^*\cap\ker(WJA)^*\Bigr)\\\oplus\Bigl(\Omega^2_-\cap\ker(\tilde WA)^*\cap\ker(\tilde WJA)^*\Bigr).
		\end{multline}

		One readily checks using \eqref{E:A}
		\[
			W^*W=\tfrac12(A+JX_3),\qquad\tilde W^*\tilde W=\tfrac12(\Delta+JX_3),\qquad W^*\tilde W=0=\tilde W^*W.
		\]
		as well as
		\begin{equation}\label{E:WW*}
			WW^*\begin{pmatrix}1\\0\\1\end{pmatrix}
			=\begin{pmatrix}1\\0\\1\end{pmatrix}\tfrac12\Delta
		\end{equation}
		and
		\begin{equation}\label{E:tWtW*}
			\tilde W\tilde W^*\begin{pmatrix}1&0\\0&\sqrt2\\-1&0\end{pmatrix}
			=\begin{pmatrix}1&0\\0&\sqrt2\\-1&0\end{pmatrix}\tfrac12(\Delta-JX_3).
		\end{equation}
		From \eqref{E:D0Delta}, \eqref{E:D0D0s}, \eqref{E:A} 
		\[
			A+JAJ^*=\Delta-3JX_3.
		\]
		Squaring this relation, and using \eqref{E:AJA} we get
		\begin{equation}\label{E:AA+JAAJ}
			A^2+JA^2J^*=(\Delta-2JX_3)(\Delta-3JX_3).
		\end{equation}
		Combining this with the commutator relations 
		\[
			W\Delta-\Delta W=2WJX_3,\qquad\tilde W\Delta-\Delta\tilde W=-2\tilde WJX_3
		\]
		and the relations 
		\[
			\Delta WJW^*=-X_3WW^*,\qquad J\tilde W^*\begin{pmatrix}1&0\\0&\sqrt2\\-1&0\end{pmatrix}=\tilde W^*\begin{pmatrix}1&0\\0&\sqrt2\\-1&0\end{pmatrix}J
		\]
		and \eqref{E:WW*}, \eqref{E:tWtW*} one readily derives
		\begin{equation}\label{E:WAWJA}
			\Bigl(WA(WA)^*+WJA(WJA)^*\Bigr)\begin{pmatrix}1\\0\\1\end{pmatrix}
			=\begin{pmatrix}1\\0\\1\end{pmatrix}\tfrac12\Delta(\Delta^2+X_{33})
		\end{equation}
		and
		\begin{multline}\label{E:tWAtWJA}
			\Bigl(\tilde WA(\tilde WA)^*+\tilde WJA(\tilde WJA)^*\Bigr)\begin{pmatrix}1&0\\0&\sqrt2\\-1&0\end{pmatrix}
			\\
			=\begin{pmatrix}1&0\\0&\sqrt2\\-1&0\end{pmatrix}\tfrac12(\Delta-JX_3)(\Delta-4JX_3)(\Delta-5JX_3).
		\end{multline}
		Furthermore, for each scalar $\lambda$
		\begin{equation}\label{E:DJX3.diag}
			(\Delta-\lambda JX_3)\begin{pmatrix}1&1\\\ii&-\ii\end{pmatrix}
			=\begin{pmatrix}1&1\\\ii&-\ii\end{pmatrix}\begin{pmatrix}\Delta-\lambda\ii X_3&0\\0&\Delta+\lambda\ii X_3\end{pmatrix}.
		\end{equation}

		Consider $\psi_0$ as in \eqref{E:psi1} above, choose 
		\[
			0\neq\psi_2\in\ker(\Delta+5\ii X_3),
		\] 
		and put
		\[
			b_1=\begin{pmatrix}\psi_0\\0\\\psi_0\end{pmatrix},\qquad 
			b_2=\begin{pmatrix}\psi_2\\-\sqrt2\ii\psi_2\\-\psi_2\end{pmatrix},\qquad 
			b_3=\begin{pmatrix}\psi_0\\-\sqrt2\ii\psi_0\\-\psi_0\end{pmatrix}.
		\]
		From \eqref{E:WAWJA}, \eqref{E:spec0.Schroedinger}, and \eqref{E:X3.Schroedinger} we see that $b_1$ is a basis of the first summand in \eqref{E:kerD1C.3}.
		Using \eqref{E:tWAtWJA} and \eqref{E:DJX3.diag} we see that $b_2,b_3$ is a basis of the second summand in \eqref{E:kerD1C.3}.
		In view of \eqref{E:WW*} and \eqref{E:tWtW*}, we have 
		\begin{equation}\label{E:bi.1}
			WW^*b_1=-\tfrac\ii2X_3b_1,\qquad
			\tilde W\tilde W^*b_2=-2\ii X_3b_2,\qquad
			\tilde W\tilde W^*b_3=0.
		\end{equation}
		Using \eqref{E:kerD1C.3}, \eqref{E:AA+JAAJ}, \eqref{E:DJX3.diag}, and \eqref{E:spec0.Schroedinger}, we conclude that $W^*b_1$ and $\tilde W^*b_2$ are two nontrivial elements in the 1-dimensional subspace
		\begin{equation}\label{E:kerAAJ}
			\ker A\cap\ker AJ=\ker(\Delta-3JX_3).
		\end{equation}
		Rescaling $\psi_2$ we may w.l.o.g.\ assume $2\ii W^*b_1=\tilde W^*b_2$.
		Hence, 
		\begin{equation}\label{E:bi.2}
			\tilde WW^*b_1=-X_3b_2,\qquad W\tilde W^*b_2=X_3b_1,\qquad W\tilde W^*b_3=0.
		\end{equation}
		Using \eqref{E:DJX3.diag} we see that $J=-\ii$ on the subspace in \eqref{E:kerAAJ}.
		Combining this with \eqref{E:D2.ZJZs}, \eqref{E:ZWtW}, \eqref{E:bi.1} and \eqref{E:bi.2} we see that, with respect to the basis $b_1,b_2,b_3$, the action of $\star D_2$ on $\ker D_1^*\cap\ker C^*$ is represented by the matrix 
		\begin{equation*}
			\begin{pmatrix}-\tfrac32&\ii&0\\-\ii&-1&0\\0&0&1\end{pmatrix}X_3.
		\end{equation*}
		Since this matrix has the same spectrum as the matrix in \eqref{E:GkerC}, we obtain the equality \eqref{E:spec2.5}, whence the lemma.
	\end{proof}

	\begin{lemma}\label{L:spec.G}
		The spectrum of $G$, see \eqref{E:G}, in the Schr\"odinger representation $\rho_\hbar$ on $L^2(\R)$ consists of the eigenvalues
		\begin{equation}\label{E:spec.G}
			\mathbf i\hbar\cdot\frac{-5\pm\sqrt{8(2n+1)^2+9}}4,\qquad n=0,1,2,3,\dotsc
		\end{equation}
		which all have multiplicity one.
	\end{lemma}

	\begin{proof}
		Choose $0\neq\psi_n\in L^2(\R)$ such that $\Delta\psi_n=|\hbar|(2n+1)\psi_n$, cf.\ \eqref{E:spec0.Schroedinger}.
		The subspace spanned by $\begin{pmatrix}\psi_n\\0\end{pmatrix}$ and $\begin{pmatrix}0\\\psi_n\end{pmatrix}$ is invariant under $G$, and with respect to this basis, the action of $G$ is represented by the matrix
		\[
			\mathbf i
			\begin{pmatrix}
				-2\hbar&-\mathbf i|\hbar|(2n+1)/\sqrt2\\
				\mathbf i|\hbar|(2n+1)/\sqrt2&-\hbar/2
			\end{pmatrix},
		\]
		cf.~\eqref{E:X3.Schroedinger} and \eqref{E:G}.
		The two eigenvalues of this matrix are given in \eqref{E:spec.G}.
		As the sum of these 2-dimensional subspaces is dense in $L^2(\R)\otimes\C^2$, these eigenvalues exhaust the full spectrum of $G$.
	\end{proof}

	Combining \eqref{E:D2*D2} with Lemma~\ref{L:specsD2} and Lemma~\ref{L:spec.G}, we obtain the spectrum of $D_2^*D_2$ in the Schr\"o\-din\-ger representation $\rho_\hbar$,
	\begin{equation}\label{E:spec2.Schroedinger}
		\spec_*(D_2^*D_2):\quad
		\lambda_n^\pm=\hbar^2\left(\frac{\sqrt{8(2n+1)^2+9}\pm5}4\right)^2,\quad n=0,1,2,3\dotsc
	\end{equation}

	Incidentally, the spectrum of $D_2^*D_2$ is simple, i.e., the eigenvalues $\lambda^\pm_n$ are mutually distinct.
	Indeed, one readily checks 
	\[
		5<\sqrt{209}-9\leq\sqrt{8(2n+3)^2+9}-\sqrt{8(2n+1)^2+9}\leq8/\sqrt2\leq6
	\] 
	for all $n\geq1$ as this difference is monotone in $n$.

\subsection{The determinants}\label{SS:dets.Schroedinger}

	Having determined the spectrum of $D_q^*D_q$ explicitly, we will now evaluate the zeta function 
	\begin{equation}\label{E:zeta.schroe}
		\zeta_{D_q^*D_q}(s)=\trs(D_q^*D_q)^{-s}
	\end{equation} 
	and its derivative at $s=0$.

	\begin{proposition}\label{P:zeta.Schroedinger}
		In the Schr\"odinger representation $\rho_\hbar$ on $L^2(\R,d\theta)$ we have:
		\begin{align}
			\zeta_{D_0^*D_0}(0)&=0,&\exp\bigl(-\zeta'_{D_0^*D_0}(0)\bigr)&=\sqrt2,\label{E:zetaD0}\\
			\zeta_{D_1^*D_1}(0)&=0,&\exp\bigl(-\zeta'_{D_1^*D_1}(0)\bigr)&=\sqrt8\cdot\sin\left(\pi\cdot\tfrac{\sqrt2-1}2\right),\label{E:zetaD1}\\
			\zeta_{D_2^*D_2}(0)&=0,&\exp\bigl(-\zeta'_{D_2^*D_2}(0)\bigr)&=4\cdot\sin^2\left(\pi\cdot\tfrac{\sqrt2-1}2\right).\label{E:zetaD2}
		\end{align}
	\end{proposition}

	\begin{proof}
		Using \eqref{E:spec0.Schroedinger}, we have
		\begin{equation}\label{E:zeta0.Schroedinger}
			\zeta_{D_0^*D_0}(s)
			=|\hbar|^{-s}\sum_{n=0}^\infty(2n+1)^{-s}
			=|\hbar|^{-s}\cdot\bigl(1-2^{-s}\bigr)\cdot\zeta_\Riem(s)
		\end{equation}
		where $\zeta_\Riem(s)=\sum_{n=1}^\infty n^{-s}$ denotes the Riemann zeta function.
		Hence, \eqref{E:zetaD0} follows from the classical fact $\zeta_\Riem(0)=-\frac12$.

		Using \eqref{E:spec1.Schroedinger}, writing 
		\[
			(2n+1)\bigl|(2n+1)^2-2\bigr|
			=\begin{cases}
				2^3\bigl(n+\tfrac12\bigr)\bigl(n+\tfrac{1+\sqrt2}2\bigr)\bigl(n+\tfrac{1-\sqrt2}2\bigr)&\text{for $n\geq1$,}
				\\
				1&\text{for $n=0$,}
			\end{cases}
		\]
		and proceeding as in \cite[Theorem~1]{M06}, we find:
		\begin{align*}
			\zeta_{D_1^*D_1}(s)
			&=(2|\hbar|)^{-3s}\left(8^s+\sum_{n=1}^\infty\left(n+\tfrac12\right)^{-s}\left(n+\tfrac{1+\sqrt2}2\right)^{-s}\left(n+\tfrac{1-\sqrt2}2\right)^{-s}\right)
			\\
			&=(2|\hbar|)^{-3s}\Bigg\{8^s+\zeta_\Riem(3s)-\frac{3s}2\zeta_\Riem(3s+1)\\
			&\qquad+\sum_{n=1}^\infty n^{-3s}\underbrace{\left(\left(1+\tfrac1{2n}\right)^{-s}\left(1+\tfrac{1+\sqrt2}{2n}\right)^{-s}\left(1+\tfrac{1-\sqrt2}{2n}\right)^{-s}-1+\frac{3s}{2n}\right)}_{O(n^{-2})}\Bigg\}.
		\end{align*}
		For the latter estimate recall that 
		\begin{equation}\label{E:binom}
			(1+z)^{-s}=1-sz+O(z^2),\qquad\text{as $z\to0$,}
		\end{equation} 
		uniformly for $s$ in compact subsets.
		Hence, the sum above converges normally for $\Re s>-1/3$.
		Since $\zeta_\Riem(0)=-\frac12$ and 
		\begin{equation}\label{E:zeta.s=1}
			\zeta_\Riem(s+1)=\frac1s+\gamma+O(s),\qquad\text{as $s\to0$,}
		\end{equation} 
		this yields $\zeta_{D_1^*D_1}(0)=0$, the first equality in \eqref{E:zetaD1}.
		Using Weierstra{\ss}'s representation of the Gamma function, 
		\begin{equation}\label{E:weier}
			\frac1{\Gamma(z)}=ze^{\gamma z}\prod_{n=1}^\infty\left(1+\frac zn\right)e^{-z/n},
		\end{equation} 
		and $\zeta_\Riem'(0)=-\log\sqrt{2\pi}$, this also gives
		\[
			\exp\bigl(-\zeta_{D_1^*D_1}'(0)\bigr)
			=\frac{-(2\pi)^{3/2}}{\Gamma\bigl(\tfrac12\bigr)\Gamma\bigl(\frac{1+\sqrt2}2\bigr)\Gamma\bigl(\frac{1-\sqrt2}2\bigr)}.
		\]
		Using $\Gamma\bigl(\frac12\bigr)=\sqrt\pi$, Euler's reflection formula 
		\begin{equation}\label{E:Euler.refl}
			\Gamma(z)\Gamma(1-z)=\frac\pi{\sin(\pi z)},
		\end{equation}
		and $-\sin(z)=\sin(z-\pi)$, we obtain the second equality in \eqref{E:zetaD1}.

		Recall the spectrum of $D_2^*D_2$ from \eqref{E:spec2.Schroedinger}.
		We start by comparing 
		\[
			\zeta_{D_2^*D_2}(s)=\sum_{n=0}^\infty\bigl((\lambda_n^+)^{-s}+(\lambda_n^-)^{-s}\bigr)
		\]
		with 
		\[
			\xi(s):=\sum_{n=0}^\infty\bigl(\lambda_n^+\lambda_n^-\bigr)^{-s}.
		\]
		To this end, write $\lambda_n^\pm=q_n\pm p_n$ with 
		\[
			q_n=\hbar^2\cdot\frac{8(2n+1)^2+34}{16},\qquad p_n=\hbar^2\cdot\frac{10\sqrt{8(2n+1)^2+9}}{16}.
		\]
		Note that $q_n>p_n$ for all $n$.
		Then
		\begin{multline*}
			\sum_{n=0}^\infty\Bigl((\lambda_n^+)^{-s}+(\lambda_n^-)^{-s}-2(\lambda_n^+\lambda_n^-)^{-s/2}\Bigr)
			\\
			=\sum_{n=0}^\infty q_n^{-s}\underbrace{\left(\left(1+\tfrac{p_n}{q_n}\right)^{-s}+\left(1-\tfrac{p_n}{q_n}\right)^{-s}-2\left(1-\tfrac{p_n^2}{q_n^2}\right)^{-s/2}\right)}_{O(n^{-2})}
		\end{multline*}
		where the estimate follows from \eqref{E:binom} and $\frac{p_n}{q_n}=O(n^{-1})$.
		Hence, the sum on the right hand side converges normally for $\Re s>-\frac12$.
		We conclude
		\begin{equation}\label{E:zetaD2zeta}
			\zeta_{D_2^*D_2}(0)=2\xi(0),\qquad\zeta'_{D_2^*D_2}(0)=\xi'(0).
		\end{equation}
		Hence, it remains to study $\xi(s)$.
		As 
		\[
			\lambda_n^+\lambda_n^-
			=|\hbar|^4\cdot4\cdot\left(\bigl(n+\tfrac12\bigr)^2-\tfrac12\right)^2
			=|\hbar|^4\cdot4\cdot\left(n+\tfrac{1+\sqrt2}2\right)^2\left(n+\tfrac{1-\sqrt2}2\right)^2,
		\] 
		we proceed as before and obtain:
		\begin{align*}
			\xi(s)
			&=|\hbar|^{-4s}\cdot4^{-s}\cdot\left(16^s+\sum_{n=1}^\infty\left(n+\tfrac{1+\sqrt2}2\right)^{-2s}\left(n+\tfrac{1-\sqrt2}2\right)^{-2s}\right)
			\\
			&=|\hbar|^{-4s}\cdot 4^{-s}\cdot\Bigg\{16^s+\zeta_\Riem(4s)-2s\zeta_\Riem(4s+1)\\
			&\qquad+\sum_{n=1}^\infty n^{-4s}\underbrace{\left(\left(1+\tfrac{1+\sqrt2}{2n}\right)^{-2s}\left(1+\tfrac{1-\sqrt2}{2n}\right)^{-2s}-1+\frac{2s}n\right)}_{O(n^{-2})}\Bigg\}
		\end{align*}
		Hence, the sum on the right hand side converges normally for $\Re s>-1/4$.
		Using $\zeta_\Riem(0)=-\frac12$ and \eqref{E:zeta.s=1} this yields $\xi(0)=0$, whence the first equality in \eqref{E:zetaD2}, see also \eqref{E:zetaD2zeta}.
		Using $\zeta_\Riem'(0)=-\log\sqrt{2\pi}$, \eqref{E:weier} and \eqref{E:Euler.refl}, we also obtain, cf.~\cite[Theorem~1]{M06},
		\[
			\exp\bigl(-\xi'(0)\bigr)
			=\frac{(2\pi)^2}{\Gamma^2\bigl(\frac{1+\sqrt2}2\bigr)\Gamma^2\bigl(\frac{1-\sqrt2}2\bigr)}
			=4\sin^2\left(\pi\tfrac{1+\sqrt2}2\right)
			=4\sin^2\left(\pi\tfrac{\sqrt2-1}2\right).
		\]
		Combining this with \eqref{E:zetaD2zeta}, we get the second equality in \eqref{E:zetaD2}.
	\end{proof}

	So far we have assumed $a_g=1$ and $b_g=g$ in this section.
	Now suppose $g$ is any graded Euclidean inner product on $\goe$ such that $b_g$ is proportional to $g$.
	Then each matrix in \eqref{E:hq} is proportional to the identical matrix, and hence $D_q^{*_{g_h}}D_q=c_qD_q^*D_q$ for some positive constant $c_q$.
	Therefore, cf.~\eqref{E:zeta.|rho.D|} and \eqref{E:zeta.schroe},
	\[
		\zeta_{|\rho_\hbar(D_q)|_{h_g}}(s)=c_q^{-s/2}\zeta_{D_q^*D_q}(s/2).
	\]
	As $\zeta_{D_q^*D_q}(0)$ vanishes, we have 
	\[
		\zeta_{|\rho_\hbar(D_q)|_{h_g}}'(0)
		=\tfrac12\zeta'_{D_q^*D_q}(0).
	\]
	Theorem~\ref{T:dets.Schroedinger} now follows from Proposition~\ref{P:zeta.Schroedinger}, \eqref{E:det.def}, and \eqref{E:Hodge.det}.

\section{Asymptotic expansion of the heat trace}\label{S:heat.asymp}

	The purpose of this section is to work out the structure of the heat trace asymptotics for general left invariant positive Rockland operators \cite[Chapter~4]{FR16} in irreducible unitary representations of $G$, see Theorem~\ref{T:asymp}.
	Subsequently, we will use this to derive Theorem~\ref{T:zeta}.

	\begin{theorem}\label{T:asymp}
		Let $V$ finite dimensional complex vector space equipped with a Hermitian inner product, and suppose $A=A^*\in\mathcal U^{-2\kappa}(\goe)\otimes\eend(V)$ is a left invariant positive Rockland differential operator on the 5-dimensional Lie group $G$ which is homogeneous of order $2\kappa>0$ with respect to the grading automorphism.
		Then, for every $B\in\eend(V)$ the following hold true:
		\begin{enumerate}[(I)]
		\item	For all nontrivial irreducible unitary representations $\rho$ of $G$ and all $M$,
			\begin{equation}\label{E:asymp.infty}
				\tr\left(Be^{-t\rho(A)}\right)=O(t^{-M}),
			\end{equation}
                as $t\to\infty$.
		\item	In the Schr\"odinger representation $\rho_\hbar$, 
			\begin{equation}\label{E:asymp.Schroedinger}
				\tr\left(Be^{-t\rho_\hbar(A)}\right)
				\sim\sum_{j=0}^\infty t^{(j-1)/\kappa}d_j,
			\end{equation}
				as $t\to0$ with coefficients $d_j=d_{A,B,\hbar,j}$. 
		\item	In the generic representation $\rho_{\lambda,\mu,\nu}$, 
			\begin{equation}\label{E:asymp.gen}
				\tr\left(Be^{-t\rho_{\lambda,\mu,\nu}(A)}\right)
				\sim\sum_{j=0}^\infty t^{(j-3)/4\kappa}a_j,
			\end{equation}
			as $t\to0$. 
			Moreover, the coefficients $a_j=a_{A,B,\lambda,\mu,\nu,j}$ vanish for odd $j$.
		\end{enumerate}
	\end{theorem}

	\begin{proof}
		We consider $A$ as a left invariant differential operator, acting on $V$-valued functions on $G$, which is homogeneous of order $2\kappa$ with respect to the grading automorphism, i.e., $\Phi_\tau^*A=\tau^{2\kappa}A$ for $\tau\neq0$.
		As $A$ is a positive Rockland operator, it has a heat kernel $k_t\in\mathcal S(G)\otimes\eend(V)$, the $\eend(V)$-valued Schwartz functions on $G$, see \cite[Theorem~4.2.7]{FR16}. 
		Hence, 
		\begin{equation}\label{E:kt.A}
			\bigl(e^{-tA}f\bigr)(g)
			=\int_Gk_t(h^{-1}g)f(h)dh
			=\int_Gk_t(h^{-1})f(gh)dh
		\end{equation}
		for $f\in\mathcal S(G)\otimes V$ and $g\in G$.
		By homogeneity,
		\begin{equation}\label{E:kt.homog}
			k_{\tau^{2\kappa}t}(\Phi^{-1}_\tau(g))
			=\tau^{-10}k_t(g),\qquad\tau\neq0
		\end{equation}
		as the homogeneous dimension of $G$ is 10. 

		We first consider generic representations of $G$.
		To this end, suppose $\alpha\in\goe^*$ does not vanish on the center $\goe_{-3}$ of $\goe$.
		Then there exists a unique polarization algebra $\poe_\alpha$ for $\alpha$, i.e., a unique maximal subalgebra satisfying $\goe_\alpha\subseteq\poe_\alpha$ and $\alpha([\poe_\alpha,\poe_\alpha])=0$.
		Indeed, putting $a_i:=\alpha(X_i)$ we have $(a_4,a_5)\neq(0,0)$, the isotropy algebra $\goe_\alpha$ is spanned by $a_5X_1-a_4X_2+a_3X_3,X_4,X_5$, and $\poe_\alpha$ is spanned by $a_5X_1-a_4X_2,X_3,X_4,X_5$.
		Note that $\poe_\alpha$ is an ideal of dimension 4 which is invariant under the grading automorphism $\phi_\tau$.
		The corresponding connected Lie subgroup of $G$ will be denoted by $P_\alpha=\exp(\poe_\alpha)$.

		The irreducible unitary representation $\pi_\alpha$, corresponding to $\alpha$ via Kirillov's orbit method \cite{K62,K04,CG90} acts on the Hilbert space 
		\[
			\mathcal H_\alpha
			=\left\{f\colon G\to\C\middle|\begin{array}{c}\forall p\in P_\alpha,g\in G:f(pg)
			=e^{\mathbf i\alpha(\log p)}f(g)\\|f|\in L^2(P_\alpha\setminus G)\end{array}\right\}
		\]
		by right translations, i.e., $(\pi_\alpha(h)f)(g)=f(gh)$ for $f\in\mathcal H_\alpha$ and $g,h,\in G$.
		It will be convenient to use the complementary vector 
		\begin{equation}\label{E:Zalpha}
			Z_\alpha=\frac1{\sqrt{a_4^2+a_5^2}}\bigl(a_4X_1+a_5X_2\bigr)
		\end{equation}
		to $\poe_\alpha$ in $\goe$ to identify the Hilbert space $\mathcal H_\alpha$ with $L^2(\R)=L^2(\R,dx)$ via $f\leftrightarrow\tilde f$ where $\tilde f(x)=f(\exp(xZ_\alpha))$ for $x\in\R$.
		As $P_\alpha$ is a normal in $G$, for $p\in P_\alpha$ and $y\in\R$, the action of $h=p\exp(yZ_\alpha)$ on $\tilde f\in L^2(\R)$ becomes, cf.~\cite[Eq.~(25)]{D58},
		\begin{equation}\label{E:rep.Kirr}
			\bigl(\pi_\alpha(h)\tilde f\bigr)(x)
			=\tilde f(x+y)e^{\ii\alpha(\Ad_{\exp(xZ_\alpha)}\log p)}.
		\end{equation}
		Combining this with \eqref{E:kt.A}, we conclude that $\pi_\alpha(A)$ has heat kernel
		\begin{multline}\label{E:tkat}
			\tilde k_{\alpha,t}(x,y)
			=\int_{P_\alpha}k_t\bigl(\exp(-yZ_\alpha)p\exp(xZ_\alpha)\bigr)e^{-\ii\alpha(\log p)}dp
			\\
			=\int_{P_\alpha}k_t\bigl(\exp((x-y)Z_\alpha)p\bigr)e^{-\ii\alpha(\Ad_{\exp(xZ_\alpha)}\log p)}dp,
		\end{multline}
		that is,
		\[
			\bigl(e^{-t\pi_\alpha(A)}\tilde f\bigr)(x)
			=\int_{\R}\tilde k_{\alpha,t}(x,y)\tilde f(y)dy
		\]
		for $\tilde f\in L^2(\R)\otimes V$ and $x\in\R$.
		As $k_t$ is in the Schwartz space, the operator $e^{-t\pi_\alpha(A)}$ is trace class on $L^2(\R)$ and 
		\begin{equation}\label{E:tr.pre}
			\tr\left(Be^{-t\pi_\alpha(A)}\right)
			=\int_{-\infty}^\infty\tr\bigl(B\tilde k_{t,\alpha}(x,x)\bigr)dx.
		\end{equation}

		Using the homogeneity in \eqref{E:kt.homog} with $\tau=t^{-1/2\kappa}$ and the fact that the homogeneous dimension of $\poe_\alpha$ is $9$, we obtain from \eqref{E:tkat}
		\[
			\tilde k_{\alpha,t}(x,x)
			=t^{-1/2\kappa}\int_{P_\alpha}k_1(p)e^{-\ii\alpha(\Ad_{\exp(xZ_\alpha)}\phi_{t^{-1/2\kappa}}\log p)}dp.
		\]
		We write this in the form 
		\[
			\tilde k_{\alpha,t}(x,x)
			=t^{-1/2\kappa}\cdot\hat k_\alpha\bigl(\alpha\circ\Ad_{\exp(xZ_\alpha)}\circ\phi_{t^{-1/2\kappa}}|_{\poe_\alpha}\bigr)
		\] 
		where $\hat k_\alpha\in\mathcal S(\poe_\alpha^*)\otimes\eend(V)$, 
		\begin{equation}\label{E:hatkalpha}
			\hat k_\alpha(\xi)
			=\int_{\poe_\alpha}k_1(\exp(X))e^{-\ii\xi(X)}dX,\qquad\xi\in\poe_\alpha^*
		\end{equation}
		denotes the Fourier transform of $k_1\circ\exp|_{\poe_\alpha}\in\mathcal S(\poe_\alpha)\otimes\eend(V)$.
		Hence, 
		\begin{equation}\label{E:traDq}
			\tr\left(Be^{-t\pi_\alpha(A)}\right)
			=t^{-1/2\kappa}\int_{-\infty}^\infty\tr\left(B\hat k_\alpha\bigl(\alpha\circ\Ad_{\exp(xZ_\alpha)}\circ\phi_{t^{-1/2\kappa}}|_{\poe_\alpha}\bigr)\right)dx.
		\end{equation}
		
		We identify $\poe_\alpha^*$ with $\R^4$ using the basis $Z_\alpha',X_3,X_4,X_5$ of $\poe_\alpha$ where
		\begin{equation}\label{E:Z'alpha}
			Z'_\alpha:=\frac1{\sqrt{a_4^2+a_5^2}}\bigl(-a_5X_1+a_4X_2\bigr).
		\end{equation}
		In other words, we are identifying via 
		\begin{equation}\label{E:ident}
			\poe_\alpha^*\cong\R^4,\qquad
			\xi\leftrightarrow\bigl(\xi(Z_\alpha'),\xi(X_3),\xi(X_4),\xi(X_5)\bigr).
		\end{equation}
		Using the commutator relations in \eqref{E:basis.Xi}, \eqref{E:Zalpha}, and \eqref{E:Z'alpha} we get $[Z_\alpha,Z_\alpha']=X_3$ and $[Z_\alpha,X_3]=\frac1{\sqrt{a_4^2+a_5^2}}(a_4X_4+a_5X_5)$.
		As $\Ad_{\exp(xZ_\alpha)}X=X+x[Z_\alpha,X]+\frac{x^2}2[Z_\alpha,[Z_\alpha,X]]$ and $\alpha(X_i)=a_i$, this yields
		\begin{equation}\label{E:qwerty}
			\alpha\circ\Ad_{\exp(xZ_\alpha)}\circ\phi_{t^{-1/2\kappa}}|_{\poe_\alpha}
			=\begin{pmatrix}
				t^{1/2\kappa}\frac{\nu+\bigl(a_3+x\sqrt{a_4^2+a_5^2}\bigr)^2}{2\sqrt{a_4^2+a_5^2}}\\
				t^{2/2\kappa}\bigl(a_3+x\sqrt{a_4^2+a_5^2}\bigr)\\
				t^{3/2\kappa}a_4\\
				t^{3/2\kappa}a_5
			\end{pmatrix}
		\end{equation}
		via the identification in \eqref{E:ident}, where 
		\begin{equation}\label{E:nu}
			\nu=-\bigl(a_3^2+2a_1a_5-2a_2a_4\bigr).
		\end{equation}
		Substituting $t^{1/4\kappa}\bigl(a_3+x\sqrt{a_4^2+a_5^2}\bigr)\leftrightarrow\sqrt[4]{a_4^2+a_5^2}x$ in the integral \eqref{E:traDq} yields 
		\footnote{Equation~\eqref{E:tr.eAt} can alternatively be obtained from a matrix version of Kirillov's character formula, cf.~\cite[Theorems~4.3.1 and 4.3.3]{CG90}.}
		\begin{equation}\label{E:tr.eAt}
			\tr\left(Be^{-t\pi_\alpha(A)}\right)
			=\frac{t^{-3/4\kappa}}{\sqrt[4]{a_4^2+a_5^2}}
			\int_{-\infty}^\infty
			\tr\Biggl(B\hat k_\alpha
			\begin{pmatrix}
				\frac{t^{2/4\kappa}\nu}{2\sqrt{a_4^2+a_5^2}}+\frac{x^2}2\\
				t^{3/4\kappa}\sqrt[4]{a_4^2+a_5^2}x\\
				t^{6/4\kappa}a_4\\
				t^{6/4\kappa}a_5
			\end{pmatrix}
			\Biggr)dx.
		\end{equation}

		By Taylor's theorem,
		\begin{align}
			\notag
			\hat k_\alpha
			&
			\begin{pmatrix}
				\frac{t^{2/4\kappa}\nu}{2\sqrt{a_4^2+a_5^2}}+\frac{x^2}2\\
				t^{3/4\kappa}\sqrt[4]{a_4^2+a_5^2}x\\
				t^{6/4\kappa}a_4\\
				t^{6/4\kappa}a_5
			\end{pmatrix}
			=\sum_{\substack{i_1,i_2,i_3,i_4\in\N_0\\i_1+i_2+i_3+i_4\leq N}}
			t^{(2i_1+3i_2+6i_3+6i_4)/4\kappa}
			\\\notag
			&\qquad\cdot(\nu/2)^{i_1}x^{i_2}(a_4^2+a_5^2)^{(i_2-2i_1)/4}a_4^{i_3}a_5^{i_4}
			\frac{\partial_1^{i_1}\partial_2^{i_2}\partial_3^{i_3}\partial_4^{i_4}\hat k_\alpha}{i_1!i_2!i_3!i_4!}\begin{pmatrix}x^2/2\\0\\0\\0\end{pmatrix}
			\\\notag
			&+\sum_{\substack{i_1,i_2,i_3,i_4\in\N_0\\i_1+i_2+i_3+i_4=N+1}}
			t^{(2i_1+3i_2+6i_3+6i_4)/4\kappa}(\nu/2)^{i_1}x^{i_2}(a_4^2+a_5^2)^{(i_2-2i_1)/4}a_4^{i_3}a_5^{i_4}
			\\\label{E:Taylor.00}
			&\qquad\cdot(N+1)\int_0^1(1-\xi)^N\frac{\partial_1^{i_1}\partial_2^{i_2}\partial_3^{i_3}\partial_4^{i_4}\hat k_\alpha}{i_1!i_2!i_3!i_4!}
			\begin{pmatrix}
				\xi\frac{t^{2/4\kappa}\nu}{2\sqrt{a_4^2+a_5^2}}+\frac{x^2}2\\
				\xi t^{3/4\kappa}\sqrt[4]{a_4^2+a_5^2}x\\
				\xi t^{6/4\kappa}a_4\\
				\xi t^{6/4\kappa}a_5
			\end{pmatrix}
			d\xi.
		\end{align}
		As $\hat k_\alpha$ is a Schwartz function, 
		\[
			x^{i_2}
			\frac{\partial_1^{i_1}\partial_2^{i_2}\partial_3^{i_3}\partial_4^{i_4}\hat k_\alpha}{i_1!i_2!i_3!i_4!}
				\begin{pmatrix}
					\xi\frac{t^{2/4\kappa}\nu}{2\sqrt{a_4^2+a_5^2}}+\frac{x^2}2\\
					\xi t^{3/4\kappa}\sqrt[4]{a_4^2+a_5^2}x\\
					\xi t^{6/4\kappa}a_4\\
					\xi t^{6/4\kappa}a_5
				\end{pmatrix}
			=O\left(\frac1{1+x^2}\right),
		\]
		uniformly for $0<t\leq1$, $x\in\R$, and $0\leq\xi\leq1$.
		Combining this estimate with \eqref{E:tr.eAt} and \eqref{E:Taylor.00}, we obtain  
		\begin{equation}\label{E:tr.pi.alpha}
			\tr\left(Be^{-t\pi_\alpha(A)}\right)
			=\sum_{j=0}^Nt^{(j-3)/4\kappa}\cdot a_{A,B,\alpha,j}+O\left(t^{(N-2)/4\kappa}\right)
		\end{equation}
		as $t\to0$ with
		\begin{multline}\label{E:aqbj}
			a_{A,B,\alpha,j}
			=\sum_{\substack{i_1,i_2,i_3,i_4\in\N_0\\2i_1+3i_2+6i_3+6i_4=j}}
			(\nu/2)^{i_1}(a_4^2+a_5^2)^{(i_2-2i_1-1)/4}a_4^{i_3}a_5^{i_4}
			\\
			\cdot\int_{-\infty}^\infty
			\tr\Biggl(B\frac{\partial_1^{i_1}\partial_2^{i_2}\partial_3^{i_3}\partial_4^{i_4}\hat k_\alpha}{i_1!i_2!i_3!i_4!}
			\begin{pmatrix}x^2/2\\0\\0\\0\end{pmatrix}
			\Biggr)x^{i_2}dx.
		\end{multline}
		The integral vanishes for odd $i_2$ as the integrand is an odd function of $x$.
		Hence, $a_{A,B,\alpha,j}$ vanishes for odd $j$.
		
		Furthermore, as $\hat k_\alpha$ is in the Schwartz space, for each $M$
		\[
			\hat k_\alpha
			\begin{pmatrix}
				\frac{t^{2/4\kappa}\nu}{2\sqrt{a_4^2+a_5^2}}+\frac{x^2}2\\
				t^{3/4\kappa}\sqrt[4]{a_4^2+a_5^2}x\\
				t^{6/4\kappa}a_4\\
				t^{6/4\kappa}a_5
			\end{pmatrix}
			=O\left(\frac{t^{-M}}{1+x^2}\right)
		\]
		uniformly for $x\in\R$ and $1\leq t<\infty$.
		Combining this estimate with \eqref{E:tr.eAt}, we get
		\begin{equation}\label{E:tr.pi.alpha.infty}
			\tr\left(Be^{-t\pi_\alpha(A)}\right)
			=O\left(t^{-M}\right)
		\end{equation}
		as $t\to\infty$, for each $M$.

		For the infinitesimal representation we obtain from \eqref{E:rep.Kirr} and \eqref{E:qwerty} 
		\[
			\pi_\alpha(Z_\alpha)=\partial_x,\qquad
			\pi_\alpha(Z_\alpha')=\frac{\mathbf i}{\sqrt{a_4^2+a_5^2}}\cdot\frac{\nu+\left(a_3+x\sqrt{a_4^2+a_5^2}\right)^2}2,
		\]
		and
		\[
			\pi_\alpha(X_3)=\mathbf i\left(a_3+x\textstyle\sqrt{a_4^2+a_5^2}\right),\qquad
			\pi_\alpha(X_4)=\mathbf ia_4,\qquad
			\pi_\alpha(X_5)=\mathbf ia_5.
		\]
		As $X_1=\frac1{\sqrt{a_4^2+a_5^2}}\bigl(a_4Z_\alpha-a_5Z_\alpha'\bigr)$ and $X_2=\frac1{\sqrt{a_4^2+a_5^2}}\bigl(a_5Z_\alpha+a_4Z_\alpha'\bigr)$, this also yields
		\begin{align*}
			\pi_\alpha(X_1)
			&=\frac{a_4}{\sqrt{a_4^2+a_5^2}}\partial_x-\frac{\mathbf ia_5}{a_4^2+a_5^2}\cdot\frac{\nu+\left(a_3+x\sqrt{a_4^2+a_5^2}\right)^2}2,
			\\
			\pi_\alpha(X_2)
			&=\frac{a_5}{\sqrt{a_4^2+a_5^2}}\partial_x+\frac{\mathbf ia_4}{a_4^2+a_5^2}\cdot\frac{\nu+\left(a_3+x\sqrt{a_4^2+a_5^2}\right)^2}2,
		\end{align*}
		and
		\[
			\pi_\alpha\bigl(X_3X_3+2X_1X_5-2X_2X_4\bigr)=\nu.
		\]
		Hence, $\pi_\alpha$ is unitarily equivalent to the representation $\rho_{\lambda,\mu,\nu}$ with $\lambda=a_4$, $\mu=a_5$ and $\nu$ given by \eqref{E:nu}, cf.~\eqref{E:rep.gen.X1}--\eqref{E:rep.gen.X5}.
		The unitary intertwiner is the one corresponding to the affine reparametrization $\theta\sqrt[3]{a_4^2+a_5^2}=a_3+x\sqrt{a_4^2+a_5^2}$.
		This shows parts (III) and (I) for generic representations, cf.~\eqref{E:tr.pi.alpha} and \eqref{E:tr.pi.alpha.infty}.

		Let us now turn to Schr\"odinger representations.
		To this end, suppose $\alpha\in\goe^*$ is such that $a_4=0=a_5$ and $a_3\neq0$.
		We now consider 
		\begin{equation}\label{E:ZZp}
			Z_\alpha:=z_1X_1+z_2X_2
			\qquad\text{and}\qquad
			Z'_\alpha:=-z_2X_1+z_1X_2
		\end{equation}
		where $z_1,z_2$ are any real numbers such that $z_1^2+z_2^2=1$.
		Then $Z'_\alpha,X_3,X_4,X_5$ span a polarization algebra $\poe_\alpha$ for $\alpha$, and equation \eqref{E:traDq} remains valid with $\hat k_\alpha$ as in \eqref{E:hatkalpha}.
		Using the aforementioned basis of $\poe_\alpha$ to identify $\poe^*_\alpha$ with $\R^4$, we find
		\begin{equation}\label{E:qwerty2}
			\alpha\circ\Ad_{\exp(xZ_\alpha)}\circ\phi_{t^{-1/2\kappa}}|_{\poe_\alpha}
			=\begin{pmatrix}
				t^{1/2\kappa}\bigl(-z_2a_1+z_1a_2+a_3x\bigr)\\
				t^{1/\kappa}a_3\\
				0\\
				0
			\end{pmatrix}.
		\end{equation}
		Substituting $t^{1/2\kappa}\bigl(-z_2a_1+z_1a_2+a_3x\bigr)\leftrightarrow x$ in the integral \eqref{E:traDq}, we obtain
		\begin{equation}\label{E:tr.eAt.h}
			\tr\left(Be^{-t\pi_\alpha(A)}\right)
			=\frac{t^{-1/\kappa}}{|a_3|}
			\int_{-\infty}^\infty
			\tr\Biggl(B\hat k_\alpha
			\begin{pmatrix}
				x\\
				t^{1/\kappa}a_3\\
				0\\
				0
			\end{pmatrix}
			\Biggr)dx.
		\end{equation}

		By Taylor's theorem,
		\begin{multline}\label{E:Taylor.00.h}
			\hat k_\alpha
			\begin{pmatrix}x\\t^{1/\kappa}a_3\\0\\0\end{pmatrix}
			=\sum_{j=0}^Nt^{j/\kappa}a_3^j
			\frac{\partial_2^j\hat k_\alpha}{j!}\begin{pmatrix}x\\0\\0\\0\end{pmatrix}
			\\
			+t^{(N+1)/\kappa}a_3^{N+1}
			\int_0^1(1-\xi)^N\frac{\partial_2^{N+1}\hat k_\alpha}{N!}
			\begin{pmatrix}x\\\xi t^{1/\kappa}a_3\\0\\0\end{pmatrix}
			d\xi.
		\end{multline}
		As $\hat k_\alpha$ is in the Schwartz space, 
		\[
			\partial_2^{N+1}\hat k_\alpha
			\begin{pmatrix}x\\\xi t^{1/\kappa}a_3\\0\\0\end{pmatrix}
			=O\left(\frac1{1+x^2}\right),
		\]
		uniformly for $0<t\leq1$, $x\in\R$, and $0\leq\xi\leq1$.
		Combining this estimate with \eqref{E:tr.eAt.h} and \eqref{E:Taylor.00.h}, we obtain  
		\begin{equation}\label{E:tr.pi.alpha.h}
			\tr\left(Be^{-t\pi_\alpha(A)}\right)
			=\sum_{j=0}^Nt^{(j-1)/\kappa}\cdot d_{A,B,\alpha,j}+O\left(t^{N/\kappa}\right)
		\end{equation}
		as $t\to0$ with
		\begin{equation}\label{E:dABalphaj}
			d_{A,B,\alpha,j}
			=\frac{a_3^j}{|a_3|}
			\int_{-\infty}^\infty
			\tr\Biggl(B\frac{\partial_2^j\hat k_\alpha}{j!}
			\begin{pmatrix}x\\0\\0\\0\end{pmatrix}
			\Biggr)dx.
		\end{equation}

		Furthermore, as $\hat k_\alpha$ is a Schwartz function, for each $M$
		\[
			\hat k_\alpha
			\begin{pmatrix}x\\t^{1/\kappa}a_3\\0\\0\end{pmatrix}
			=O\left(\frac{t^{-M}}{1+x^2}\right),
		\]
		uniformly for $x\in\R$ and $1\leq t<\infty$.
		Combining this estimate with \eqref{E:tr.eAt.h}, we get
		\begin{equation}\label{E:tr.pi.alpha.infty.h}
			\tr\left(Be^{-t\pi_\alpha(A)}\right)
			=O\left(t^{-M}\right)
		\end{equation}
		as $t\to\infty$, for each $M$.
	
		From \eqref{E:rep.Kirr} and \eqref{E:qwerty2} one readily derives $\pi_\alpha(X_3)=\mathbf ia_3$ and $\pi_\alpha(X_4)=0=\pi_\alpha(X_5)$.
		Hence, $\pi_\alpha$ is unitarily equivalent to the Schr\"odinger representation $\rho_\hbar$ with $\hbar=a_3$, cf.~\eqref{E:Schroedinger}.
		This shows parts (II) and (I) for Schr\"odinger representations, cf.~\eqref{E:tr.pi.alpha.h} and \eqref{E:tr.pi.alpha.infty.h}.

		It remains to check part (I) for nontrivial scalar representations.
		As these representations are 1-dimensional, the statement follows immediately from the fact that all eigenvalues of $\rho(A)$ are strictly positive.
	\end{proof}

	\begin{remark}
		If $\nu=0$, then $a_{A,B,\alpha,j}$ in \eqref{E:aqbj} vanishes unless $j$ is a multiple of $6$, and we obtain from \eqref{E:asymp.gen} an asymptotic expansion of the form
		\[
			\tr\left(e^{-t\rho_{\lambda,\mu,0}(A)}\right)
			\sim\sum_{j=0}^\infty t^{3(2j-1)/4\kappa}a_{6j}
		\]
		as $t\to0$, consistent with \cite[Eq.~(14)]{V80} and \cite[Eq.~(5.2)]{V92}.
	\end{remark}

	The homogeneity in \eqref{E:kt.homog} with $\tau=-1$ gives $k_t(\Phi_{-1}(g))=k_t(g)$ and therefore 
	\begin{equation}\label{E:homog-1}
		\hat k_\alpha\begin{pmatrix}-x_1\\x_2\\-x_3\\-x_4\end{pmatrix}
		=\hat k_\alpha\begin{pmatrix}x_1\\x_2\\x_3\\x_4\end{pmatrix}
	\end{equation}
	cf.~\eqref{E:hatkalpha} and \eqref{E:ident}.
	This symmetry will be used in the subsequent section.

	Let us close this section with a

	\begin{proof}[{Proof of Theorem~\ref{T:zeta}}]
		Let $\rho$ be a nontrivial irreducible unitary representation of $G$.
		Recall that the Rumin--Seshadri operator $\Delta_{h,q}$ in \eqref{E:Rumin.Seshadri} is a positive Rockland operator, cf.~\cite[Lemma~2.14]{DH22}.
		Applying Theorem~\ref{T:asymp} with $A=\Delta_{h,q}$ and $B=\id$, we see that
		\begin{equation}\label{E:Melin.Delta}
			\zeta_{\rho(\Delta_{h,q})}(s)
			=\tr\rho(\Delta_{h,q})^{-s}
			=\frac1{\Gamma(s)}\int_0^\infty t^{s-1}\tr\left(e^{-t\rho(\Delta_{h,q})}\right)dt
		\end{equation}
		converges for $\Re s$ sufficiently large, and that this zeta function extends to a meromorphic function on the entire complex plane which has only simple poles and is holomorphic at $s=0$.
		Proceeding recursively using \eqref{E:zeta.Delta}, we conclude that the zeta function $\zeta_{|\rho(D_q)|_h}(s)$ has the same properties, whence part (I).

		Suppose $\rho=\rho_\hbar$ is a Schr\"odinger representation.
		Using the expansion \eqref{E:asymp.Schroedinger}, we conclude that \eqref{E:Melin.Delta} converges for $\Re s>1/\kappa$ and that the poles of this zeta function can only be located at $s=(1-j)/\kappa$ with $1\neq j\in\N_0$.
		Proceeding recursively, using \eqref{E:zeta.Delta} and $a_qk_q=\kappa$, we obtain part (II).

		For generic representations $\rho=\rho_{\lambda,\mu,\nu}$ we may proceed analogously using the expansion \eqref{E:asymp.gen} to obtain part (III).
	\end{proof}

\section{Polyhomogeneous expansion of the heat trace}\label{S:poly.heat}

	In this section we consider a family of generic representations $\rho_{r\lambda,r\mu,\nu}$ of $G$ parametrized by $r>0$.
	For a positive left invariant homogeneous Rockland differential operator $A$ on $G$ and $\nu<0$ we will establish the structure of the polyhomogeneous expansion of the heat trace $\tr\left(e^{-t\rho_{r\lambda,r\mu,\nu}(A)}\right)$ as $(r,t)\to(0,0)$ and as $(r,t)\to(0,\infty)$.
	More precisely, we aim at proving the following result.

	\begin{theorem}\label{T:poly}
		Let $V$ finite dimensional complex vector space equipped with a Hermitian inner product, and suppose $A=A^*\in\mathcal U^{-2\kappa}(\goe)\otimes\eend(V)$ is a left invariant positive Rockland differential operator on the 5-dimensional Lie group $G$ which is homogeneous of order $2\kappa>0$ with respect to the grading automorphism.
		Let $\rho_{\lambda,\mu,\nu}$ be a generic representation of $G$ with $\nu<0$.
		Then:
		\begin{enumerate}[(a)]
			\item	$t^{3/4\kappa}\sqrt r\tr\left(e^{-t\rho_{r\lambda,r\mu,\nu}(A)}\right)$ extends smoothly in the variables $\bigl(\frac{t^{1/4\kappa}}{\sqrt r},r^2\bigr)\in[0,\infty)^2$.
			\item	$t^{1/\kappa}\tr\left(e^{-t\rho_{r\lambda,r\mu,\nu}(A)}\right)$ extends smoothly in the variables $\bigl(\frac r{t^{1/2\kappa}},t^{1/\kappa}\bigr)\in[0,\infty)^2$.
		\end{enumerate}
		There exist smooth functions $a_j(r)$, $c_k(y)$, $e_l(t)$ and constants $b_{j,k}$ and $d_{k,l}$ such that for any integer $N$, real numbers $T>0$, $\varepsilon>0$ and $M$ we have, with $y=\frac{t^{1/4\kappa}}{\sqrt r}$:
		\begin{align}
			\label{E:tr.a}
			t^{3/4\kappa}\sqrt r\tr\left(e^{-t\rho_{r\lambda,r\mu,\nu}(A)}\right)
			&=\textstyle\sum_{j=0}^Ny^jr^2a_j(r)
			+O\left(y^{N+1}\right)
                        \\
			\label{E:tr.b}
			t^{3/4\kappa}\sqrt r\tr\left(e^{-t\rho_{r\lambda,r\mu,\nu}(A)}\right)
			&=\textstyle\sum_{j+k\leq N}r^{2k}y^jb_{j,k}
			+O\left(\left(r^2+y\right)^{N+1}\right)
			\\
			\label{E:tr.c}
			t^{3/4\kappa}\sqrt r\tr\left(e^{-t\rho_{r\lambda,r\mu,\nu}(A)}\right)
			&=\textstyle\sum_{k=0}^Nr^{2k}y^3c_k(y^{4\kappa})
			+O\bigl(r^{2(N+1)}\bigr)
		\end{align}
		uniformly for $0<y\leq T$, $0<r\leq1$;
		\begin{align}
			\label{E:tr.tc}
			t^{1/\kappa}\tr\left(e^{-t\rho_{r\lambda,r\mu,\nu}(A)}\right)
			&=\textstyle\sum_{k=0}^Nt^{k/\kappa}y^{4-4k}c_k(y^{4\kappa})
			+O\bigl(t^{(N+1)/\kappa}\bigr)
			\\
			\label{E:tr.d}
			t^{1/\kappa}\tr\left(e^{-t\rho_{r\lambda,r\mu,\nu}(A)}\right)
			&=\textstyle\sum_{k+l\leq N}t^{k/\kappa}y^{-2l}d_{k,l}
			+O\left(\left(t^{1/\kappa}+y^{-2}\right)^{N+1}\right)
			\\
			\label{E:tr.e}
			t^{1/\kappa}\tr\left(e^{-t\rho_{r\lambda,r\mu,\nu}(A)}\right)
			&=\textstyle\sum_{l=0}^Ny^{-2l}t^{(l/2+1)/\kappa}e_l(t)
			+O\left(y^{-2(N+1)}\right)
		\end{align}
		uniformly for $\varepsilon\leq y<\infty$, $0<t\leq T$; and
		\begin{align}
			\label{E:tr.infty}
			\tr\left(e^{-t\rho_{r\lambda,r\mu,\nu}(A)}\right)
			&=\textstyle\sum_{l=0}^Nr^le_l(t)
			+O\left(r^{N+1}t^{-M}\right)
		\end{align}
		uniformly for $\varepsilon\leq t<\infty$, $0<r\leq1$.
		Furthermore,
		\begin{align}
			\label{E:a.b}
			r^2a_j(r)&\textstyle=\sum_{k=0}^Nr^{2k}b_{j,k}+O\left(r^{2(N+1)}\right)&&0<r\leq1
			\\
			\label{E:c.b}
			y^3c_k(y^{4\kappa})&\textstyle=\sum_{j=0}^Ny^jb_{j,k}+O\left(y^{N+1}\right)&&0<y\leq T
			\\
			\label{E:c.d}
			y^{4-4k}c_k(y^{4\kappa})&\textstyle=\sum_{l=0}^Ny^{-2l}d_{k,l}+O\left(y^{-2(N+1)}\right)&&\varepsilon\leq y<\infty
			\\
			\label{E:e.d}
			t^{(l/2+1)/\kappa}e_l(t)&\textstyle=\sum_{k=0}^Nt^{k/\kappa}d_{k,l}+O\left(t^{(N+1)/\kappa}\right)&&0<t\leq T
			\\
			\label{E:e.infty}
			e_l(t)&\textstyle=O\left(t^{-M}\right)&&\varepsilon\leq t<\infty
		\end{align}
		The estimates are all uniform in the indicated range for $r$ and $t$, the implicit constants, however, depend on $A$, $\lambda$, $\mu$, $\nu$, $\varepsilon$, $T$, $M$, and $N$.
		Moreover, $a_j$, $b_{j,k}$ vanish if $j$ is odd; $b_{j,k}$, $c_k$, $d_{k,l}$ vanish if $k$ is odd; $d_{k,l}$, $e_l$ vanish if $l$ is odd; and $b_{j,k}$ vanishes if $3k>j$.
		Finally,
		\begin{equation}\label{E:e0}
			e_0(t)=\tr\left(e^{-t\rho_\hbar(A)}\right)+\tr\left(e^{-t\rho_{-\hbar}(A)}\right)
		\end{equation}
		where $\rho_{\pm\hbar}$ denotes the Schr\"odinger representation with $\hbar=\sqrt{|\nu|}$.
	\end{theorem}

	\begin{proof}
		W.l.o.g.\ we may assume $\lambda^2+\mu^2=1$.
		For $r>0$ let $\alpha_r\in\goe^*$ be such that $\alpha_r(X_1)=0=\alpha_r(X_2)$, $\alpha_r(X_3)=\sqrt{|\nu|}$, and $\alpha_r(X_4)=r\lambda$, $\alpha_r(X_5)=r\mu$.
		Put $\alpha:=\alpha_1$, and note that the vectors $Z_{\alpha_r}=Z_\alpha$ in \eqref{E:Zalpha} and $Z'_{\alpha_r}=Z'_\alpha$ in \eqref{E:Z'alpha} are independent of $r$.
		In particular, the polarization algebra $\poe_{\alpha_r}=\poe_\alpha$ does not depend on $r$ either.
		As observed in the preceding section, the representation $\pi_{\alpha_r}$ associated to $\alpha_r$ via Kirillov's orbit method is unitarily equivalent to the generic representation $\rho_{r\lambda,r\mu,\nu}$, cf.~\eqref{E:nu}.
		From \eqref{E:tr.eAt} we thus obtain
		\begin{equation}\label{E:tr.k.alpha}
			\tr\left(e^{-t\rho_{r\lambda,r\mu,\nu}(A)}\right)
			=\frac{t^{-3/4\kappa}}{\sqrt r}\int_{-\infty}^\infty\tr\hat k_\alpha
			\begin{pmatrix}t^{2/4\kappa}\nu/2r+x^2/2\\t^{3/4\kappa}\sqrt rx\\t^{6/4\kappa}r\lambda\\t^{6/4\kappa}r\mu\end{pmatrix}dx
		\end{equation}
		where $\hat k_\alpha=\hat k_{\alpha_r}\in\mathcal S(\R^4)\otimes\eend(V)$ is independent of $r$ and $\nu$.

		Let $\alpha_0^\pm\in\goe^*$ be such that $\alpha_0^\pm(X_1)=0=\alpha_0^\pm(X_2)$, $\alpha_0^\pm(X_3)=\pm\sqrt{|\nu|}$, and $\alpha_0^\pm(X_4)=0=\alpha_0^\pm(X_5)$.
		In the preceding section we have seen that the associated representation $\pi_{\alpha_0^\pm}$ is unitarily equivalent to the Schr\"odinger representation $\rho_{\pm\sqrt{|\nu|}}$.
		Note that $\poe_\alpha$ also is a polarization algebra for $\alpha_0^\pm$, and we may use the same complementary vectors $Z=Z_\alpha$ and $Z'=Z_\alpha'$ for $\alpha_0^\pm$, cf.~\eqref{E:ZZp}.
		From \eqref{E:tr.eAt.h} we thus obtain
		\begin{equation}\label{E:tr.k.alpha.h}
			\tr\left(e^{-t\rho_{\pm\sqrt{|\nu|}}(A)}\right)
			=\frac{t^{-1/\kappa}}{\sqrt{|\nu|}}
			\int_{-\infty}^\infty
			\tr\hat k_\alpha
			\begin{pmatrix}
				x\\
				\pm t^{1/\kappa}\sqrt{|\nu|}\\
				0\\
				0
			\end{pmatrix}
			dx
		\end{equation}
		where the $\hat k_\alpha$ in the integrand is the same as in \eqref{E:tr.k.alpha}.

		From \eqref{E:tr.k.alpha} we have
		\begin{equation}\label{E:tr.f}
			t^{3/4\kappa}\sqrt r\tr\left(e^{-t\rho_{r\lambda,r\mu,\nu}(A)}\right)
			=\tilde f\left(\tfrac{t^{1/4\kappa}}{\sqrt r},r^2\right)
		\end{equation}
		where 
		\begin{equation}\label{E:f.oo}
			\tilde f(y,v)
			=\int_{-\infty}^\infty
			\tr\hat k_\alpha
			\begin{pmatrix}\nu y^2/2+x^2/2\\vy^3x\\v^2y^6\lambda\\v^2y^6\mu\end{pmatrix}
			dx.
		\end{equation}
		We will show that the function $\tilde f(y,v)$ is smooth in $(y,v)\in\R^2$.

		To this end, we consider slightly more general functions of the form
		\[
			\tilde g(y,v)
			=p(y,v)\int_{-\infty}^\infty 
			k\begin{pmatrix}\nu y^2/2+x^2/2\\vy^3x\\v^2y^6\lambda\\v^2y^6\mu\end{pmatrix}
			dx
		\]
		where $p(y,v)$ is a polynomial, $k\in\mathcal S(\R^4)$ and $y,v\in\R$.
		The integral converges absolutely since $k$ is in the Schwartz space.
		Moreover, the partial derivatives $\tfrac{\partial\tilde g}{\partial y}(y,v)$ and $\tfrac{\partial\tilde g}{\partial v}(y,v)$ are functions of the same form as $\tilde g(y,v)$.
		Indeed, differentiation under the integral is justified by the uniform and fast decay in $x$ of the integrand and its derivatives.
		We conclude that all partial derivatives of $\tilde g$ exist, whence $\tilde g$ is smooth.
		In particular, the function $\tilde f(y,v)$ in \eqref{E:f.oo} is smooth in $(y,v)\in\R^2$.
		In view of \eqref{E:tr.f} this yields part (a) of the theorem.
		The expansions in \eqref{E:tr.a}, \eqref{E:tr.b}, \eqref{E:tr.c}, \eqref{E:a.b}, \eqref{E:c.b} follow immediately from Taylor's theorem applied to the smooth extension in (a). 
		Substituting $-x\leftrightarrow x$ in the integral \eqref{E:f.oo}, we obtain $\tilde f(-y,v)=\tilde f(y,v)=\tilde f(y,-v)$.
		Hence, $a_j(r)$, $b_{j,k}$, and $c_k(y)$ all vanish if $j$ or $k$ is odd, respectively.
		Note that $\tilde g(0,v)$ is polynomial in $v$.
		Hence, $r^2a_j(r)$ is a polynomial in $r^2$.
	
		Let us work out more explicit formulas for the coefficients $a_j(r)$, $b_{j,k}$, and $c_k(y)$.
		By Taylor's theorem,
		\begin{align}\notag
			\hat k_\alpha&\begin{pmatrix}\nu y^2/2+x^2/2\\vy^3x\\v^2y^6\lambda\\v^2y^6\mu\end{pmatrix}
			=\sum_{\substack{i_1,i_2,i_3,i_4\in\N_0\\i_1+i_2+i_3+i_4\leq N}}
			v^{i_2+2i_3+2i_4}y^{2i_1+3i_2+6i_3+6i_4}
			\\\notag
			&\hspace{5cm}\cdot(\nu/2)^{i_1}x^{i_2}\lambda^{i_3}\mu^{i_4}
			\frac{\partial_1^{i_1}\partial_2^{i_2}\partial_3^{i_3}\partial_4^{i_4}\hat k_\alpha}{i_1!i_2!i_3!i_4!}
			\begin{pmatrix}x^2/2\\0\\0\\0\end{pmatrix}
			\\\notag
			&+\sum_{\substack{i_1,i_2,i_3,i_4\in\N_0\\i_1+i_2+i_3+i_4=N+1}}
			v^{i_2+2i_3+2i_4}y^{2i_1+3i_2+6i_3+6i_4}(\nu/2)^{i_1}x^{i_2}\lambda^{i_3}\mu^{i_4}
			\\\label{E:Taylor.0}
			&\qquad\cdot(N+1)\int_0^1(1-\xi)^N\frac{\partial_1^{i_1}\partial_2^{i_2}\partial_3^{i_3}\partial_4^{i_4}\hat k_\alpha}{i_1!i_2!i_3!i_4!}
			\begin{pmatrix}\xi\nu y^2/r+x^2/2\\\xi vy^3x\\\xi v^2y^6\lambda\\\xi v^2y^6\mu\end{pmatrix}d\xi.
		\end{align}
		As $\hat k_\alpha$ is a Schwartz function, 
		\begin{equation}\label{E:esti.0}
			x^{i_2}\frac{\partial_1^{i_1}\partial_2^{i_2}\partial_3^{i_3}\partial_4^{i_4}\hat k_\alpha}{i_1!i_2!i_3!i_4!}
			\begin{pmatrix}\xi\nu y^2/2+x^2/2\\\xi vy^3x\\\xi v^2y^6\lambda\\\xi v^2y^6\mu\end{pmatrix}
			=O\left(\frac1{1+x^2}\right),
		\end{equation}
		uniformly for $|y|\leq T$, $v\in\R$, $x\in\R$, and $0\leq\xi\leq1$.
		Combining this estimate with \eqref{E:f.oo} and \eqref{E:Taylor.0}, we obtain  
		\[
			\tilde f(y,v)
			=\sum_{\substack{j,k\in\N_0\\j+k\leq N}}v^ky^jb_{j,k}
			+O\left((|v|+|y|)^{N+1}\right),
		\]
		uniformly for $|y|\leq T$ and $|v|\leq T$ with
		\begin{equation}\label{E:bjk.formula}
			b_{j,k}
			=\sum_{\substack{i_1,i_2,i_3,i_4\in\N_0\\2i_1+3i_2+6i_3+6i_4=j\\i_2+2i_3+2i_4=k}}
			(\nu/2)^{i_1}\lambda^{i_3}\mu^{i_4}\int_{-\infty}^\infty\tr
			\frac{\partial_1^{i_1}\partial_2^{i_2}\partial_3^{i_3}\partial_4^{i_4}\hat k_\alpha}{i_1!i_2!i_3!i_4!}
			\begin{pmatrix}x^2/2\\0\\0\\0\end{pmatrix}x^{i_2}dx.
		\end{equation}
		Combining this with \eqref{E:tr.f}, we obtain the estimate in \eqref{E:tr.b}.
		Clearly, $b_{j,k}$ vanishes for $3k>j$.
		Hence, \eqref{E:f.oo}, \eqref{E:Taylor.0}, and \eqref{E:esti.0} also yield
		\[
			\tilde f(y,v)
			=\sum_{j=0}^Ny^jva_j(\sqrt v)+O\left(|y|^{N+1}\right)
		\]
		uniformly on $|y|\leq T$ and $0<v\leq T$ with 
		\begin{equation}\label{E:aj.formula}
			va_j(\sqrt v)
			=\sum_{k=0}^{\lfloor j/3\rfloor}v^kb_{j,k}.
		\end{equation}
		Combining this with \eqref{E:tr.f}, we obtain the estimate in \eqref{E:tr.a}.
		By Taylor's theorem,
		\begin{align}\notag
			\hat k_\alpha&\begin{pmatrix}\nu y^2/2+x^2/2\\vy^3x\\v^2y^6\lambda\\v^2y^6\mu\end{pmatrix}
			=\sum_{\substack{i_2,i_3,i_4\in\N_0\\i_2+i_3+i_4\leq N}}
			v^{i_2+2i_3+2i_4}y^{3i_2+6i_3+6i_4}
			\\\notag
			&\hspace{4cm}\cdot x^{i_2}\lambda^{i_3}\mu^{i_4}
			\frac{\partial_2^{i_2}\partial_3^{i_3}\partial_4^{i_4}\hat k_\alpha}{i_2!i_3!i_4!}
			\begin{pmatrix}\nu y^2/2+x^2/2\\0\\0\\0\end{pmatrix}
			\\\notag
			&+\sum_{\substack{i_2,i_3,i_4\in\N_0\\i_2+i_3+i_4=N+1}}
			v^{i_2+2i_3+2i_4}y^{3i_2+6i_3+6i_4}x^{i_2}\lambda^{i_3}\mu^{i_4}
			\\\label{E:Taylor.1}
			&\qquad\cdot(N+1)\int_0^1(1-\xi)^N\frac{\partial_2^{i_2}\partial_3^{i_3}\partial_4^{i_4}\hat k_\alpha}{i_2!i_3!i_4!}
			\begin{pmatrix}\nu y^2/2+x^2/2\\\xi vy^3x\\\xi v^2y^6\lambda\\\xi v^2y^6\mu\end{pmatrix}d\xi.
		\end{align}
		As $\hat k_\alpha$ is in the Schwartz space, 
		\[
			x^{i_2}\frac{\partial_2^{i_2}\partial_3^{i_3}\partial_4^{i_4}\hat k_\alpha}{i_2!i_3!i_4!}
			\begin{pmatrix}\nu y^2/2+x^2/2\\\xi vy^3x\\\xi v^2y^6\lambda\\\xi v^2y^6\mu\end{pmatrix}
			=O\left(\frac1{1+x^2}\right)
		\]
		uniformly for $|y|\leq T$, $v\in\R$, $x\in\R$, and $0\leq\xi\leq1$.
		Combining this estimate with \eqref{E:f.oo} and \eqref{E:Taylor.1}, we obtain  
		\begin{equation}\label{E:tr.c.old}
			\tilde f(y,v)
			=\sum_{k=0}^Nv^ky^3c_k(y^{4\kappa})+O\bigl(|v|^{N+1}\bigr)
		\end{equation}
		uniformly for $0<y\leq T$ and $|v|\leq T$ with 
		\begin{multline}\label{E:ck.formula}
			y^3c_k(y^{4\kappa})
			=y^{3k}\sum_{\substack{i_2,i_3,i_4\in\N_0\\i_2+2i_3+2i_4=k}}
			\lambda^{i_3}\mu^{i_4}
			\\
			\cdot\int_{-\infty}^\infty\tr\frac{\partial_2^{i_2}\partial_3^{i_3}\partial_4^{i_4}\hat k_\alpha}{i_2!i_3!i_4!}
			\begin{pmatrix}\nu y^2/2+x^2/2\\0\\0\\0\end{pmatrix}x^{i_2}dx.
		\end{multline}
		Combining this with \eqref{E:tr.f}, we obtain the estimate in \eqref{E:tr.c}.
		These representations clearly exhibit the fact that the coefficients $a_j(r)$, $b_{j,k}$ and $c_k(y)$ vanish if $j$ or $k$ is odd.
		Indeed, the integrals in \eqref{E:bjk.formula} and \eqref{E:ck.formula} vanish as the integrands are odd functions of $x$ in this case.

		Substituting $t^{1/2\kappa}\nu/2r+x^2/2\leftrightarrow x$ in the integral \eqref{E:tr.k.alpha}, we obtain 
		\begin{multline}\label{E:tr.D.pm}
			\tr\left(e^{-t\rho_{r\lambda,r\mu,\nu}(A)}\right)
			=\frac{t^{-1/\kappa}}{\sqrt{|\nu|}}
			\int^\infty_{\frac{-|\nu|t^{1/2\kappa}}{2r}}\tr\hat k_\alpha
			\begin{pmatrix}x\\\sqrt{|\nu|}t^{1/\kappa}\sqrt{1+\frac{2rx}{|\nu|t^{1/2\kappa}}}\\t^{3/2\kappa}r\lambda\\t^{3/2\kappa}r\mu\end{pmatrix}
			\frac{dx}{\sqrt{1+\frac{2rx}{|\nu|t^{1/2\kappa}}}}
			\\
			+\frac{t^{-1/\kappa}}{\sqrt{|\nu|}}
			\int^\infty_{\frac{-|\nu|t^{1/2\kappa}}{2r}}\tr\hat k_\alpha
			\begin{pmatrix}x\\-\sqrt{|\nu|}t^{1/\kappa}\sqrt{1+\frac{2rx}{|\nu|t^{1/2\kappa}}}\\t^{3/2\kappa}r\lambda\\t^{3/2\kappa}r\mu\end{pmatrix}
			\frac{dx}{\sqrt{1+\frac{2rx}{|\nu|t^{1/2\kappa}}}}.
		\end{multline}
		Hence,
		\begin{equation}\label{E:tr.fpm}
			t^{1/\kappa}\tr\left(e^{-t\rho_{r\lambda,r\mu,\nu}(A)}\right)
			=f\left(\tfrac r{t^{1/2\kappa}},t^{1/\kappa}\right)
			+f\left(\tfrac r{t^{1/2\kappa}},-t^{1/\kappa}\right)
		\end{equation}
		where, for $z>0$ and $w\in\R$,
		\begin{equation}\label{E:fpm}
			f(z,w)
			=\frac1{\sqrt{|\nu|}}
			\int_{-|\nu|/2z}^\infty\tr\hat k_\alpha
			\begin{pmatrix}x\\\sqrt{|\nu|}w\sqrt{1+\frac{2zx}{|\nu|}}\\w^2z\lambda\\w^2z\mu\end{pmatrix}
			\frac{dx}{\sqrt{1+\frac{2zx}{|\nu|}}}.
		\end{equation}
		We will show that $f(z,w)$ extends to a smooth function in $(z,w)\in[0,\infty)\times\R$.

		Splitting the integral and substituting, we write this in the form
		\begin{equation}\label{E:fpm12}
			f(z,w)=f_1(z,w)+f_2(z,w)
		\end{equation}
		with
		\begin{equation}\label{E:f1pm}
			f_1(z,w)
			=\frac{\sqrt{|\nu|}}{2z}\int_{-1}^{-1/2}\tr\hat k_\alpha
			\begin{pmatrix}x|\nu|/2z\\\sqrt{|\nu|}w\sqrt{1+x}\\w^2z\lambda\\w^2z\mu\end{pmatrix}
			\frac{dx}{\sqrt{1+x}}
		\end{equation}
		and
		\begin{equation}\label{E:f2pm}
			f_2(z,w)
			=\frac1{\sqrt{|\nu|}}\int_{-|\nu|/4z}^\infty\tr\hat k_\alpha
			\begin{pmatrix}x\\\sqrt{|\nu|}w\sqrt{1+\frac{2zx}{|\nu|}}\\w^2z\lambda\\w^2z\mu\end{pmatrix}
			\frac{dx}{\sqrt{1+\frac{2zx}{|\nu|}}}.
		\end{equation}

		To analyze $f_1(z,w)$ we consider functions of the form
		\[
			g_1(z,w)
			=z^lw^p\int_{-1}^{-1/2} 
			k\begin{pmatrix}x|\nu|/2z\\\sqrt{|\nu|}w\sqrt{1+x}\\w^2z\lambda\\w^2z\mu\end{pmatrix}
			(1+x)^{\frac{q-1}2}dx
		\]
		where $l\in\Z$, $p,q\in\N_0$, $k\in\mathcal S(\R^4)$, $z>0$, and $w\in\R$.
		As $k$ is a Schwartz function, 
		\begin{equation}\label{E:g1pm.infty}
			g_1(z,w)
			=O\left(|w|^pz^N\right)
		\end{equation}
		for each $N$, uniformly for $0<z\leq T$ and $w\in\R$.
		In particular, $g_1(z,w)$ extends continuously to $z=0$ where it is flat.
		Note that $\frac{\partial g_1}{\partial z}(z,w)$ and $\frac{\partial g_1}{\partial w}(z,w)$ are finite linear combinations of functions of the same type as $g_1(z,w)$.
		Hence, all partial derivatives of the function $f_1(z,w)$ in \eqref{E:f1pm} extend continuously to $z=0$.
		We conclude that $f_1(z,w)$ extends to a smooth function in $(z,w)\in[0,\infty)\times\R$ which is flat along $z=0$.
		From \eqref{E:g1pm.infty} we obtain
		\begin{equation}\label{E:f1pm.infty}
			f_1(z,w)
			=O\left(z^N\right)
		\end{equation}
		for all $N$, uniformly in $0<z\leq T$ and $w\in\R$.

		To analyze $f_2(z,w)$ we consider functions of the form
		\[
			g_2(z,w)
			=z^lw^p\int_{-|\nu|/4z}^\infty
			k\begin{pmatrix}x\\\sqrt{|\nu|}w\sqrt{1+\frac{2zx}{|\nu|}}\\w^2z\lambda\\w^2z\mu\end{pmatrix}
			\left(1+\tfrac{2zx}{|\nu|}\right)^{\frac{q}2}dx.
		\]
		where $l,p\in\N_0$, $q\in\Z$, $k\in\mathcal S(\R^4)$, $z>0$, and $w\in\R$.
		As $k$ is a Schwartz function, 
		\[
			k\begin{pmatrix}x\\\sqrt{|\nu|}w\sqrt{1+\frac{2zx}{|\nu|}}\\w^2z\lambda\\w^2z\mu\end{pmatrix}
			\left(1+\frac{2zx}{|\nu|}\right)^{\frac{q}2}
			=k\begin{pmatrix}x\\\sqrt{|\nu|}w\\0\\0\end{pmatrix}
			+O\left(\frac z{1+x^2}\left(\frac1{1+w^2}\right)^M\right)
		\]
		for each $M$, uniformly in $0<z\leq T$, $w\in\R$, and $x\geq-|\nu|/4z$.
		Moreover,
		\[
			\int_{-\infty}^{-|\nu|/4z}
                        k\begin{pmatrix}x\\\sqrt{|\nu|}w\\0\\0\end{pmatrix}
                        dx
			=O\left(z^N\left(\frac1{1+w^2}\right)^M\right)
		\]
		for all $N$ and $M$, uniformly in $0<z\leq T$ and $w\in\R$.
		Hence,	
		\begin{equation}\label{E:g2pm.infty}
			g_2(z,w)
                        =z^lw^p\int_{-\infty}^\infty
                        k\begin{pmatrix}x\\\sqrt{|\nu|}w\\0\\0\end{pmatrix}
                        dx
			+O\left(z^{l+1}\left(\frac1{1+w^2}\right)^M\right)
		\end{equation}
		for each $M$, uniformly in $0<z\leq T$ and $w\in\R$.
		In particular, $g_2(z,w)$ extends continuously to $z=0$.
		Moreover, $\frac{\partial g_2}{\partial z}(z,w)$ and $\frac{\partial g_2}{\partial w}(z,w)$ are finite linear combinations of functions of the same type as $g_2(z,w)$ and $g_3(z,w)$ where the latter is of the form
		\[
			g_3(z,w)
			=z^lw^pk\begin{pmatrix}-|\nu|/4z\\\sqrt{|\nu|/2}w\\w^2z\lambda\\w^2z\mu\end{pmatrix}
		\]
		with $l\in\Z$ and $p\in\N_0$.
		As $k$ is a Schwartz function, 
		\begin{equation}\label{E:g3pm.infty}
			g_3(z,w)
			=O\left(z^N\left(\frac1{1+w^2}\right)^M\right)
		\end{equation}
		for all $M$ and $N$, uniformly for $0<z\leq T$ and $w\in\R$.
		Furthermore, $\frac{\partial g_3}{\partial z}(z,w)$ and $\frac{\partial g_3}{\partial w}(z,w)$ are finite linear combinations of functions of the same type as $g_3(z,w)$.
		We conclude that all partial derivatives of the function $f_2(z,w)$ in \eqref{E:f2pm} extend continuously to $z=0$.
		Hence, $f_2(z,w)$ extends to a smooth functions in $(z,w)\in[0,\infty)\times\R$.

		Combining the observations in the preceding two paragraphs with \eqref{E:fpm12}, we see that $f(z,w)$ extends to a smooth function in $(z,w)\in[0,\infty)\times\R$ with
		\begin{equation}\label{E:f.infty}
			f(0,w)
			=\frac1{\sqrt{|\nu|}}\int_{-\infty}^\infty\tr\hat k_\alpha
			\begin{pmatrix}x\\\sqrt{|\nu|}w\\0\\0\end{pmatrix}dx.
		\end{equation}
		In view of \eqref{E:tr.fpm}, this yields part (b) of the theorem.
		The expansions in \eqref{E:tr.tc}, \eqref{E:tr.d}, \eqref{E:tr.e}, \eqref{E:c.d}, \eqref{E:e.d} follow immediately from Taylor's theorem applied to the smooth extension in (b) and $z=y^{-2}$. 
		To see that the coefficient in \eqref{E:tr.tc} is indeed $y^{4-4k}c_k(y^{4\kappa})$, one has to compare this expansion with the one in \eqref{E:tr.c} for fixed $0<y<\infty$ using the relation $y^4r^2=t^{1/\kappa}$.
		As $f(z,w)+f(z,-w)$ is an even function of $w$, the coefficient $d_{k,l}$ vanishes for odd $k$, cf.~\eqref{E:tr.fpm}.
		
		Let us work out an explicit formula for $e_l(t)$.
		By Taylor's theorem,
		\begin{align}\notag
			\hat k_\alpha&\begin{pmatrix}x\\\sqrt{|\nu|}w\sqrt{1+\frac{2zx}{|\nu|}}\\w^2z\lambda\\w^2z\mu\end{pmatrix}
			=\sum_{\substack{i_2,i_3,i_4\in\N_0\\i_2+i_3+i_4\leq N}}
			w^{i_2+2i_3+2i_4}z^{i_3+i_4}
			\\\notag
			&\qquad\cdot\left(\textstyle{\sqrt{1+\frac{2zx}{|\nu|}}-1}\right)^{i_2}|\nu|^{i_2/2}\lambda^{i_3}\mu^{i_4}
			\frac{\partial_2^{i_2}\partial_3^{i_3}\partial_4^{i_4}\hat k_\alpha}{i_2!i_3!i_4!}
			\begin{pmatrix}x\\\sqrt{|\nu|}w\\0\\0\end{pmatrix}
			\\\notag
			&+\sum_{\substack{i_2,i_3,i_4\in\N_0\\i_2+i_3+i_4=N+1}}
			w^{i_2+2i_3+2i_4}z^{i_3+i_4}
			\left(\textstyle{\sqrt{1+\frac{2zx}{|\nu|}}-1}\right)^{i_2}|\nu|^{i_2/2}\lambda^{i_3}\mu^{i_4}(N+1)
			\\\label{E:Taylor.nu}
			&\qquad\cdot\int_0^1(1-\xi)^N\frac{\partial_2^{i_2}\partial_3^{i_3}\partial_4^{i_4}\hat k_\alpha}{i_2!i_3!i_4!}
			\begin{pmatrix}x\\\sqrt{|\nu|}w\left(1+\xi\left(\sqrt{1+\frac{2zx}{|\nu|}}-1\right)\right)\\\xi w^2z\lambda\\\xi w^2z\mu\end{pmatrix}d\xi.
		\end{align}
		for $x\geq-|\nu|/2z$.
		As $\hat k_\alpha$ is in the Schwartz space and $|\sqrt{1+u}-1|\leq|u|$ for $-1\leq u<\infty$, we have
		\begin{multline}\label{E:esti.o}
			\left(\textstyle{\sqrt{1+\frac{2zx}{|\nu|}}-1}\right)^{i_2}
			\frac{\partial_2^{i_2}\partial_3^{i_3}\partial_4^{i_4}\hat k_\alpha}{i_2!i_3!i_4!}
			\begin{pmatrix}x\\\sqrt{|\nu|}w\left(1+\xi\left(\sqrt{1+\frac{2zx}{|\nu|}}-1\right)\right)\\\xi w^2z\lambda\\\xi w^2z\mu\end{pmatrix}
			\\
			=O\left(\left(\frac1{1+w^2}\right)^M\frac{z^{i_2}}{1+x^2}\right)
		\end{multline}
		for each $M$, uniformly for $0<z\leq T$, $w\in\R$, $x\geq-|\nu|/4z$ and $0\leq\xi\leq1$.

		Combining \eqref{E:f2pm} with \eqref{E:Taylor.nu} and \eqref{E:esti.o}, we obtain
		\begin{align}\notag
			f_2(z,w)
			&=\sum_{\substack{i_2,i_3,i_4\in\N_0\\i_2+i_3+i_4\leq N}}
			w^{i_2+2i_3+2i_4}z^{i_3+i_4}|\nu|^{(i_2-1)/2}\lambda^{i_3}\mu^{i_4}
			\\\notag&\qquad
			\cdot\int_{-|\nu|/4z}^{|\nu|/4z}\tr\frac{\partial_2^{i_2}\partial_3^{i_3}\partial_4^{i_4}\hat k_\alpha}{i_2!i_3!i_4!}
			\begin{pmatrix}x\\\sqrt{|\nu|}w\\0\\0\end{pmatrix}
			\frac{\left(\sqrt{1+\frac{2zx}{|\nu|}}-1\right)^{i_2}dx}{\sqrt{1+\frac{2zx}{|\nu|}}}
			\\\label{E:taylor.pm.pre}
			&\qquad+O\left(\left(\frac1{1+w^2}\right)^Mz^{N+1}\right)
		\end{align}
		for all $N$ and $M$, uniformly for $0<z\leq T$ and $w\in\R$.
		Let the coefficients $C_{i_2,q}$ be defined through the Taylor expansion
		\begin{equation}\label{E:taylor.sqrt}
			\frac{\left(\sqrt{1+2u}-1\right)^{i_2}}{\sqrt{1+2u}}
			=\sum_{q=0}^QC_{i_2,q}u^{i_2+q}
			+O\left(u^{i_2+Q+1}\right),\qquad|u|<1/4
		\end{equation}
		with leading coefficient $C_{i_2,0}=1$.
		As $\hat k_\alpha$ is in the Schwartz space, for all $q\in\N_0$ and any $N$ and $M$, we have
		\begin{equation}\label{E:int.nu42.iv}
			\int_{-|\nu|/4z}^{|\nu|/4z}
			\frac{\partial_2^{i_2}\partial_3^{i_3}\partial_4^{i_4}\hat k_\alpha}{i_2!i_3!i_4!}
			\begin{pmatrix}x\\\sqrt{|\nu|}w\\0\\0\end{pmatrix}
			x^qdx
			=O\left(\left(\frac1{1+w^2}\right)^M\right)
		\end{equation}
		uniformly for $0<z\leq T$ and $w\in\R$; and
		\begin{equation}\label{E:int.nu42.iii}
			\int_{|x|\geq\frac{|\nu|}{4z}}
			\frac{\partial_2^{i_2}\partial_3^{i_3}\partial_4^{i_4}\hat k_\alpha}{i_2!i_3!i_4!}
			\begin{pmatrix}x\\\sqrt{|\nu|}w\\0\\0\end{pmatrix}
			x^qdx
			=O\left(\left(\frac1{1+w^2}\right)^Mz^N\right)
		\end{equation}
		uniformly for $0<z\leq T$ and $w\in\R$.
		Combining \eqref{E:taylor.pm.pre} with \eqref{E:taylor.sqrt}, \eqref{E:int.nu42.iv}, and \eqref{E:int.nu42.iii} we obtain
		\begin{multline}\label{E:taylor.f2.pm}
			f_2(z,w)
			=\sum_{\substack{q,i_2,i_3,i_4\in\N_0\\q+i_2+i_3+i_4\leq N}}
			w^{i_2+2i_3+2i_4}z^{q+i_2+i_3+i_4}|\nu|^{-(2q+i_2+1)/2}\lambda^{i_3}\mu^{i_4}C_{i_2,q}
			\\
			\cdot\int_{-\infty}^\infty\tr\frac{\partial_2^{i_2}\partial_3^{i_3}\partial_4^{i_4}\hat k_\alpha}{i_2!i_3!i_4!}
			\begin{pmatrix}x\\\sqrt{|\nu|}w\\0\\0\end{pmatrix}
			x^{i_2+q}dx
			+O\left(\left(\frac1{1+w^2}\right)^Mz^{N+1}\right)
		\end{multline}
		for all $N$ and $M$, uniformly for $0<z\leq T$ and $w\in\R$.
		Combining this with \eqref{E:fpm12} and \eqref{E:f1pm.infty} we get
		\begin{multline}\label{E:taylor.pm}
			f(z,w)
			=\sum_{\substack{q,i_2,i_3,i_4\in\N_0\\q+i_2+i_3+i_4\leq N}}
			w^{i_2+2i_3+2i_4}z^{q+i_2+i_3+i_4}|\nu|^{-(2q+i_2+1)/2}\lambda^{i_3}\mu^{i_4}C_{i_2,q}
			\\
			\cdot\int_{-\infty}^\infty\tr\frac{\partial_2^{i_2}\partial_3^{i_3}\partial_4^{i_4}\hat k_\alpha}{i_2!i_3!i_4!}
			\begin{pmatrix}x\\\sqrt{|\nu|}w\\0\\0\end{pmatrix}
			x^{i_2+q}dx
			+O\left(\left(\frac1{1+w^2}\right)^Mz^{N+1}\right)
		\end{multline}
		for all $N$ and $M$, uniformly for $w\in\R$ and $0<z\leq T(1+w^2)^{-1/4}$.
		Combining this with \eqref{E:tr.fpm}, we obtain
		\begin{multline}\label{E:tr.e.old}
			t^{1/\kappa}\tr\left(e^{-t\rho_{r\lambda,r\mu,\nu}(A)}\right)
			=\sum_{l=0}^N\left(\frac r{t^{1/2\kappa}}\right)^lt^{(l/2+1)/\kappa}e_l(t)
			\\
			+O\left(\left(\frac1{1+t}\right)^M\left(\frac r{t^{1/2\kappa}}\right)^{N+1}\right)
		\end{multline}
		for all $M$ and $N$, uniformly for $r,t>0$ satisfying $\frac r{t^{1/2\kappa}}\leq T(1+t^{2/\kappa})^{-1/4}$, where $e_l(t)=e_l^+(t)+e_l^-(t)$ with
		\begin{multline}\label{E:epm.formula}
			e^\pm_l(t)
			=\sum_{\substack{q,i_2,i_3,i_4\in\N_0\\q+i_2+i_3+i_4=l}}
			t^{(i_2+2i_3+2i_4-l/2-1)/\kappa}
			\frac{(\pm1)^{i_2}\lambda^{i_3}\mu^{i_4}}{|\nu|^{(2q+i_2+1)/2}}C_{i_2,q}
			\\\cdot\int_{-\infty}^\infty\tr\frac{\partial_2^{i_2}\partial_3^{i_3}\partial_4^{i_4}\hat k_\alpha}{i_2!i_3!i_4!}
			\begin{pmatrix}x\\\pm\sqrt{|\nu|}t^{1/\kappa}\\0\\0\end{pmatrix}
			x^{q+i_2}dx.
		\end{multline}
		The estimate in \eqref{E:tr.e.old} holds uniformly for $\varepsilon\leq\frac{t^{1/2\kappa}}r<\infty$ and $0<t\leq T$, whence the estimate in \eqref{E:tr.e}.
		It also holds uniformly for $\varepsilon\leq t<\infty$ and $0<r\leq1$, whence the estimate in \eqref{E:tr.infty}.
		If $l$ is odd, then $e_l^\pm(t)$ vanishes for the integrand in \eqref{E:epm.formula} is an odd function of $x$ in view of \eqref{E:homog-1}.
		In particular, $e_l(t)=0$ for odd $l$.
		For the leading term we find
		\[
			e^\pm_0(t)
			=\frac{t^{-1/\kappa}}{\sqrt{|\nu|}}\int_{-\infty}^\infty\tr\hat k_\alpha
			\begin{pmatrix}x\\\pm\sqrt{|\nu|}t^{1/\kappa}\\0\\0\end{pmatrix}dx.
		\]
		Combining this with \eqref{E:tr.k.alpha.h} we obtain \eqref{E:e0}.

		Proceeding analogously, one can show $d_{k,l}=d^+_{k,l}+d^-_{k,l}$ where
		\begin{equation}\label{E:dpm.formula}
			d^\pm_{k,l}
			=\sum_{\substack{p,i_2,i_3,i_4\in\N_0\\i_2+2i_3+2i_4=k\\p+i_3+i_4=l}}
			\frac{(\pm1)^{i_2}\lambda^{i_3}\mu^{i_4}}{|\nu|^{(2p-i_2+1)/2}}\binom{\frac{i_2-1}2}p2^p
			\int_{-\infty}^\infty\frac{\partial_2^{i_2}\partial_3^{i_3}\partial_4^{i_4}\hat k_\alpha}{i_2!i_3!i_4!}
			\begin{pmatrix}x\\0\\0\\0\end{pmatrix}x^pdx.
		\end{equation}
		If $k$ is odd, then $d_{k,l}^+$ and $d_{k,l}^-$ differ by a sign and $d_{k,l}$ vanishes.
		If $l$ is odd, then $d^\pm_{k,l}$ vanishes as the integrand in \eqref{E:dpm.formula} is an odd function of $x$ in view of \eqref{E:homog-1}.
		In particular, $d_{k,l}=0$ for odd $l$.
		The estimate in \eqref{E:e.infty} follows from \eqref{E:epm.formula}.
	\end{proof}

	\begin{remark}\label{R:poly.pos}
		For $\nu\geq0$, Theorem~\ref{T:poly}(a) and the corresponding expansions in \eqref{E:tr.a}, \eqref{E:tr.b}, \eqref{E:tr.c}, \eqref{E:a.b}, and \eqref{E:c.b} remain true.
		Indeed, the formula in \eqref{E:tr.k.alpha} holds for all $\nu$ and so do the expressions for the coefficients in \eqref{E:bjk.formula}, \eqref{E:aj.formula}, and \eqref{E:ck.formula}.
		The proof we gave remains valid for all $\nu$.
		
		For $\nu>0$, Theorem~\ref{T:poly}(b) remains true and the extension vanishes to infinite order along the boundary $\frac r{t^{1/2\kappa}}=0$.
		In particular, \eqref{E:tr.tc}, \eqref{E:tr.d}, \eqref{E:tr.e}, \eqref{E:tr.infty}, and \eqref{E:c.d} remain true with $d_{k,l}=0$ and $e_l=0$.
		Indeed, writing \eqref{E:fpm} in the form
		\[
			f(z,w)
			=\frac1{2z}\int_\nu^\infty\tr\hat k_\alpha
			\begin{pmatrix}x/2z\\w\sqrt{x-\nu}\\w^2z\lambda\\w^2z\mu\end{pmatrix}
			\frac{dx}{\sqrt{x-\nu}}
		\]
		equation \eqref{E:tr.fpm} remains true for all $\nu\in\R$.
		If $\nu>0$, then $f(z,w)$ extends to a smooth function in $(z,w)\in[0,\infty)\times\R$ that vanishes to infinte order along the boundary $z=0$.
		This can be shown exactly as the analogous fact for $f_1$, cf.~\eqref{E:f1pm}.
	\end{remark}

\section{Asymptotics of the zeta function}\label{S:zeta.asymp}

	We continue to consider a left invariant homogeneous positive Rockland differential operator $A$ on the 5-dimensional Lie group $G$, and a generic representation $\rho_{\lambda,\mu,\nu}$ of $G$ with $\nu<0$.
	In this section we will use the polyhomogeneous expansion in Theorem~\ref{T:poly} to establish the asymptotics of the zeta function 
	\begin{equation}\label{E:zeta}
		\zeta_{\rho_{r\lambda,r\mu,\nu}(A)}(s)
		=\tr\left((\rho_{r\lambda,r\mu,\nu}(A))^{-s}\right)
		=\frac1{\Gamma(s)}\int_0^\infty t^{s-1}\tr\left(e^{-t\rho_{r\lambda,r\mu,\nu}(A)}\right)dt
	\end{equation}
	for $r\to0$ as formulated in Theorem~\ref{T:Melin} below.

	Suppose $r>0$.
	In view of \eqref{E:tr.a} and \eqref{E:tr.infty}, the integral on the right hand side in \eqref{E:zeta} converges for $\Re s>3/4\kappa$, and extends to a meromorphic function in $s$ on the entire complex plane which has only simple poles located at $s=(3-j)/4\kappa$ with residuum $r^{(3-j)/2}a_j(r)$, $j\in\N_0$.
	We will use the representation
	\begin{multline}\label{E:tr.A.reg}
		\int_0^\infty t^{s-1}\tr\left(e^{-t\rho_{r\lambda,r\mu,\nu}(A)}\right)dt
		\\
		=\int_0^{r^{2\kappa}}t^{s-1}\left(\tr\left(e^{-t\rho_{r\lambda,r\mu,\nu}(A)}\right)
		-\sum_{j=0}^J\left(\frac t{r^{2\kappa}}\right)^{(j-3)/4\kappa}a_j(r)\right)dt
		\\
		+r^{2\kappa s}\sum_{j=0}^J\frac{a_j(r)}{s+(j-3)/4\kappa}
		+\int_{r^{2\kappa}}^\infty t^{s-1}\tr\left(e^{-t\rho_{r\lambda,r\mu,\nu}(A)}\right)dt,
	\end{multline}
	where the first integral on the right hand side converges for $\Re s>(2-J)/4\kappa$, and the second integral on the right hand side converges for all $s$.
	
	In view of \eqref{E:e.d} and \eqref{E:e.infty}, the integral $\int_0^\infty t^{s-1}e_l(t)dt$ converges for $\Re s>(l/2+1)/\kappa$, and extends to a meromorphic function in $s$ on the entire complex plane which has only simple poles located at $s=(l/2+1-k)/\kappa$ with residuum $d_{k,l}$, $k\in\N_0$.
	We will use the representation
	\begin{multline}\label{E:el.reg}
		\int_0^\infty t^{s-1}e_l(t)dt
		=\int_0^1t^{s-1}\left(e_l(t)-\sum_{k=0}^Kt^{(k-1-l/2)/\kappa}d_{k,l}dt\right)
		\\
		+\sum_{k=0}^K\frac{d_{k,l}}{s+(k-1-l/2)/\kappa}
		+\int_1^\infty t^{s-1}e_l(t)dt,
	\end{multline}
	where the first integral on the right hand side converges for $\Re s>(l/2-K)/\kappa$, and the second integral on the right hand side converges for all $s$.
	By \eqref{E:e0},
	\begin{equation}\label{E:e0.zeta}
		\frac1{\Gamma(s)}\int_0^\infty t^{s-1}e_0(t)dt=\zeta_{\rho_{\sqrt{|\nu|}}(A)}(s)+\zeta_{\rho_{-\sqrt{|\nu|}}(A)}(s).
	\end{equation}
	
	In view of \eqref{E:c.b} and \eqref{E:c.d}, the integral $\int_0^\infty y^{s-1}c_k(y)dy$ can be regularized, defining a meromorphic function in $s$ on the entire complex plane which has only simple poles located at $s=(3-j)/4\kappa$ with residuum $b_{j,k}$, $j\in\N_0$ and at $s=(l/2+1-k)/\kappa$ with residuum $-d_{k,l}$, $l\in\N_0$.
	We will use the representation
	\begin{multline}\label{E:ck.reg}
		\int_0^\infty y^{s-1}c_k(y)dy
		=\int_0^1y^{s-1}\left(c_k(y)-\sum_{j=0}^Jy^{(j-3)/4\kappa}b_{j,k}\right)dy
		\\
		+\int_1^{r^{-2\kappa}}y^{s-1}c_k(y)dy
		+\int_{r^{-2\kappa}}^\infty y^{s-1}\left(c_k(y)-\sum_{l=0}^Ly^{(2k-2-l)/2\kappa}d_{k,l}\right)dy
		\\
		+\sum_{j=0}^J\frac{b_{j,k}}{s+(j-3)/4\kappa}
		-\sum_{l=0}^L\frac{r^{-2\kappa s+l+2-2k}d_{k.l}}{s+(k-l/2-1)/\kappa},
	\end{multline}
	where the first integral on the right hand side converges for $\Re s>(2-J)/4\kappa$, the second integral converges for all $s$, and the third integral on the right hand side converges for $\Re s<((L+3)/2-k)/\kappa$.
	
	Put 
	\[
		P:=\left(\left\{\frac{3-j}{4\kappa}:j\in\N_0\right\}\cup\left\{\frac{l/2+1-k}\kappa:k,l\in\N_0\right\}\right)\setminus(-\N_0).
	\]
	As $\frac1{\Gamma(s)}$ is an entire function that vanishes for $s\in-\N_0$, the observations in the preceding three paragraphs show that the poles of the functions $\zeta_{\rho_{r\lambda,r\mu,\nu}(A)}(s)$ and $\frac1{\Gamma(s)}\int_0^\infty t^{s-1}e_l(t)dt$ and $\frac1{\Gamma(s)}\int_0^\infty y^{s-1}c_k(y)dy$ are all contained in $P$.

	\begin{theorem}\label{T:Melin}
		Let $V$ finite dimensional complex vector space equipped with a Hermitian inner product, and suppose $A=A^*\in\mathcal U^{-2\kappa}(\goe)\otimes\eend(V)$ is a left invariant positive Rockland differential operator on the 5-dimensional Lie group $G$ which is homogeneous of order $2\kappa>0$ with respect to the grading automorphism.
		Let $\rho_{\lambda,\mu,\nu}$ be a generic representation of $G$ with $\nu<0$.
		If $\sigma$ and $N$ are real numbers, then
		\begin{multline}\label{E:zetas}
			\zeta_{\rho_{r\lambda,r\mu,\nu}(A)}(s)
			=\sum_{l=0}^{\lceil N\rceil-1}\frac{r^l}{\Gamma(s)}\int_0^\infty t^{s-1}e_l(t)dt
			\\
			+\sum_{k=0}^{\left\lceil\frac{N-2\kappa\sigma}2\right\rceil}\frac{r^{2\kappa s+2k-2}}{\Gamma(s)}\int_0^\infty y^{s-1}c_k(y)dy
			+O\left(r^N\right)
		\end{multline}
		uniformly on compact subsets of $\{s\in\C\setminus P:\Re s\geq\sigma\}$ and for $0<r\leq1$.
		Here $\lceil x\rceil=\min\{k\in\Z:x\leq k\}$ denotes the ceiling function.
		In particular, we have an asymptotic expansion of the form
		\begin{multline}\label{E:zetap0}
			\zeta_{\rho_{r\lambda,r\mu,\nu}(A)}'(0)
			\sim\sum_{l=0}^\infty r^l\frac{\partial}{\partial s}\bigg|_{s=0}\frac1{\Gamma(s)}\int_0^\infty t^{s-1}e_l(t)dt
			\\
			+\sum_{k=0}^\infty r^{2k-2}\frac{\partial}{\partial s}\bigg|_{s=0}\frac1{\Gamma(s)}\int_0^\infty y^{s-1}c_k(y)dy
			-2\kappa\log r\sum_{k=1}^\infty r^{2k-2}d_{k,2k-2},
		\end{multline}
		as $r\to0$.
		For the constant term we find
		\begin{equation}\label{E:const}
			\LIM_{r\to0}\zeta_{\rho_{r\lambda,r\mu,\nu}(A)}'(0)
			=\zeta_{\rho_\hbar(A)}'(0)
			+\zeta_{\rho_{-\hbar}(A)}'(0),
		\end{equation}
		where $\rho_{\pm\hbar}$ denotes the Schr\"odinger representation with $\hbar=\sqrt{|\nu|}$.
	\end{theorem}

	Let us point out that the unitary dual of $G$ is homeomorphic to the space of coadjoint orbits $\goe^*/G$ with the quotient topology \cite{B73}.
	This is a non-Hausdorff space, and the family of representations $\rho_{r\lambda,r\mu,\nu}$ has two limits, as $r\to0$, namely the Schr\"odinger representations $\rho_{\hbar}$ and $\rho_{-\hbar}$, corresponding to the two summands on the right hand side in \eqref{E:const} with $\hbar=\sqrt{|\nu|}$.

	\begin{proof}[{Proof of Theorem~\ref{T:Melin}}]
		Put
		\begin{multline}\label{E:fMN}
			f(t,r)
			:=\tr\left(e^{-t\rho_{r\lambda,r\mu,\nu}(A)}\right)
			+\sum_{j=0}^J\,\,\sum_{k=0}^Kr^{2k-2}\left(\frac t{r^{2\kappa}}\right)^{(j-3)/4\kappa}b_{j,k}
			\\-\sum_{j=0}^J\left(\frac t{r^{2\kappa}}\right)^{(j-3)/4\kappa}a_j(r)
			-\sum_{k=0}^Kr^{2k-2}c_k\left(\frac t{r^{2\kappa}}\right).
		\end{multline}
		According to Theorem~\ref{T:poly}(a), the function $t^{3/4\kappa}\sqrt rf(t,r)$ is smooth in the variables $\bigl(\frac{t^{1/4\kappa}}{\sqrt r},r^2\bigr)\in[0,\infty)^2$.
		From \eqref{E:tr.a} and \eqref{E:c.b} we obtain
		\[
			t^{3/4\kappa}\sqrt rf(r,t)\in O\left(\left(\frac{t^{1/4\kappa}}{\sqrt r}\right)^{J+1}\right),
		\] 
		uniformly for $0<t\leq r^{2\kappa}\leq1$.
		Combining \eqref{E:tr.c} and \eqref{E:a.b} we get 
		\[
			t^{3/4\kappa}\sqrt rf(r,t)\in O\left(r^{2(K+1)}\right),
		\] 
		uniformly for $0<t\leq r^{2\kappa}\leq1$.
		By smoothness, these yield
		\[
			t^{3/4\kappa}\sqrt rf(r,t)\in O\left(r^{2(K+1)}\left(\frac{t^{1/4\kappa}}{\sqrt r}\right)^{J+1}\right),
		\]
		uniformly for $0<t\leq r^{2\kappa}\leq1$.
		In other words,
		\[
			f(r,t)\in O\Bigl(r^{2K+1-J/2}t^{(J-2)/4\kappa}\Bigr),
		\] 
		uniformly for $0<t\leq r^{2\kappa}\leq1$.
		Hence, 
		\[
			\int_0^{r^{2\kappa}}t^{s-1}f(r,t)dt=O\left(\frac{r^{2\kappa\Re s+2K}}{\Re s+(J-2)/4\kappa}\right),
		\] 
		uniformly for $\Re s>-(J-2)/4\kappa$ and $0<r\leq1$.
		Combining this with \eqref{E:fMN}, we get
		\begin{multline}\label{E:tr.Melin.a}
			\int_0^{r^{2\kappa}}t^{s-1}\left(\tr\left(e^{-t\rho_{r\lambda,r\mu,\nu}(A)}\right)-\sum_{j=0}^J\left(\frac t{r^{2\kappa}}\right)^{(j-3)/4\kappa}a_j(r)\right)dt
			\\=\sum_{k=0}^Kr^{2\kappa s+2k-2}\int_0^1y^{s-1}\left(c_k(y)-\sum_{j=0}^Jy^{(j-3)/4\kappa}b_{j,k}\right)dy
			\\+O\left(\frac{r^{2\kappa\Re s+2K}}{\Re s+(J-2)/4\kappa}\right)
		\end{multline}
		uniformly for $\Re s>-(J-2)/4\kappa$ and $0<r\leq1$.

		Put 
		\begin{multline}\label{E:gMN}
			g(t,r)
			:=\tr\left(e^{-t\rho_{r\lambda,r\mu,\nu}(A)}\right)
			+\sum_{k=0}^K\sum_{l=0}^Lt^{(k-1)/\kappa}\left(\frac t{r^{2\kappa}}\right)^{-l/2\kappa}d_{k,l}
			\\-\sum_{k=0}^Kt^{(k-1)/\kappa}\left(\frac t{r^{2\kappa}}\right)^{(2-2k)/2\kappa}c_k\left(\frac t{r^{2\kappa}}\right)
			-\sum_{l=0}^L\left(\frac t{r^{2\kappa}}\right)^{-l/2\kappa}t^{l/2\kappa}e_l(t).
		\end{multline}
		According to Theorem~\ref{T:poly}(b), the function $t^{1/\kappa}g(t,r)$ is smooth in the variables $\bigl(\frac r{t^{1/2\kappa}},t^{1/\kappa}\bigr)\in[0,\infty)^2$.
		From \eqref{E:tr.tc} and \eqref{E:e.d} we obtain
		\[
			t^{1/\kappa}g(r,t)\in O\left(t^{(K+1)/\kappa}\right),
		\] 
		uniformly for $0<r^{2\kappa}\leq t\leq1$.
		Combining \eqref{E:tr.e} and \eqref{E:c.d} we get 
		\[
			t^{1/\kappa}g(r,t)\in O\left(\left(\frac r{t^{1/2\kappa}}\right)^{L+1}\right),
		\] 
		uniformly for $0<r^{2\kappa}\leq t\leq1$.
		By smoothness, these yield
		\[
			t^{1/\kappa}g(r,t)\in O\left(t^{(K+1)/\kappa}\left(\frac r{t^{1/2\kappa}}\right)^{L+1}\right),
		\] 
		uniformly for $0<r^{2\kappa}\leq t\leq1$.
		In other words,
		\[
			g(r,t)\in O\left(t^{(2K-1-L)/2\kappa}r^{L+1}\right),
		\] 
		uniformly for $0<r^{2\kappa}\leq t\leq1$.
		Hence, 
		\[
			\int_{r^{2\kappa}}^1t^{s-1}g(r,t)dt
			=O\left(\frac{r^{L+1}}{\Re s+(2K-1-L)/2\kappa}\right),
		\] 
		uniformly for $\Re s>-(2K-1-L)/2\kappa$ and $0<r\leq1$.
		Combining this with \eqref{E:gMN}, we get
		\begin{multline}\label{E:tr.Melin.b}
			\int_{r^{2\kappa}}^1t^{s-1}\tr\left(e^{-t\rho_{r\lambda,r\mu,\nu}(A)}\right)dt
			\\=\sum_{l=0}^Lr^l\int_{r^{2\kappa}}^1t^{s-1}\left(e_l(t)-\sum_{k=0}^Kt^{(2k-2-l)/2\kappa}d_{k,l}\right)dt
			\\+\sum_{k=0}^Kr^{2\kappa s+2k-2}\int_1^{r^{-2\kappa}}y^{s-1}c_k(y)dy
			\\+O\left(\frac{r^{L+1}}{\Re s+(2K-1-L)/2\kappa}\right)
		\end{multline}
		uniformly for $\Re s>-(2K-1-L)/2\kappa$ and $0<r\leq1$.

		From \eqref{E:tr.infty}, we get
		\begin{multline}\label{E:tr.Melin.c}
			\int_1^\infty t^{s-1}\tr\left(e^{-t\rho_{r\lambda,r\mu,\nu}(A)}\right)dt
			=\sum_{l=0}^Lr^l\int_1^\infty t^{s-1}e_l(t)dt
			+O\left(\frac{r^{L+1}}{M-\Re s}\right)
		\end{multline}
		uniformly for $\Re s<M$ and $0<r\leq1$.

		From \eqref{E:e.d} we obtain,
		\[
			\int_0^{r^{2\kappa}}t^{s-1}\left(e_l(t)-\sum_{k=0}^Kt^{(2k-2-l)/2\kappa}d_{k,l}\right)dt
			=O\left(\frac{r^{2\kappa\Re s+2K-l}}{\Re s+(2K-l)/2\kappa}\right)
		\]
		uniformly for $\Re s>-(2K-l)/2\kappa$ and $0<r\leq1$.
		Hence,
		\begin{equation}\label{E:tr.Melin.d}
			\sum_{l=0}^Lr^l\int_0^{r^{2\kappa}}t^{s-1}\left(e_l(t)-\sum_{k=0}^Kt^{(2k-2-l)/2\kappa}d_{k,l}\right)dt
			=O\left(\frac{r^{2\kappa\Re s+2K}}{\Re s+(2K-L)/2\kappa}\right)
		\end{equation}
		uniformly for $\Re s>-(2K-L)/2\kappa$ and $0<r\leq1$.

		From \eqref{E:c.d} we get,
		\[
			\int_{r^{-2\kappa}}^\infty y^{s-1}\left(c_k(y)-\sum_{l=0}^{L'}y^{(2k-2-l)/2\kappa}d_{k,l}\right)dy
			=O\left(\frac{r^{L'+3-2k-2\kappa\Re s}}{(L'+3-2k)/2\kappa-\Re s}\right)
		\]
		uniformly for $\Re s<(L'+3-2k)/2\kappa$ and $0<r\leq1$.
		Hence,
		\begin{multline}\label{E:tr.Melin.e}
			\sum_{k=0}^Kr^{2\kappa s+2k-2}\int_{r^{-2\kappa}}^\infty y^{s-1}\left(c_k(y)-\sum_{l=0}^{L'}y^{(2k-2-l)/2\kappa}d_{k,l}\right)dy
			\\
			=O\left(\frac{r^{L'+1}}{(L'+3-2K)/2\kappa-\Re s}\right)
		\end{multline}
		uniformly for $\Re s<(L'+3-2K)/2\kappa$.

		Let $S$ be a compact subset of $\{s\in\C\setminus P:\Re s\geq\sigma\}$.
		Choose $\Sigma$ such that $\Re s\leq\Sigma$ for all $s\in S$.
		Suppose $J>2-4\kappa\sigma$, $L\geq N-1$, $2K>L+3-2\kappa\sigma$, $M>\Sigma$, and $L'>2K-3+2\kappa\Sigma$.
		Combining \eqref{E:tr.Melin.a}, \eqref{E:tr.Melin.b}, \eqref{E:tr.Melin.c}, \eqref{E:tr.Melin.d}, \eqref{E:tr.Melin.e} and using \eqref{E:tr.A.reg}, \eqref{E:el.reg}, \eqref{E:ck.reg}, we obtain
		\begin{multline}\label{E:dsP.0}
			\int_0^\infty t^{s-1}\tr\left(e^{-t\rho_{r\lambda,r\mu,\nu}(A)}\right)dt
			\\
			=\sum_{l=0}^Lr^l\int_0^\infty t^{s-1}e_l(t)dt
			+\sum_{k=0}^Kr^{2\kappa s+2k-2}\int_0^\infty y^{s-1}c_k(y)dy
			\\
			+r^{2\kappa s}\sum_{j=0}^J\frac{a_j(r)-\sum_{k=0}^Kr^{2k-2}b_{j,k}}{s+(j-3)/4\kappa}
			\\
			-\sum_{l=L+1}^{L'}r^l\sum_{k=0}^K\frac{d_{k,l}}{s+(2k-2-l)/2\kappa}
			+O\left(r^N\right)
		\end{multline}
		uniformly for $s\in S$ and $0<r\leq1$.

		Recall that $\frac1{\Gamma(s)}$ is an entire function that vanishes for $s\in-\N_0$.
		Hence, as $S$ does not intersect $P$, and since $L+1\geq N$ , we have
		\begin{equation}\label{E:dsP.1}
			\frac1{\Gamma(s)}\sum_{l=L+1}^{L'}r^l\sum_{k=0}^K\frac{d_{k,l}}{s+(2k-2-l)/2\kappa}
			=O\left(r^N\right)
		\end{equation}
		uniformly for $s\in S$ and $0<r\leq1$.
		Similarly, using \eqref{E:a.b} and $2\kappa\sigma+2K\geq N$, we find
		\begin{equation}\label{E:dsP.2}
			\frac{r^{2\kappa s}}{\Gamma(s)}\sum_{j=0}^J\frac{a_j(r)-\sum_{k=0}^Kr^{2k-2}b_{j,k}}{s+(j-3)/4\kappa}
			=O\left(r^N\right)
		\end{equation}
		uniformly for $s\in S$ and $0<r\leq1$.
		By \eqref{E:el.reg}, the poles of $\frac1{\Gamma(s)}\int_0^\infty t^{s-1}e_l(t)dt$ are all contained in $P$, hence disjoint from $S$.
		For each $l\geq N$ we thus have
		\begin{equation}\label{E:dsP.3}
			\frac{r^l}{\Gamma(s)}\int_0^\infty t^{s-1}e_l(t)dt
			=O\left(r^N\right)
		\end{equation}
		uniformly for $s\in S$ and $0<r\leq1$.
		By \eqref{E:ck.reg} the poles of $\frac1{\Gamma(s)}\int_0^\infty y^{s-1}c_k(y)dy$ are all contained in $P$, hence disjoint from $S$.
		For each $k\geq(N+2-2\kappa\sigma)/2$ we thus have
		\begin{equation}\label{E:dsP.4}
			\frac{r^{2\kappa s+2k-2}}{\Gamma(s)}\int_0^\infty y^{s-1}c_k(y)dy
			=O\left(r^N\right)
		\end{equation}
		uniformly for $s\in S$ and $0<r\leq1$.
		Combining \eqref{E:dsP.0} with \eqref{E:dsP.1}, \eqref{E:dsP.2}, \eqref{E:dsP.3}, \eqref{E:dsP.4}, and \eqref{E:zeta}, we obtain \eqref{E:zetas}.

		From \eqref{E:ck.reg}, using $\frac1{\Gamma(s)}=s+O(s^2)$ as $s\to0$, we obtain
		\[
			\left.\left(\frac1{\Gamma(s)}\int_0^\infty t^{s-1}c_k(y)dy\right)\right|_{s=0}
			=\begin{cases}
				b_{3,k}&\text{if $k=0$, and}\\
				b_{3,k}-d_{k,2k-2}&\text{if $k\geq1$.}
			\end{cases}
		\]
		As $b_{3,k}=0$, we obtain \eqref{E:zetap0} by differentiating \eqref{E:zetas} using Weierstra{\ss}' theorem on uniform convergence of holomorphic functions.

		From the asymptotic expansion in \eqref{E:zetap0} we read off \eqref{E:const}, using \eqref{E:e0.zeta} and the fact that $c_k$ vanishes for odd $k$.
	\end{proof}

	\begin{remark}\label{R:Melin.pos}
		For $\nu>0$, Theorem~\ref{T:Melin} remains true with $e_l=0$ and $d_{k,l}=0$, cf.~Remark~\ref{R:poly.pos} at the end of Section~\ref{S:poly.heat}.
		In this case $c_k$ also vanishes for odd $k$.
		In particular, we have an asymptotic expansion of the form
		\[
			\zeta_{\rho_{r\lambda,r\mu,\nu}(A)}'(0)
			=\sum_{k=0}^\infty r^{2k-2}\frac{\partial}{\partial s}\bigg|_{s=0}\frac1{\Gamma(s)}\int_0^\infty y^{s-1}c_k(y)dy
		\]
		as $r\to0$, with constant term 
		\[
			\LIM_{r\to0}\zeta_{\rho_{r\lambda,r\mu,\nu}(A)}'(0)=0.
		\]
	\end{remark}

	The following corollary is used in \cite[Lemma~9]{H23b}.

	\begin{corollary}
		In the situation of Theorem~\ref{T:Melin} suppose $(\lambda,\mu)\neq(0,0)$.
		If $N$ and $\sigma$ are real numbers and $\varepsilon>0$, then
		\begin{align*}
			\zeta_{\rho_{\lambda,\mu,-\nu}(A)}(s)
			&=\sum_{k=0}^{\left\lceil\frac{N-2\kappa\sigma}2\right\rceil}\nu^{-2\kappa s-3(k-1)/2}C_k^-(s)
			\\&\qquad\qquad
			+\sum_{l=0}^{\lceil N\rceil-1}\nu^{-\kappa s/2-3l/4}E_l(s)
			+O\bigl(\nu^{-\kappa\Re s/2-3N/4}\bigr)
			\\
			\zeta_{\rho_{\lambda,\mu,\nu}(A)}(s)
			&=\sum_{k=0}^{\left\lceil\frac{N-2\kappa\sigma}2\right\rceil}\nu^{-2\kappa s-3(k-1)/2}C_k^+(s)
			+O\bigl(\nu^{-\kappa\Re s/2-3N/4}\bigr)
		\end{align*}
		uniformly on compact subsets of $\{s\in\C\setminus P:\Re s\geq\sigma\}$ and for $\varepsilon\leq\nu<\infty$.
		Here
		\begin{align*}
			C_k^\pm(s)=\frac1{\Gamma(s)}\int_0^\infty y^{s-1}c_k^\pm(y)dy
			\quad\text{and}\quad
			E_l(s)=\frac1{\Gamma(s)}\int_0^\infty t^{s-1}e_l(t)dt
		\end{align*}
		are meromorphic functions on the entire complex plane whose poles all are contained in $P$.
		Moreover, $c_k^-(y)=c_k(y)$ and $e_l(t)$ are the functions from Theorem~\ref{T:poly} associated with $\lambda,\mu$ and $\nu=-1$; whereas $c_k^+(y)$ are the corresponding functions associated with $\lambda,\mu$ and $\nu=+1$, cf.~Remark~\ref{R:poly.pos}.
		In particular, $C_k^\pm(s)$ vanishes for odd $k$ and $E_l(s)$ vanishes for odd $l$.
	\end{corollary}

	\begin{proof}
		From \eqref{E:tr.eAt} and \eqref{E:nu} we obtain
		\[
			\tr\bigl(e^{-(\tau^{-2\kappa}t)\rho_{\tau^3\lambda,\tau^3\mu,\tau^4\nu}(A)}\bigr)
			=\tr\bigl(e^{-t\rho_{\lambda,\mu,\nu}(A)}\bigr)
		\]
		for $\tau>0$.
		Via the Melin transform in~\eqref{E:zeta} this yields
		\[
			\zeta_{\rho_{\tau^3\lambda,\tau^3\mu,\tau^4\nu}(A)}(s)
			=\tau^{-2\kappa s}\zeta_{\rho_{\lambda,\mu,\nu}(A)}(s).
		\]
		Hence, for $\nu>0$,
		\begin{align*}
			\zeta_{\rho_{\lambda,\mu,-\nu}(A)}(s)
			&=\nu^{-\kappa s/2}\zeta_{\rho_{\nu^{-3/4}\lambda,\nu^{-3/4}\mu,-1}(A)}(s)
			\\
			\zeta_{\rho_{\lambda,\mu,\nu}(A)}(s)
			&=\nu^{-\kappa s/2}\zeta_{\rho_{\nu^{-3/4}\lambda,\nu^{-3/4}\mu,1}(A)}(s).
		\end{align*}
		Combining the first equality with \eqref{E:zetas} and $r=\nu^{-3/4}$, we obtain the first expansion in the corollary, for $\varepsilon=1$.
		By continuity and compactness it remains true for all $\varepsilon>0$.
		Similarly, combining the second equality with Remark~\ref{R:Melin.pos}, we obtain the second expansion in the corollary.
	\end{proof}

\section{Proof of Theorem~\ref{T:tor.gen}}\label{S:proof}

	We begin by discussing the dependence of the torsion on the Hermitian inner product.
	To this end, suppose $h_{u,q}$ is a Hermitian inner product on $H^q(\goe)$ depending smoothly on a parameter $u\in\R$.

	The following can be traced back to the original work of Ray--Singer \cite[Theorem~2.1]{RS71}, see also \cite[Theorem~5.6]{BZ92} or \cite[Section~3.2]{RS12}.
	Proceeding exactly as in the proof of \cite[Lemma~2.13]{H22} one readily shows

	\begin{lemma}
		In each irreducible unitary representation $\rho$ of $G$ we have
		\[
			\frac\partial{\partial u}\sum_{q=0}^5(-1)^qN_q\tr\bigl(e^{-t\rho(\Delta_{h_u,q})}\bigr)
			=\kappa t\frac\partial{\partial t}\sum_{q=0}^5(-1)^q\tr\left(\dot G_{u,q}e^{-t\rho(\Delta_{h_u,q})}\right)
		\]
		where $\Delta_{h_u,q}$ is the Rumin--Seshadri operator defined in \eqref{E:Rumin.Seshadri}, and
		\[
			\dot G_{u,q}
			:=h_{u,q}^{-1}\tfrac\partial{\partial u}h_{u,q}
			\in\eend(H^q(\goe))
		\]
		denotes the logarithmic derivative of $h_{u,q}$.
	\end{lemma}

	For the zeta function this implies
	\[
		\frac\partial{\partial u}\sum_{q=0}^5(-1)^qN_q\zeta_{\rho(\Delta_{h_u,q})}(s)
		=-\sum_{q=0}^5(-1)^q\frac{\kappa s}{\Gamma(s)}\int_0^\infty t^{s-1}\tr\left(\dot G_{u,q}e^{-t\rho(\Delta_{h_u,q})}\right)dt.
	\]
	If $\rho=\rho_{\lambda,\mu,\nu}$ is a generic representation, then the constant term in the asymptotic expansion \eqref{E:asymp.gen} in Theorem~\ref{T:asymp} vanishes, and thus
	\[
		\frac\partial{\partial u}\sum_{q=0}^5(-1)^qN_q\zeta_{\rho_{\lambda,\mu,\nu}(\Delta_{h_u,q})}'(0)=0.
	\]
	Combining this with \eqref{E:tor.Delta} and \eqref{E:def.det.Delta}, we obtain

	\begin{lemma}\label{L:tau.indep.h}
		For a generic representation $\rho_{\lambda,\mu,\nu}$ of $G$, the torsion $\tau_h(\rho_{\lambda,\mu,\nu}(D))$ is independent of the graded Hermitian inner product $h$ on $H^*(\goe)$.
	\end{lemma}

	Combining Lemma~\ref{L:tau.indep.h} with \eqref{E:Aut.tau}, we conclude that $\tau_h(\rho_{\lambda,\mu,\nu}(D))$ is constant on the orbits of $\Aut_\grr(\goe)$ acting on the space of generic representations.
	As pointed out at the very end of Section~\ref{SS:Aut.dual}, this implies that $\tau_h(\rho_{\lambda,\mu,\nu}(D))$ is constant in $(\lambda,\mu,\nu)$ too, for it clearly depends smoothly on these parameters in view of \eqref{E:tr.k.alpha}.

	On the other hand, for $\nu<0$, Theorem~\ref{T:Melin} provides an asymptotic expansion of the form 
	\begin{equation}\label{E:asymp.tau2}
		\zeta_{\rho_{r\lambda,r\mu,\nu}(\Delta_{h,q})}'(0)
		\sim\sum_{k=0}^\infty z_kr^{2k-2}+\sum_{k=1}^\infty\tilde z_kr^{4k-2}\log r
	\end{equation}
	as $r\to0$, with constant term
	\begin{equation}\label{E:LIM.zeta}
		\LIM_{r\to0}\zeta_{\rho_{r\lambda,r\mu,\nu}(\Delta_{h,q})}'(0)
		=z_1
		=\zeta_{\rho_{\sqrt{|\nu|}}(\Delta_{h,q})}'(0)+\zeta_{\rho_{-\sqrt{|\nu|}}(\Delta_{h,q})}'(0).
	\end{equation}
	Recall here that $d_{k,l}$ vanishes for odd $k$ and $e_l$ vanishes for odd $l$ according to Theorem~\ref{T:poly}.
	For the analytic torsion this implies, cf.~\eqref{E:tor.Delta},
	\begin{equation}\label{E:LIM.tau2}
		\LIM_{r\to0}\log\tau_h\bigl(\rho_{r\lambda,r\mu,\nu}(D)\bigr)
		=\log\tau_h\bigl(\rho_{\sqrt{|\nu|}}(D)\bigr)+\log\tau_h\bigl(\rho_{-\sqrt{|\nu|}}(D)\bigr).
	\end{equation}
	As $\tau_h\bigl(\rho_{r\lambda,r\mu,\nu}(D)\bigr)$ is constant in $r$, we conclude
	\begin{equation}\label{E:tau.tau}
		\tau_h\bigl(\rho_{\lambda,\mu,\nu}(D)\bigr)
		=\tau_h\bigl(\rho_{\sqrt{|\nu|}}(D)\bigr)\cdot\tau_h\bigl(\rho_{-\sqrt{|\nu|}}(D)\bigr).
	\end{equation}
	If $h=h_g$ is induced from a graded Euclidean inner product $g$ on $\goe$ such that $b_g$ is proportional to $g$, then both factors on the right hand side of \eqref{E:tau.tau} are one according to \eqref{E:tor.Schroedinger} in Theorem~\ref{T:dets.Schroedinger}.
	Hence, 
	\begin{equation}\label{E:tttau}
		\tau_h\bigl(\rho_{\lambda,\mu,\nu}(D)\bigr)=1,
	\end{equation}
	initially only for $\nu<0$ and these special $h=h_g$.
	However, since the torsion is constant in $h$ and $(\lambda,\mu,\nu)$, the latter equality remains true for arbitrary $\nu$ and $h$.
	Hence, Theorem~\ref{T:tor.gen} holds for all generic representations.
	\footnote{Proceeding exactly as in the derivation of \eqref{E:tau.tau}, but using Remark~\ref{R:Melin.pos} istead of Theorem~\ref{T:Melin}, we obtain an alternative derivation of \eqref{E:tttau}, for $\nu>0$ and all $h$. As the left hand side in \eqref{E:tttau} is constant, this equality remains true for all $\nu$. This argument does not rely on Theorem~\ref{T:dets.Schroedinger}.}

	To deal with the Schr\"odinger representations, we now consider an arbitrary graded Euclidean inner product $g$ on $\goe$.
	Note that there exists a graded automorphism $\phi\in\Aut_\grr(\goe)$ that preserves $g$ and acts by $-1$ on $\goe_{-2}$.
	Indeed, in by \eqref{E:Aut.basis}, \eqref{E:Aut}, and \eqref{E:bg} this amounts to $\phi_{-1}\in\GL(\goe_{-1})$ being a reflection that preserves $g$ and $b_g$.
	Such reflections exist in view of the principal axis theorem.
	Clearly, $\phi\cdot h_g=h_g$ and $\rho_\hbar\circ\phi=\rho_{-\hbar}$, see \eqref{E:Aut.hg} and \eqref{E:Aut.rho.h}.
	Applying \eqref{E:Aut.tau}, we obtain 
	\[
		\tau_{h_g}(\rho_{-\hbar}(D))=\tau_{h_g}(\rho_\hbar(D)).
	\]
	From \eqref{E:tau.tau} we get $\tau_{h_g}(\rho_\hbar(D))\cdot\tau_{h_g}(\rho_{-\hbar}(D))=1$.
	Combining the latter two equalities and using the positivity of the torsion, we conlude $\tau_{h_g}(\rho_\hbar(D))=1$.
	This completes the proof of Theorem~\ref{T:tor.gen}.
	
	\begin{remark}\label{R:Voros}
		Let us point out more explicitly how the asymptotic in \eqref{E:asymp.tau2} differs from the one considered by Voros \cite{V04}.
		For simplicity we assume $\lambda^2+\mu^2=1$.
		From \eqref{E:rep.gen.X1} and \eqref{E:rep.gen.X2} we obtain
		\[
			\rho_{r\lambda,r\mu,\nu}(\Delta_{h,0})
			=r^{2/3}\left(-\partial_\theta^2+\tfrac14\left(\theta^2+\nu r^{-4/3}\right)^2\right).
		\]
		This operator is unitarily equivalent to
		\begin{equation}\label{E:nonVoros}
			-\partial_\theta^2+g\theta^4+\frac\nu2\theta^2+\frac{\nu^2}{16g},\qquad\text{where } g=\frac{r^2}4,
		\end{equation}
		via the unitary scaling $f(\theta)\leftrightarrow\sqrt\gamma f(\gamma\theta)$ with $\gamma=r^{1/3}$.
		Of course $r\to0$ corresponds to $g\to0$.
		However, this differs significantly from the perturbation considered by Voros \cite[Eq.~(5.7)]{V04} because
		the constant part $\frac{\nu^2}{16g}$ of the potential in \eqref{E:nonVoros} depends on $g$ in a singular manner, while the constant in Voros' potential remains fixed.
		Correspondingly, on the right hand side of \cite[Eq.~(5.7)]{V04} the determinant of only one (shifted) oscillator appears, while for $\nu<0$ in \eqref{E:LIM.zeta} two harmonic oscillators contribute to the constant term, and for $\nu>0$ in Remark~\ref{R:Melin.pos} the constant term vanishes.
	\end{remark}

\appendix

\section{The abelian case}\label{S:abelian}

	In this section we consider the abelian Lie algebra $\aoe=\R^n$ with the trivial grading $\aoe=\aoe_{-1}$.
	The Rumin complex for the corresponding simply connected Lie group $A=\R^n$ coincides with the de~Rham complex.
	We will compute the determinants of the differentials and the torsion in irreducible unitary representations, see \eqref{E:det.abel} and \eqref{E:tor.abel} below.

	Let $g$ be a Euclidean inner product on $\aoe$ and let $h_g$ denote the induced Hermitian inner product on $H^q(\aoe)=\Lambda^q\aoe^*\otimes\C$.
	Choose an orthonormal basis $X_1,\dotsc,X_n$ of $\aoe$ and let $\chi^1,\dotsc,\chi^n$ denote the dual basis of $\aoe^*$.

	As the Rumin complex coincides with the de~Rham complex, we have
	\[
		D_q=\sum_{j=1}^nX_j\otimes e_{\chi^j}\in\mathcal U^{-1}(\aoe)\otimes L\bigl(\Lambda^q\aoe^*\otimes\C,\Lambda^{q+1}\aoe^*\otimes\C\bigr)
	\]
	where $e_\chi$ denotes the exterior product with $\chi$, that is, $e_\chi\beta=\chi\wedge\beta$.

	For $\alpha\in\aoe^*$ we let $\rho_\alpha$ denote the 1-dimensional unitary representations of $A$ such that $\rho_\alpha(X)=\mathbf i\alpha(X)$ for $X\in\aoe$.
	Up to unitary equivalence, these are all the irreducible unitary representations of $A$.
	In such a representation we have
	\[
		\rho_\alpha(D_q)=\mathbf i\sum_{j=1}^n\alpha(X_j)e_{\chi^j}
		\qquad\text{and}\qquad
		\rho_\alpha(D_q)^{*_{h_g}}=-\mathbf i\sum_{k=1}^n\alpha(X_k)i_{X_k}
	\]
	where $i_X$ denotes the contraction with $X$.
	As $e_{\chi^j}i_{X_k}+i_{X_k}e_{\chi^j}=\delta^j_k$, the Laplacian acts by the same scalar in each degree,
	\[
		\rho_\alpha(D_{q-1})\rho_\alpha(D_{q-1})^{*_{h_g}}+\rho_\alpha(D_q)^{*_{h_g}}\rho_\alpha(D_q)
		=\|\alpha\|_g^2.
	\]

	Assume $\alpha\neq0$, i.e., $\rho_\alpha$ is a nontrivial representation.
	Then the Rumin complex is exact and the preceding equation yields
	\[
		\dets|\rho_\alpha(D_{q-1})|_{h_g}\cdot\dets|\rho_\alpha(D_q)|_{h_g}
		=\|\alpha\|_g^{\binom nq}.
	\]
	Using the recurrence relation for binomial coefficients, we obtain
	\begin{equation}\label{E:det.abel}
		\dets|\rho_\alpha(D_q)|_{h_g}=\|\alpha\|_g^{\binom{n-1}q}.
	\end{equation}
	For the torsion we find
	\begin{equation}\label{E:tor.abel}
		\tau_{h_g}(\rho_\alpha(D))
		=\begin{cases}
			\|\alpha\|_g&\text{for $n=1$, and}
			\\
			1&\text{if $n\geq2$.}
		\end{cases}
	\end{equation}

\section{The 3-dimensional Heisenberg group}\label{S:Heisenberg}

	In this section we consider the 3-dimensional Heisenberg algebra $\hoe=\hoe_{-2}\oplus\hoe_{-1}$ and the corresponding simply connected 3-dimensional Heisenberg group $H$.
	We will compute the determinants of the Rumin differentials and the torsion of the Rumin complex in all nontrivial irreducible unitary representations.
	The results are summarized in Propositions~\ref{P:det.Heisenberg.scalar} and \ref{P:det.Heisenberg.Schroedinger} below.
	For scalar representations this is entirely elementary.
	In the Schr\"odinger representations we will be able to reuse the crucial Lemma~\ref{L:specA2} from Section~\ref{SS:spec1.Schroedinger}.
	The Rumin--Seshadri analytic torsion of contact 3-manifolds has been computed in \cite{RS12} and \cite{AQ22}.

	Let $X_1,X_2$ be a basis of $\hoe_{-1}$ and put $X_3=[X_1,X_2]$.
	Then $X_1,X_2,X_3$ is a graded basis of $\hoe$.
	Let $\chi^1,\chi^2,\chi^3$ denote the graded dual basis of $\hoe^*=\hoe_1^*\oplus\hoe_2^*$.
	One readily checks that the following forms induce bases of $H^q(\hoe)=\frac{\ker\partial_q}{\img\partial_{q-1}}$:
	\begin{equation}\label{E:bases.Heisenberg}
		H^0(\hoe):1;\qquad
		H^1(\hoe):\chi^1,\chi^2;\qquad
		H^2(\hoe):\chi^{13},\chi^{23};\qquad
		H^3(\hoe):\chi^{123},
	\end{equation}
	where we use the notation $\chi^{j_1\cdots j_k}=\chi^{j_1}\wedge\cdots\wedge\chi^{j_k}$.
	With respect to these bases, the Rumin differentials $D_q\in\mathcal U(\hoe)\otimes L(H^q(\hoe),H^{q+1}(\hoe))$ \cite{R90,R94,R00} are represented by the matrices
	\begin{equation}\label{E:D.Heisenberg}
		D_0=\begin{pmatrix}X_1\\X_2\end{pmatrix},\quad
		D_1=\begin{pmatrix}-X_{12}-X_3&X_{11}\\-X_{22}&X_{21}-X_3\end{pmatrix},\quad
		D_2=\begin{pmatrix}-X_2&X_1\end{pmatrix}.
	\end{equation}
	By naturality, $\phi\cdot D_q=D_q$ for all for all graded automorphisms $\phi\in\Aut_\grr(\hoe)$, where the dot denotes the natural left action on $\mathcal U(\hoe)\otimes L(H^q(\hoe),H^{q+1}(\hoe))$.

	Let $g$ be a graded Euclidean inner product on $\hoe$ and let $a_g$ denote the positive real number such that 
	\[
		g\bigl([X,Y],[X,Y]\bigr)=a_g\cdot\Bigl(g(X,X)g(Y,Y)-g(X,Y)^2\Bigr)
	\]
	holds for all $X,Y\in\hoe_{-1}$.
	Clearly, $a_{\phi\cdot g}=a_g$ for each $\phi\in\Aut_\grr(\hoe)$.
	Restriction provides an isomorphism $\Aut_\grr(\hoe)=\GL(\hoe_{-1})\cong\GL_2(\R)$.
	Hence, we may assume w.l.o.g.\ that $X_1,X_2$ is an orthonormal basis of $\hoe_{-1}$.
	Then the induced Hermitian inner products $h_{g,q}$ on $H^q(\hoe)$ are represented by the matrices
	\begin{equation}\label{E:hq.Heisenberg}
		h_{g,0}=\begin{pmatrix}1\end{pmatrix},\quad
		h_{g,1}=\begin{pmatrix}1\\&1\end{pmatrix},\quad
		h_{g,2}=\frac1{a_g}\begin{pmatrix}1\\&1\end{pmatrix},\quad
		h_{g,3}=\frac1{a_g}\begin{pmatrix}1\end{pmatrix},
	\end{equation}
	with respect to the bases provided by the forms in \eqref{E:bases.Heisenberg}.

	The simply connected 3-dimensional Heisenberg group $H$ has two types of irreducible unitary representations: the scalar representations which factor through $H/[H,H]$, and the Schr\"odinger representations.
	We will discuss the determinants in each type of representation separately in the next two subsections.

\subsection{Scalar representations}

	Suppose $\alpha\in\hoe_1^*$.
	For the corresponding scalar representation $\rho_\alpha$ of $H$ on $\C$ we have
	\[
		\rho_\alpha(X_1)=\mathbf iv,\qquad
		\rho_\alpha(X_2)=\mathbf iw,\qquad
		\rho_\alpha(X_3)=0,
	\]
	where $v=\alpha(X_1)$, and $w=\alpha(X_2)$.
	Combining this with \eqref{E:D.Heisenberg}, we see that in this representation the Rumin differentials are
	\begin{equation}\label{E:D0.Heisenberg.scalar}
		\rho_\alpha(D_0)=\mathbf i\begin{pmatrix}v\\w\end{pmatrix},\qquad
		\rho_\alpha(D_2)=\mathbf i\begin{pmatrix}-w&v\end{pmatrix},
	\end{equation}
	and
	\begin{equation}\label{E:D1.Heisenberg.scalar}
		\rho_\alpha(D_1)=\begin{pmatrix}vw&-v^2\\w^2&-vw\end{pmatrix}=\begin{pmatrix}v\\w\end{pmatrix}\begin{pmatrix}w&-v\end{pmatrix}.
	\end{equation}
	As $\rho_\alpha(D_q)^{*_{h_g}}=h_{q+1}^{-1}\rho_\alpha(D_q)^*h_q$, we have
	\begin{equation}\label{E:det.Heisenberg}
		\dets|\rho_\alpha(D_q)|_{h_g}
		=\dets^{1/2}\Bigl(h_{g,q}^{-1}\rho_\alpha(D_q)^*h_{g,q+1}\rho_\alpha(D_q)\Bigr).
	\end{equation}
	Clearly, $\|\alpha\|_g=(v^2+w^2)^{1/2}$.
	
	Assume $\alpha\neq0$. Hence $\rho_\alpha$ is a nontrivial representation, and the Rumin complex is exact.
	Combining \eqref{E:hq.Heisenberg}, \eqref{E:D0.Heisenberg.scalar}, and \eqref{E:det.Heisenberg}, we obtain
	\[
		\dets|\rho_\alpha(D_0)|_{h_g}
		={\det}^{1/2}\bigl(\rho_\alpha(D_0)^*\rho_\alpha(D_0)\bigr)
		=(v^2+w^2)^{1/2}
		=\|\alpha\|_g.
	\]
	Using the factorization for $D_1$ provided in \eqref{E:D1.Heisenberg.scalar} and \eqref{E:specAB}, we find
	\begin{align*}
		\dets|\rho_\alpha(D_1)|_{h_g}
		&=\dets^{1/2}\bigl(\rho_\alpha(D_1)^*h_{g,2}\rho_\alpha(D_1)\bigr)
		\\
		&=\frac1{\sqrt{a_g}}
		{\det}^{1/2}\Bigl(\begin{pmatrix}w&-v\end{pmatrix}\begin{pmatrix}w&-v\end{pmatrix}^*\Bigr)
		{\det}^{1/2}\left(\begin{pmatrix}v\\w\end{pmatrix}^*\begin{pmatrix}v\\w\end{pmatrix}\right)
		\\
		&=\frac{v^2+w^2}{\sqrt{a_g}}
		=\frac{\|\alpha\|_g^2}{\sqrt{a_g}}.
	\end{align*}
	By Poincar\'e duality, $\dets|\rho_\alpha(D_2)|_{h_g}=\dets|\rho_\alpha(D_0)|_{h_g}$.

	Summarizing, we obtain the following analogue of Theorem~\ref{T:dets.scalar}.

	\begin{proposition}\label{P:det.Heisenberg.scalar}
		Let $D_q\in\mathcal U(\hoe)\otimes L(H^q(\hoe),H^{q+1}(\hoe))$ denote the Rumin differentials associated with the 3-dimensional Heisenberg algebra $\hoe=\hoe_{-2}\oplus\hoe_{-1}$.
		Moreover, let $h_g$ denote the Hermitian inner product on $H^*(\hoe)$ induced from a graded Euclidean inner product $g$ on $\hoe$.
		Then, in the nontrivial scalar representation $\rho_\alpha$ on $\C$ corresponding to $\alpha\in\hoe_1^*$, we have:
		\begin{align*}
			\dets\bigl(|\rho_\alpha(D_2)|_{h_g}\bigr)
			=\dets\bigl(|\rho_\alpha(D_0)|_{h_g}\bigr)
			&=\|\alpha\|_g,
			\\
			\dets\bigl(|\rho_\alpha(D_1)|_{h_g}\bigr)
			&=\frac{\|\alpha\|_g^2}{\sqrt{a_g}}.
		\end{align*}
		In particular, the torsion is 
		\[
			\tau_{h_g}\bigl(\rho_\alpha(D)\bigr)=\sqrt{a_g}.
		\]
	\end{proposition}

\subsection{Schr\"odinger representations}\label{SS:contact.Schroedinger}

	Fix $0\neq\hbar$ and consider the Schr\"odinger representation $\rho_\hbar$ of $H$ on $L^2(\R,d\theta)$, cf.~\eqref{E:Schroedinger}.
	
	Note that the operator $D_0$ in \eqref{E:D.Heisenberg} coincides with the operator in \eqref{E:D0.Schroedinger}.
	Hence, using \eqref{E:hq.Heisenberg},
	\begin{equation}\label{E:D0Del}
		\rho_\hbar(D_0)^{*_{h_g}}\rho_\hbar(D_0)=\Delta,
	\end{equation}
	where $\Delta$ denotes the operator considered in Section~\ref{S:Schroedinger}, see \eqref{E:Delta}.
	From \eqref{E:zeta0.Schroedinger} we obtain 
	\[
		\zeta_{|\rho_\hbar(D_0)|_{h_g}}(s)
		=|\hbar|^{-s/2}\cdot(1-2^{-s/2}\bigr)\cdot\zeta_\Riem(s/2)
	\]
	and then
	\begin{equation}\label{E:det.D0.Heisenberg.Schroedinger}
		\zeta_{|\rho_\hbar(D_0)|_{h_g}}(0)=0,
		\qquad
		\zeta_{|\rho_\hbar(D_0)|_{h_g}}'(0)=-\tfrac14\log2.
	\end{equation}

	For the operator $D_1$ in \eqref{E:D.Heisenberg} we have $D_1=J^*A$ and therefore, using \eqref{E:hq.Heisenberg}, 
	\begin{equation}\label{E:D1A2}
		\rho_\hbar(D_1)^{*_{h_g}}\rho_\hbar(D_1)=\tfrac1{a_g}A^2,
	\end{equation}
	where $A$ and $J$ denote the operators considered in Section~\ref{S:Schroedinger}, see \eqref{E:A} and \eqref{E:J}.
	Using \eqref{E:D1A2} and Lemma~\ref{L:specA2}, we find 
	\[
		\spec_*\bigl(\rho_\hbar(D_1)^{*_{h_g}}\rho_\hbar(D_1)\bigr)
		=\spec_*\bigl(\tfrac1{a_g}A^2\bigr)
		=\spec\bigl(\tfrac1{a_g}\Delta^2\bigr).
	\]
	In view of \eqref{E:spec0.Schroedinger} this yields
	\[
		\zeta_{|\rho_\hbar(D_1)|_{h_g}}(s)
		=\left(\frac{|\hbar|}{\sqrt{a_g}}\right)^{-s}\cdot\bigl(1-2^{-s}\bigr)\cdot\zeta_\Riem(s)
	\]
	and then
	\begin{equation}\label{E:det.D1.Heisenberg.Schroedinger}
		\zeta_{|\rho_\hbar(D_1)|_{h_g}}(0)=0,
		\qquad
		\zeta_{|\rho_\hbar(D_1)|_{h_g}}'(0)=-\tfrac12\log2.
	\end{equation}
	By Poincar\'e duality, $\zeta_{|\rho_\hbar(D_2)|_{h_g}}(s)=\zeta_{|\rho_\hbar(D_0)|_{h_g}}(s)$.

	The regularized determinants and the analytic torsion for the Heisenberg group are defined as in \eqref{E:det.def} and \eqref{E:tau.D.def}.
	More explicitly, 
	\[
		\dets\bigl(|\rho_\hbar(D_q)|_{h_g}\bigr)
		=\exp\bigl(-\zeta'_{|\rho_\hbar(D_q)|_{h_g}}(0)\bigr),
	\]
	and
	\[
		\tau_{h_g}(\rho_\hbar(D))=\frac{\dets|\rho_\hbar(D_0)|_{h_g}\cdot\dets|\rho_\hbar(D_2)|_{h_g}}{\dets|\rho_\hbar(D_1)|_{h_g}}.
	\]
	
	Summarizing \eqref{E:det.D0.Heisenberg.Schroedinger} and \eqref{E:det.D1.Heisenberg.Schroedinger}, we obtain the following analogue of Theorem~\ref{T:dets.Schroedinger}.

	\begin{proposition}\label{P:det.Heisenberg.Schroedinger}
		Let $D_q\in\mathcal U(\hoe)\otimes L(H^q(\hoe),H^{q+1}(\hoe))$ denote the Rumin differentials associated with the 3-dimensional Heisenberg algebra $\hoe=\hoe_{-2}\oplus\hoe_{-1}$.
		Moreover, let $h_g$ denote the Hermitian inner product on $H^*(\hoe)$ induced from a graded Euclidean inner product $g$ on $\hoe$.
		Then, in the Schr\"odinger representation $\rho_\hbar$ on $L^2(\R,d\theta)$, $\hbar\neq0$, we have
		\begin{align*}
			\dets\bigl(|\rho_\hbar(D_2)|_{h_g}\bigr)
			=\dets\bigl(|\rho_\hbar(D_0)|_{h_g}\bigr)
			&=2^{1/4},
			\\
			\dets\bigl(|\rho_\hbar(D_1)_{h_g}\bigr)
			&=2^{1/2}.
		\end{align*}
		In particular, the torsion is trivial,
		\[
			\tau_{h_g}\bigl(\rho_\hbar(D)\bigr)=1.
		\]
	\end{proposition}

\section*{Acknowledgments}
	This research was funded in whole or in part by the Austrian Science Fund (FWF) Grant DOI 10.55776/P31663. 



\end{document}